\documentclass[oneside,reqno,11pt]{amsart}

\usepackage[top=1in, bottom=1.1in, left=1.2in, right=1.2in]{geometry}

\usepackage{amsmath}
\usepackage{amsfonts}
\usepackage{amssymb}
\usepackage{graphicx}
\usepackage{hyperref}
\usepackage{mathrsfs}
\usepackage{pb-diagram}
\usepackage{epstopdf}
\usepackage{amsmath,amsfonts,amsthm,enumerate,amscd,latexsym,curves}
\usepackage{bbm}
\usepackage[mathscr]{eucal}
\usepackage{epsfig,epsf,epic}
\usepackage{caption}
\usepackage{subcaption}
\usepackage{multirow}
\usepackage{color}
\usepackage[all]{xy}
\usepackage[cjk]{kotex}
\usepackage[utf8]{inputenc}
\usepackage{mathabx}
\usepackage[sc]{mathpazo}

\usepackage[normalem]{ulem}
\usepackage{fancybox}

\usepackage{tikz}
\usetikzlibrary{matrix}
\tikzset{ 
table/.style={
  matrix of nodes,
  row sep=-\pgflinewidth,
  column sep=-\pgflinewidth,
  nodes={rectangle,text width=7em,align=center},
  text depth=1.25ex,
  text height=2.5ex,
  nodes in empty cells
},
}

\newtheorem{thm}{Theorem}[section]
\newtheorem{lemma}[thm]{Lemma}

\newtheorem{prop}[thm]{Proposition}
\newtheorem{defn}[thm]{Definition}
\newtheorem{example}[thm]{Example}
\newtheorem{remark}[thm]{Remark}

\newtheorem{assumption}[thm]{Assumption}

\numberwithin{equation}{section}

\newcommand{\bL}{\mathbb{L}}

\newcommand{\Z}{\mathbb{Z}}

\newcommand{\R}{\mathbb{R}}
\newcommand{\C}{\mathbb{C}}

\newcommand{\AI}{A_\infty}

\newcommand{\WT}[1]{\widetilde{#1}}

\newcommand{\one}{\mathbf{1}}

\begin{document}

\author[Hansol Hong]{Hansol Hong}
\address{Department of Mathematics \\ Yonsei University \\ 50 Yonsei-Ro \\ Seodaemun-Gu \\ Seoul 03722 \\ Korea} 
\email{hansolhong@yonsei.ac.kr}

\title{Maurer-Cartan deformation of Lagrangians}

\begin{abstract}
The Maurer-Cartan algebra of a Lagrangian $L$ is the algebra that encodes the deformation of the Floer complex $CF(L,L;\Lambda)$ as an $A_\infty$-algebra. 
We identify the Maurer-Cartan algebra with the $0$-th cohomology of the Koszul dual dga of $CF(L,L;\Lambda)$. Making use of the identification, we prove that there exists a natural isomorphism between the Maurer-Cartan algebra of $L$ and a suitable subspace of the completion of the wrapped Floer cohomology of another Lagrangian $G$ when $G$ is \emph{dual} to $L$ in the sense to be defined. In view of mirror symmetry, this can be understood as specifying a local chart associated with $L$ in the mirror rigid analytic space. We examine the idea by explicit calculation of the isomorphism for several interesting examples.

%
\end{abstract}
\maketitle

\tableofcontents


\section{Introduction}

Deformation of a Lagrangian submanifold has been a central problem in symplectic geometry. The geometric deformation of a Lagrangian $L$ is locally parametrized by $H^1 (L)$, and McLean's classical result \cite{Mc} showed that the deformation of a special Lagrangian is unobstructed in the sense that deformation parametrized by  $H^1 (L)$ automatically satisfies a special Lagrangian condition. 
Such a deformation problem draws more attention after Strominger-Yau-Zaslow \cite{SYZ96} asserts that the mirror complex manifold of a symplectic manifold should be obtained as a complexified deformation space of a special Lagrangian torus.

On the other hand, in view of homological mirror symmetry, a Lagrangian (together with some additional structure) may be identified as an object of the Fukaya category. In this spirit, one can instead study deformation of the endomorphism algebra of a Lagrangian regarded as an object of the $A_\infty$-category, which is, in fact, the core of Lagrangian Floer theory of the Lagrangian. 
Such a deformation may admit further nontrivial obstructions from holomorphic disks bounded by the Lagrangian. 
More specifically, for given a compact Lagrangian $L$, one can study deformation of the $A_\infty$-algebra $CF(L,L)$ following \cite{FOOO}. The deformation is parametrized by $b \in CF^1(L,L) \cong H^1(L)$ (using the canonical model) that solves the following nonlinear equation
\begin{equation}\label{eqn:mcintro}
m_0(1) + m_1 (b) + m_2 (b,b) + m_2(b,b,b)+ \cdots =0
\end{equation}
called the \emph{Maurer-Cartan} equation, and $b$ satisfying the equation kills the obstruction $m_0$ for $L$ being an well-defined object in the Fukaya category. The set of solutions, denoted by $\mathcal{MC} (L)$, is called the Maurer-Cartan space of $L$.

Through a series of papers \cite{CHL}, \cite{CHL-toric} and \cite{CHL_gl}, we propose a mirror construction by gluing via quasi-isomorphisms local mirrors obtained as Maurer-Cartan spaces of Lagrangians (or strictly speaking, weak Maurer-Cartan spaces to produce a Landau-Ginzburg mirror). It often happens that a single Lagrangian already has big enough deformation to produce a full mirror, especially when we deform the Floer theory of an immersed Lagrangian $L$ by generators from self-intersections. 
While the Maurer-Cartan deformation of immersed Lagrangians can be nicely formulated in the realm of  immersed Floer theory \cite{AJ}, their geometric deformation seems a very delicate problem that has not been fully understood.

Another benefit of using $\mathcal{MC} (L)$ is that the construction is essentially algebraic once the necessary part of the $A_\infty$-structure on the Floer complex is known. This enables us to extend the deformation into noncommutative directions without much further effort, yet find new interesting geometric phenomena (see \cite{CHL2}). In this regard, it is more natural to extract the \emph{Maurer-Cartan algebra} $A_L$ of $L$ from \eqref{eqn:mcintro} (see Definition \ref{def:mcalgell}) rather than the space, which supposedly serves as the ring of functions on $\mathcal{MC}(L)$.

The first half of the article aims to reformulate the mirror construction via $\mathcal{MC} (L)$ in terms of well-known homological algebra tools, bar/cobar (denoted by $B$/$\Omega$) construction. The upshot is to realize the Maurer-Cartan algebra $A_L$ as the Koszul dual of the $A_\infty$-algebra $CF(L,L)$. 

\begin{thm}[Proposition \ref{prop:mcalgequalbvvee} and Proposition \ref{prop:omevbvd}]\label{thm:kdmcalg}
Let $L$ be a compact graded Lagrangian in a symplectic manifold $M$ such that $CF^0 (L,L)$ is spanned by the unit class. Then the Maurer-Cartan algebra $A_L$ can be identified as the $0$-th cohomology of the dga $\mathcal{A}_L:=\hom_{CF(L,L)} (\Lambda,\Lambda)$ where $\Lambda$ becomes a $CF(L,L)$-module via the augmentation $CF(L,L) \to \Lambda \cong CF^0 (L,L)$ given as the projection. This dga can be computed equivalently as either $\Omega \left(CF(L,L)^\vee \right)$ or $(B\, CF(L,L))^\vee$, where $(-)^\vee$ is the topological dual of a normed vector space. 
\end{thm}

Abstractly, the above can be understood as an adaption of the result of Efimov-Lunts-Orlov \cite{ELO} to the $T$-adic setting, which tells us that the derived deformation of an $A_\infty$-algebra is represented by the Koszul dual algebra.
For an $A_\infty$-algebra $A$ over a base ring $\Bbbk$, the dga $\hom_A (\Bbbk,\Bbbk)$ is referred to as the Koszul dual, following \cite{LPWZ}.
There are many different versions of Koszul duals in different generalities, and the one presented here is most adapted to the filtered $A_\infty$-setting. In particular, we use the topological dual as mentioned in the statement, and this is closely related with convergence in $T$-adic topology.

Observe that \eqref{eqn:mcintro} is an infinite sum in general. Therefore any construction rooted from the Maurer-Cartan deformation (of a compact Lagrangian) is necessarily accompanied with the convergence issue. In Floer theory, it is usually handled by introducing the field $\Lambda$ with a non-Archimedean valuation, recording the areas of contributing holomorphic disks as exponents of $T$. Moreover, $b$ in \eqref{eqn:mcintro} should have a positive $T$-adic valuation to ensure the convergence.
Accordingly, the Koszul dual will be carried out over $\Lambda$-coefficient in our setting.
This may seem as merely changing the coefficient ring, but we will see that the areas of holomorphic disks provide crucial information about the location and the size of $\mathcal{MC}(L)$ in the mirror space. 
In particular, it turns $\mathcal{MC}(L)$ into an analytic neighborhood in the mirror space for immersed $L$, not just a formal one that one would obtain over $\C$ . 

Under some additional assumptions, an $A_\infty$-algebra and its Koszul dual dga share important homological algebraic properties such as equivalences between certain module categories over these algebras. 
We expect that purely algebraic results known for Koszul duality may lead to some nontrivial geometric observation on Lagrangian Floer theory and mirror symmetry.
There is an interpretation of the mirror symmetry via Koszul duality between sheaves constructed from Lagrangian torus fibers. See \cite{Tu-FM} for more details.

In the second half of the article, we focus on a pair of Lagrangians $(G,L)$ in a Liouville manifold $M$ such that their Floer complex (over $\Lambda$)  has  a nontrivial component of rank 1 in degree 0, only. We will view these Lagrangians  as objects in the wrapped Fukaya category defined over $\Lambda$.
By introducing a semi-simple coefficient ring $\Bbbk_\Lambda:=\oplus_{i=1}^k \Lambda \langle \pi_i \rangle $, our setting includes the case when $G=\cup_{i=1}^r G_i$ and $L=\cup_{i=1}^r L_i$ consist of the same number of irreducible components such that $|G_i \cap L_j| = \delta_{ij}$. For given such a pair $(G,L)$, $L$ determines an augmentation $\epsilon_L: CF(G,G) \to \Bbbk_{\Lambda}$ by identifying $CF(G,L) \cong \Bbbk_{\Lambda}$ in $m_2 :CF(G,G) \otimes CF(G,L) \to CF(G,L)$. (Here, $CF(G,G)$ is the wrapped Floer cohomology of $G$ in $\Lambda$-coefficient,)

The key observation is that the augmentation $\epsilon_L$ can be regarded as specifying a point in the ``spec" of $CF(G,G)$. To make more sense of it, suppose that $CF(G,G)$ is concentrated at degree 0 so that it is simply an algebra (over $\Bbbk_\Lambda$) for degree reason.  If $G$ generates the Fukaya category of $M$, then $\check{M}:=``\mathrm{Spec}"  CF(G,G)$ (or more precisely a would-be space whose function ring is $CF(G,G)$) should give a mirror of $M$. Pascaleff \cite{Pas_CY} considered a similar situation in which $M$ is a $\mathrm{Log}$ Calabi-Yau surface, and $G$ a Lagrangian section. 

On the other hand, in spirit of SYZ mirror symmetry, the mirror space $\check{M}$ is supposed to be a deformation space of a Lagrangian, or the moduli of Lagrangians (in some loose sense). Hence each point of $\check{M}$ should correspond to some Lagrangian. In this point of view, it is natural to think that $L$ sits in $\check{M}$ as the point corresponding to the maximal ideal $\ker \epsilon_L$ in of the algebra $CF(G,G)$.

Having identified $L$ with a point in the mirror $\check{M}$ (constructed from the noncompact generator $G$), the next question to ask is how to describe its neighborhood in $\check{M}$. Recall that we already have an intrinsic deformation space of $L$, which is the Maurer-Cartan space $\mathcal{MC} (L)$. The following theorem answers how $\mathcal{MC}(L)$ sits in $\check{M}$ as a neighborhood of the point corresponding to $L$ or dually, it describes the relation between $CF(G,G)$ and the Maurer-Cartan algebra $A_L$. 

 \begin{thm}[Proposition \ref{prop:algebraickappa1}, Theorem \ref{thm:mainkd}]\label{thm:koszulintro}
Suppose $L$ satisfies the conditions in Theorem \ref{thm:kdmcalg}.
If two Lagrangians $L=\oplus_{i=1}^r L_i $ and $G=\oplus_{i=1}^r G_i$ in a Liouville manifold $M$ satisfy $ |G_i \cap L_j | = \delta_{ij}$, then there is a natural $A_\infty$-algebra homomorphism
$$ \kappa: CF(G,G) \to \mathcal{A}_L,$$
where $\mathcal{A}_L$ is the Koszul dual of $CF(L,L)$.

If $G$ generates $L$ in the wrapped Fukaya category of $M$, and  $HF^{>0} (G,G) =0$, then $\kappa$ induces the isomorphism between a certain proper  subalgbera of the formal completion of $HW^0 (G,G)$ and the Maurer-Cartan algebra $A_L$.

\end{thm}

See Proposition \ref{prop:unfiltereddercompl} and Proposition \ref{cor:formalversionk} for analogous statements in the unfiltered setting.
We remark that the two statements \cite[Theorem 2 and Corollary 52]{EL-duality} combined together proves the unfiltered version of the above when $H^1(L_i)=0$ for each $i$. 

The homomorphism $\kappa=\{\kappa_l\}_{l \geq 1}$ admits an explicit formula
$$\sum m_k (Z_1,\cdots, Z_l,P, \tilde{b},\cdots, \tilde{b}) = P \, \kappa_l (Z_1,\cdots, Z_l),$$
for $Z_1,\cdots,Z_l \in CF(G,G)$, where $P$ generates $CF(G,L)$ and $\tilde{b} = \sum x_i X_i$ is the formal linear combination of all generators $X_i$ of $CF(L,L)$ except the unit. $m_k$ on the left hand side can be expanded using \eqref{eqn:earconv} so that it becomes a sum of $m_k$'s having usual inputs from the Floer complexes multiplied by monomials in $x_i$. Making use of its compatibility with algebraic structures, $\kappa$ can be explicitly computed in many cases, at least on the level of cohomology. 

We provide detailed computations for several important classes of examples in Section \ref{sec:locsyzapp1}.
Observe that two typical examples of such a pair $(G,L)$ arises in 
\begin{itemize}
\item $M$, a cotangent bundle with a cotangent fiber $G$ and the zero section $L$, 
\item $M$ with a SYZ fibration where $G$ is a Lagrangian section and $L$ a torus fiber. 
\end{itemize}
Etg\"{u} and Lekili \cite{EL-koszul} have given a nice explanation on the former case related with Koszul duality, and analyzed in detail the case of plumbing of spheres in this context. Also Li \cite{Li,Li2,Li3} found such a duality pattern in many Weinstein manifolds beyond plumbings.

Our focus is more on the latter one, and in particular, the case when the compact Lagrangian $L$ has nontrivial $H^1 (L)$ so that its Maurer-Cartan deformation is nonempty. This is an essential reason why we need to deal with the convergence issue more carefully. Namely, the elements in $H^1 (L)$ can be inserted arbitrarily many times to $A_\infty$-operations without changing the degree of the output. Keeping track of the degree changes in algebraic constructions, $H^1 (L)$ is responsible for the $0$-th cohomology of $\mathcal{A}_L$, the Maurer-Cartan algebra $A_L$.

With help of the theorem, we shall make an atlas of the mirror rigid analytic varieties obtained from the generating Lagrangian $G$ of several local SYZ models, locating the Maurer-Cartan deformation space of important classes of Lagrangians (as well as torus fibers) in the mirror space. This procedure takes the opposite way to gluing construction in \cite{HL-gokova} or \cite{CHL_gl} in the sense that we describe local charts in terms of \emph{global coordinates} on the mirror given $G$.

The examples include the cotangent bundle of $T^n$ (or a trivial SYZ fibration), a conic bundle with a nodal torus fiber, and a deformed conifold. Purpose of the last example is to examine the noncommutative (quiver) situation.
Interestingly, we will see that immersed Lagrangians (such as nodal torus fiber) tend to have  bigger Maurer-Cartan neighborhoods than embedded ones (such as torus fibers). Intuitively, this means: the more singular the Lagrangian is, the more Lagrangians it can produce by deformation. It is, of course, plausible since one can smooth out singularity to make the Lagrangian less singular. 

On the contrary, torus fibers have very small Maurer-Cartan neighborhood. Note that different torus fibers do not intersect. This implies that their Maurer-Cartan neighborhoods do not overlap, since the Maurer-Cartan deformation is based on Lagrangian intersection Floer theory. 
For this reason, one should take into account every fiber to obtain a full deformation space of Lagrangian torus fibers. This is along the same line of thoughts as in the family Floer program \cite{Ab-famFl1} or the family of valuations in \cite{FOOO10b}\footnote{There are some cases such as toric manifolds where the symplectic information is concentrated only at critical fibers, in which case it may be enough to look at their weak Maurer-Cartan deformation to construct a mirror \cite{CHL-toric}.}.

Here are some future directions. One potential application of Theorem \ref{thm:koszulintro} is the computation of the Maurer-Cartan algebra of $L$ from the wrapped Floer cohomology of its dual $G$. 
Usually, the Maurer-Cartan algebra involves nonlinear relations from higher $m_k$-operations that makes it hard to compute by direct count of holomorphic disks. 
On the other hand, in many examples including conic bundles, the wrapped Floer cohomology is concentrated at degree $0$, and is relatively easy to compute by explicit perturbations. Also there is a nice generation result of the wrapped Fukaya category provided by \cite{GPS} and \cite{CRGG}. 
 
It would be also interesting to extend the result into weak Maurer-Cartan deformation and its associated Landau-Ginzburg mirror. In \ref{subsec:popex1}, we shall perform some relevant construction for a certain weakly unobstructed immersed circle in the pair-of-pants, which may give us a hint for a more general formulation.

\begin{center}
{\bf Notations}
\end{center}

We will work over several different coefficient rings depending on the geometric situation. The most frequently used is the universal Novikov field over $\C$ defined as
\begin{equation}\label{eqn:univnovfield}
\Lambda := \left\{\sum_{i=0}^\infty a_i T^{\lambda_i} \mid a_i \in\C, \lim_{i \to \infty} \lambda_i  +\infty \right\}.
\end{equation}
We also use the following subrings of $\Lambda$ often:
$$
\Lambda_0 := \left\{\sum_{i=0}^\infty a_i T^{\lambda_i} \in \Lambda \mid \lambda_i\ge 0  \right\}, \qquad
\Lambda_+ := \left\{\sum_{i=0}^\infty a_i T^{\lambda_i}\in \Lambda \mid \lambda_i>0  \right\}.
$$
Finally, the following is the set of elements in $\Lambda_0$ which has multiplicative inverses:
$$\Lambda_\mathrm{U}:=\{a_0 + \lambda: a_0 \in \C^\times \textrm{ and } \lambda \in \Lambda_+ \},
$$
which can be thought of as unitary elements in $\Lambda$.
In addition, we will also use their semi-simple generalizations, whose definition is to be given later.

\begin{center}
{\bf Aknowledgement}
\end{center}
I would like to thank Junwu Tu for explanation on the important constructions in homological algebra relevant to the paper. I am grateful to Cheuk Yu Mak and Cheol-Hyun Cho for valuable discussions. Finally, I would like to thank the anonymous referee for his careful reading and constructive comments. 
  This work   is supported by the National Research Foundation of Korea (NRF) grant funded by the Korea government (MSIT) (No. 2020R1C1C1A01008261 and No.2020R1A5A1016126).

\section{Maurer-Cartan space of a Lagrangian}\label{sec:mcspaceofLag}

We briefly recall Lagrangian Floer theory following \cite{FOOO}, especially the Maurer-Cartan deformation of a Lagrangian, which is the main subject in the paper. As we are to relate this with an endomorphism of a noncompact Lagrangians in later applications, a short review on wrapped Floer theory will be also given. Along the way, we present the precise $T$-adic setting for both of Floer theories that we work with throughout the paper. 

\subsection{Fukaya $\AI$-algebra and wrapped Fukaya category}\label{subsec:fukaialg}
Let us start with a compact unobstructed Lagrangian $L$ in a symplectic manifold $M$. We allow $L$ to be immersed with clean intersections. The unobstructedness of $L$ means that the evaluation image of the moduli space of holomorphic discs bounding $L$ with one marked point vanishes, which results in the absence of $m_0$-term in the $\AI$-structure below. For instance, if there is no nonconstant disc bounding $L$, then $L$ is automatically unobstructed. We will use the canonical model for the Floer complex $CF(L,L;\Lambda)$ that is modeled on $H^\ast (L;\Lambda)$. When $L$ is immersed, we additionally include two copies of cocycles on the self-intersection loci into $CF(L,L;\Lambda)$. Each of them spans rank-1 $\Lambda$-vector space in $CF(L,L;\Lambda)$ whose elements will be called \emph{immersed generators}. 

$CF(L,L;\Lambda)$ has a $\C$-vector subspace $CF(L,L;\C)$ generated by $H^\ast (L;\C)$ and immersed generators (over $\C$) so that $CF(L,L;\Lambda)=CF(L,L;\C) \otimes \Lambda$. There exists a non-Archimedean valuation $\nu$ on $CF(L,L;\Lambda)$ which is completely determined by setting it to be zero on  $CF(L,L;\C)$. See \ref{subsec:underlyingmods}.

\begin{thm}[\cite{FOOO},\cite{AJ}]\label{thm:FOOOmain}
There exist a sequence of multilinear maps 
$$m_k : CF(L,L;\Lambda)^{\otimes k } \to CF(L,L;\Lambda)$$ 
for $k \geq 1$ which satisfies
\begin{equation}\label{eqn:airelforLfooo}
\sum_{i, k_1+k_2 = k+1} (-1)^\ast m_{k_1} (X_1, \cdots, X_i, m_{k_2} (X_{i+1}, \cdots, X_{i+k_2}), X_{i+k_2 +1} ,\cdots, X_k),
\end{equation}
where $\ast = (-1)^{|X_1|' + \cdots + |X_i|'}$, and $|X_i|'$ denotes the shifted degree $|X_i|' = |X_i|-1$.
\end{thm}
A vector space $V$ with a family of multilinear operations $\{m_k : V^{\otimes k} \to V\}_{k \geq 1}$ satisfying \ref{eqn:airelforLfooo} is called an $\AI$-algebra, and hence the theorem simply tells us that $CF(L,L;\Lambda)$ admits an $\AI$-structure.
If $L$ is obstructed, \eqref{eqn:airelforLfooo} should additionally involve a certain nonzero cohomology class usually called $m_0(1)$. Each $m_k$ can be further decomposed as
\begin{equation}\label{eqn:mkbeta}
m_k = \sum_{\beta \in \pi_2 (M,L)} m_{k,\beta} T^{\omega(\beta)}
\end{equation}
which can be an infinite sum, but Gromov compactness ensures that this operation converges over $\Lambda$.

Now we recall the wrapped Floer cohomology first defined in \cite{AS_w}.
Let $M$ be a Liouville manifold with an exact symplectic form $\omega = d \theta$. Let us take a Liouville vector field $Z$ such that $\iota_Z \omega = \theta$. We require $M$ to be symplectomorphic to $[1,\infty) \times \Sigma$ outside a compact region $M^{in}$ for some contact hypersurface $\Sigma$. More precisely, $[1,\infty) \times \Sigma$ is equipped with $\omega=d(r \alpha)$ where $\alpha:=\theta|_{\Sigma}$ and $r$ is the standard coordinate on $(1,\infty)$. A Lagrangian $L$ is said to be conical at infinity if $L$ intersects $\Sigma$ transversally and $\Theta|_L \equiv 0$ on this region.

For given an exact Lagrangian $L$ conical at infinity, we define its (self-)wrapped Floer complex as follows. We take a linear Hamiltonian $H : M \to \R$ in the sense that $H=r$ on the region symplectomorphic to $(1,\infty) \times \Sigma$. 
Then the wrapped Floer complex $CW (L,L)$ is defined to be
$$CW(L,L) = \oplus_{w=1}^{\infty} CF(L,L;wH )[q]$$
where $CF(L,L;wH)$ is formally generated over $\C$ by time-1 Hamiltonian chords of $wH$ from $L$ to itself, and $q$ is a formal variable with $\deg q = -1$.
The operations
\begin{equation}\label{eqn:wrmkself}
m_k: CW (L,L)^{\otimes k} \to CW (L,L)
\end{equation}
are defined by counting pseudo-holomorphic maps that carries some additional data about continuation between Hamiltonians with different slopes $w$. The $m_1$-cohomology of $CW (L,L)$ denoted by $HW(L,L)$ is called the wrapped Floer cohomology of $L$. We refer readers to \cite[Section 2,3]{AS_w} for detailed construction. See also \cite[Section 3]{SR} for the construction in the monotone setting. We remark that unlike \eqref{eqn:mkbeta}, the output of \eqref{eqn:wrmkself} is a finite sum of intersection points.

More generally, for a $(k+1)$-tuple of a Lagrangian $L_0, \cdots, L_{k}$, one can analogously define
\begin{equation}\label{eqn:mkonwrappedfuk}
m_k: CW(L_0,L_1) \otimes \cdots \otimes CW(L_{k-1}, L_k) \to CW(L_0, L_k).
\end{equation}
where $CW(L_i,L_j) = \oplus_{w=1}^{\infty} CF(L_i,L_j;wH )[q]$ and $CW(L_i,L_j; wH)$ is the $\mathbb{C}$-vector space generated by $wH$-Hamiltonian chords from $L_i$ to $L_j$.
The wrapped Fukaya category $\mathcal{W}Fuk (M)$ is then defined as the $\mathbb{C}$-linear category whose objects are exact Lagrangians conical at infinity and morphism space between $L$ and $L'$ is given by $\hom_{\mathcal{W}Fuk (M)} (L,L')=CW(L,L')$. $\mathcal{W}Fuk (M)$ is naturally an $\AI$-category with operations $m_k$ \eqref{eqn:mkonwrappedfuk}.

\subsubsection*{Wrapped Fukaya category over $\Lambda$}
For later use, we extend the coefficient field for the wrapped Fukaya category to the Novikov field $\Lambda$ so that the $\AI$-operations involve the powers of $T$, recording the areas of contributing disks. We first set
$$CW(L_i,L_j;\Lambda) := \oplus_{w=1}^{\infty} CF(L_i,L_j;wH;\Lambda )[q]$$
where everything is the same as before except that $CF(L_i,L_j;wH;\Lambda )$ is now $\Lambda$-vector space freely generated by $wH$-Hamiltonian chords from $L_i$ to $L_j$. $m_k$-operations are defined analogously to \eqref{eqn:mkbeta}, that is, we count the pseudo-holomorphic map $u:S \to M$ contributing to $m_k$ with the weight $T^{E_{top} (u)} \in \Lambda$ for
$$E_{top} (u):=\int_S u^\ast \omega - d (u^\ast H \wedge \gamma).$$ 
Here, $\gamma$ is a part of the data contained in the disk moduli which encodes the interpolation among linear Hamiltonians with different slopes associated with inputs and outputs. In fact, we do not actually need its precise definition here, since in the exact setting, one can recover the weight simply from the actions of inputs and outputs as \cite[(7.9)]{AS_w}. Namely, if $u$ takes its inputs $P_i \in CF(L_{i-1},L_i ; w_i H )$ for $ 1 \leq i \leq k$ and the out put $Q \in CF(L_0, L_k; w_0 H)$, then Stokes formula tells us that 
\begin{equation}\label{eqn:topenergydist}
E_{top} (u) = \mathcal{A}_{w_0 H} (Q) - \sum_{i=1}^{k} \mathcal{A}_{w_i H} ( P_i),
\end{equation} 
where 
$$\mathcal{A}_{wH} (P)= \int_0^1 - P^\ast \theta + wH(P(t)) dt + f_j (P(1)) - f_i (P(0))$$
for a Hamiltonical chord $P \in CF(L_i,L_j;wH)$ and the primitives $f_i, f_j$ satisfying $d f_i = \theta|_{L_i}$. 
Sometimes it is useful to identify $P$ as an intersection point $P \in \phi_{wH} ( L_i) \cap L_j$ (abusing notation), in which case the valuation can be written as
\begin{equation}\label{eqn:auxvalH}
\mathcal{A}_{wH} (P):=- \iota_{X_{wH}} \theta (P)+f_j (P) - f_i (P).
\end{equation}
Notice that $ \iota_{X_{wH}} \theta +  f_L$ is a primitive of $\phi_{wH} (L)$, see for e.g., \cite[Lemma 3.2]{Gao}.

Let us denote by $\mathcal{W}Fuk^{\Lambda} (M)$ the category with the same set of objects as $\mathcal{W}Fuk^{\Lambda} (M)$, but 
$$ \hom_{\mathcal{W}Fuk^{\Lambda} (M)} (L,L') :=CW(L,L';\Lambda).$$
The above discussion gives the natural extension of $A_\infty$-operations on $\mathcal{W}Fuk^{\Lambda} (M)$, and it becomes the filtered $A_\infty$-category over $\Lambda$. 
Furthermore, by absorbing the weight $T^{E_{top} (u)}$ appearing in the $A_\infty$-operations into inputs and outputs using \eqref{eqn:topenergydist}, we have 
\begin{equation}\label{eqn:insertenergy}
\mathcal{W}Fuk (M) \otimes_{\mathbb{C}} \Lambda \hookrightarrow \mathcal{W}Fuk^{\Lambda} (M).
\end{equation}
To fix the notations, we describe the functor \eqref{eqn:insertenergy} more explicitly below.
For exact Lagrangians $L$ and $L'$, suppose a Hamiltonian chord $P \in CF( L,L';wH)$ gives rise to the generator $\widetilde{P} \in \hom_{\mathcal{W}Fuk(M)} (L,L') = CW(L,L')$.
Then the above functor sends $\widetilde{P}$ regarded as a morphism in $\mathcal{W}Fuk (M)$ to $T^{\mathcal{A}_{wH} (P)} P$ where $f_L$ and $f_L'$ are such that 
$$  \theta|_L = d f_L \quad \mbox{and} \quad  \theta|_{L'} = d f_{L'},$$
and the same symbol $P$ now represents the standard generator of the morphism space in $\mathcal{W}Fuk^{\Lambda} (M)$ supported at the Hamiltonian chord $P$. 
By abuse of notation, we denote the normalized generator $T^{\mathcal{A}_{wH} (P)} P$ by $\widetilde{P}$ regarded as a morphisms in $\mathcal{W}Fuk^\Lambda (M)$, which is actually the image of the standard generator (written by the same notation $\widetilde{P}$) in $\mathcal{W}Fuk(M)$ supported at $P$. $\widetilde{P}$ will be referred to as an \emph{exact generator} for this reason. 

\begin{defn}\label{def:exactgens}
A morphism in $\mathcal{W}Fuk^\Lambda (M)$ is said to be an exact generator if it is an image of a generator $\hom_{\mathcal{W}Fuk(M)} (L,L') $ under \eqref{eqn:insertenergy}. 
\end{defn}

Throughout, our notation will follow the same convention, i.e., for a geometric generator (a chord or an intersection point) $P$, $\widetilde{P}$ is the (image in $\mathcal{W}Fuk^{\Lambda} (M)$ of) morphism in the exact category, and hence among them, $A_\infty$-operations as well as continuations maps do not produce any nontrivial powers of $T$.
If $L \in \mathcal{W}^\Lambda$ is a compact object, then one may alternatively use Morse-Bott model $H^\ast (L;\Lambda)$ (or its variants) for $\hom_{\mathcal{W}^\Lambda} (L,L)$, which coincides with $CF(L,L;\Lambda)$ described earlier. For non-immersed generators in $H^\ast (L;\C) (\subset H^\ast (L;\Lambda)$, the corresponding exact generators are simply themselves, which is somewhat consistent with \eqref{eqn:auxvalH} since $H=0$ for the Morse-Bott model and $L=L'$ for an endomorphism space. By definition, their valuations $\nu$ vanish also.

On the other hand, immersed generators have slightly a different feature in this regard. Let $X$ be an immersed generator of the (Morse-Bott) Floer complex $CF(L,L;\Lambda)$ of an exact Lagrangian immersion $\iota: \underline{L} \looparrowright L$, supported at the self-intersection $\iota(X_+) = \iota(X_-)$. Suppose that the generator $X$ represents the branch jump from $X_-$ to $X_+$ in Floer theory of $L$.
One can assign another real number to $X$ analogously to \eqref{eqn:auxvalH} as follows. Let $f_{\underline{L}}$ be the chosen primitive of $\iota^\ast \theta$ for the exact brane $L$. We denote by $\mathcal{A}(X)$ the difference of the values of $f_{\underline{L}}$ at the branch jump $X$, 
\begin{equation}\label{eqn:actionforimm}
\mathcal{A}(X):=f_{\underline{L}} (X_+) - f_{\underline{L}} (X_-) \in \R.
\end{equation}
It will play an important role, indicating the position of $L$ in the global moduli of Lagrangian. 
 As before, if we use the scaling $\widetilde{X}=T^{\mathcal{A}(X)} X $ for immersed generators, then the $\AI$-operation does not involve the area term, which is essentially going back to the $\C$-linear category $\mathcal{W}^\C$. Thus an exact immersed generator $\widetilde{X}$ can have nonzero valuation $\nu (\widetilde{X})  = \mathcal{A}(X) \neq 0$.

We finish the discussion by the following simple example of $M=T^\ast S^1$ to illustrate how the actions of generators of the wrapped Floer cohomology are given.

\begin{example}\label{ex:cotcirclenegt}
Let $G$ be a cotangent fiber in $M:=T^\ast S^1$. 
If we identify $T^\ast S^1 = \R \times S^1$ and use $(s,t)$ as local coordinates (with $s \in \R$ and $t \in [0,1] / 0 \! \sim  \! 1$), the symplectic form is given by $\omega=dsdt$ with $\Theta=sdt$. (In view of $X = \mathbb{C}^\times$, $(t,s)$ and $z \in \mathbb{C}^\times$ are related by $\log z = s+ it$.) Identifying the subset $\{ \leq 1-\epsilon |s| \}$ as the conical end, one take $H(s,t) := |s| / (1-\epsilon)$ together with a suitable interpolation by a quadratic function over $-1 + \epsilon < s < 1- \epsilon$. Here, $\epsilon$ is a small positive real number so as to avoid integer Reeb chords at the contact boundary of $M$.

For $w=1,2,\cdots$, $CF(G,G;wH;\Lambda)$ is spanned by chords $\{Z_{i}^{(w)} : - \frac{w}{1-\epsilon} < i <  \frac{w}{1-\epsilon} \}$ where $Z_{i}^{(w)}$ wraps $i$ times the cylinder $M$. By a direct calculation, $\mathcal{A}_{wH} (Z_{i}^{(w)})= C \left( - \frac{i^2}{w} + \frac{w}{(1-\epsilon)^2}\right)$, and hence  the corresponding exact generator is given by
$\widetilde{Z_{i}^{(w)}}= T^{C \left( - \frac{i^2}{w} + \frac{w}{(1-\epsilon)^2}\right)} Z_{i}^{(w)}$. ($C$ is a fixed positive real number only depending on $H$.)

The obvious continuation map $CF(G,G;wH;\Lambda) \to CF(G,G;(w+1)H;\Lambda)$ identifies $\widetilde{Z_{i}^{(w)}}$ and $\widetilde{Z_{i}^{(w+1)}}$.
 $\left\{\widetilde{Z_0}:=\widetilde{Z_0^{(1)}} , \widetilde{Z_i}:=\widetilde{Z_{i}^{(|i|)}} : i \in \mathbb{Z} \setminus \{0\} \right\}$ (or more precisely, their images in the cohomology) can be taken as a basis of the cohomology $HW(G,G;\Lambda)$. 
It is well-known that $m_2 (\widetilde{Z_i}, \widetilde{Z_j}) = \widetilde{Z_{i+j}}$ is the only nontrivial operation.

\end{example}

\subsection{Maurer-Cartan equation}\label{subsec:MCeqns}

Let $\bL$ be a graded unobstructed immersed Lagrangian, which is compact. We assume that $\bL$ consists of a single irreducible component for simplicity. Consider its Lagrangian Floer complex $CF(\bL,\bL;\Lambda)$ which is naturally an $A_\infty$-algebra. 
We assume that $CF(\bL,\bL;\Lambda)$ is supported on nonnegative degrees and the degree 0 component $CF^0 (\bL,\bL;\Lambda)$ is generated by the unit class. By passing to the minimal model, the condition is satisfied if $\bL$ does not have any immersed generators with nonpositive degrees.
To be more concrete, we will take the underlying vector space of $CF(\bL,\bL;\Lambda)$ to be $H^\ast (\bL;\Lambda)$ (adjoined with immersed generators if $\bL$ is immersed). 

In \cite{CHL}, the localized mirror associated to $\bL$ was constructed as the space of Maurer-Cartan deformation of $\bL$ defined in the following way\footnote{\cite{CHL} mainly deals with the \emph{weak} Maurer-Cartan deformation}. Let $\{X_1 ,\cdots, X_l\}$ be a basis of the degree 1 component $CF^1 (\bL,\bL;\Lambda_0)$, inducing that of $CF^1 (\bL,\bL;\Lambda) = CF^1 (\bL,\bL;\Lambda_0) \otimes_{\Lambda_0} \Lambda$ formed by elements with zero valuations. For (free) formal variables $x_1, \cdots, x_l$, we consider a linear combination $b= \sum x_i X_i$, and compute 
\begin{equation}\label{eqn:m1bm2bb}
 m_1(b) + m_2(b,b) + \cdots
\end{equation}
where we use the convention
\begin{equation}\label{eqn:earconv}
m_k (x_{i_1} X_{i_1}, \cdots, x_{i_k} X_{i_k}) = x_{i_k} \cdots x_{i_1} m_k (X_{i_1}, \cdots, X_{i_k}).
\end{equation}
Notice that elements of $\Lambda_+$ can only be plugged into formal variables $x_1,\cdots,x_l$ in order to to make sense of possibly an infinite sum \eqref{eqn:m1bm2bb}.

Since $CF(\bL,\bL;\Lambda)$ is $\Z$-graded, the outcome must be a linear combination of degree 2 generators of $CF(\bL,\bL;\Lambda)$ with coefficient being noncommutative formal power series in $x_1,\cdots, x_l$, i.e.,
$$    m_1(b) + m_2(b,b) + \cdots =  f_1 (x_1,\cdots,x_l) X_{l+1}  + \cdots + f_{N'} (x_1,\cdots, x_l) X_{N'}$$
where $X_{l+1}, \cdots, X_{N'}$ generate $CF^2 (\bL,\bL;\Lambda)$. Let $\Lambda\{\!\{ x_1, \cdots, x_l \}\!\}$ denote the space of (noncommutative) formal power series over $\Lambda$ with bounded coefficients. 
We see that if $x_i$'s are taken from the algebra
\begin{equation}\label{eqn:MCalgL}
 \dfrac{ \Lambda\{\!\{ x_1, \cdots, x_l \}\!\}}{\langle \langle f_1, \cdots, f_{N'} \rangle \rangle},
\end{equation}
then the associated $b$ solves the Maurer-Cartan equation
\begin{equation}\label{eqn:MCeqns}
 m_1(b) + m_2(b,b) + \cdots=0.
\end{equation}
Here, the ideal $\langle \langle f_1, \cdots, f_{N'} \rangle \rangle$ is defined as follows. Note that a series $\sum_{l=0}^\infty \sum_{I \in (\Z_{>0})^l} \lambda_I  x^I$ belongs to $\Lambda\{\!\{ x_1, \cdots, x_l \}\!\} $ if and only if the set $\{ val(\lambda_I) \colon I \in (\Z_{>0})^l \}$ ($\subset \Lambda$) is bounded below. 
(Here, $x^I = x_{i_1} \cdots x_{i_l}$ for $I=(i_1,\cdots,i_l)$.) 
We call $\inf \{ val(\lambda_I) \colon I \in (\Z_{>0})^l \}$ 
the ($T$-adic) \emph{valuation} of $\sum_{I \in (\Z_{>0})^l} T^{\lambda_I} x^I$. $\langle \langle f_1, \cdots, f_{N'} \rangle \rangle$ consists of elements  of the form
\begin{equation}\label{eqn:glfilhlfirst}
\sum_{k=1}^{N'} \sum_{l=0}^{\infty} \sum_{|I | + |J| = l} \lambda_{k, I, J} \,\, x^{I} \, f_k \, x^{J}
\end{equation}
such that valuations of elements in the set $\{ \lambda_{k, I , J}   \}$ are bounded below, where $|I|=m$ for $I \in (\Z_{>0})^m$.
As we will see later, $\Lambda\{\!\{ x_1, \cdots, x_l \}\!\}$ can be obtained as the continuous dual of the tensor algebra of $CF^1(\mathbb{L},\mathbb{L})$ (over $\Lambda$) with respect to a natural $T$-adic topology.  Correspondingly, $\langle \langle f_1, \cdots, f_{N'} \rangle \rangle$ is the closure of the ideal generated by $f_1,\cdots, f_{N'}$.

If  $\{Y_i\}$ is another basis of $CF^1 (\bL,\bL;\Lambda_0)$, then $Y_i = \sum_{j=1}^l a_{ij} X_j$ for some $a_{ij}$ such that $(a_{ij}) = (\underline{a}_{ij}) + (b_{ij}) $ where $(\underline{a}_{ij})$ is an invertible matrix with entries in $\mathbb{C}$ and $b_{ij} \in \Lambda_+$. Thus, the inverse $(a^{ij})$ of $(a_{ij})$ is given as $(a^{ij}) = (\underline{a}^{ij}) - (\underline{a}^{ij}) (b_{ij}) (\underline{a}^{ij}) + (\underline{a}^{ij}) (b_{ij}) (\underline{a}^{ij}) (b_{ij}) (\underline{a}^{ij}) - \cdots$
where $(\underline{a}^{ij}) := (\underline{a}^{ij})^{-1} $   (the expression converges since $b_{ij} \in \Lambda_+$). 
The associated variable $y_i$ can be written as $y_i = \sum_{j=1}^{l} a^{ji} x_j$, and this gives an isomorphism between $\Lambda\{\!\{ x_1, \cdots, x_l \}\!\}$ and $\Lambda\{\!\{ y_1, \cdots, y_l \}\!\}$ by a valuation preserving (linear) coordinate change.

\begin{defn}\label{def:mcalgell}
The Maurer-Cartan algebra of $\bL$ is defined to be
$$A_\bL :=   \dfrac{ \Lambda\{\!\{ x_1, \cdots, x_l \}\!\}}{\langle \langle f_1, \cdots, f_{N'} \rangle \rangle}.$$
\end{defn}

Intuitively, the Maurer-Cartan space from $\bL$ is the (possibly noncommutative) space whose function ring is the Maurer-Cartan algebra of $\bL$. For a Maurer-Cartan solution $b$ (that is, $b$ satisfying \eqref{eqn:MCeqns}) with coefficients in $\Lambda_+$, $(\bL,b)$ defines an unobstructed object whose Floer cohomology is well-defined. Therefore, the Maurer-Cartan space can be thought of as a local moduli formed by objects in the Fukaya category near $\bL$ (see, for instance, \cite{HL-gokova}). 

\begin{example}\label{ex:T^nMCspace}
Let us consider $(\C^\ast)^n=T^\ast T^n$, the cotangent bundle of the $n$-dimensional torus $\bL:=T^n$. Since $\AI$-operations admits no contribution from nonconstant discs, we only need to take into account the cup product on $H^\ast (\bL)$. Since $H^\ast (\bL)$ is an exterior algebra generated by $d \theta_1, \cdots, d \theta_n$, we see that the Maurer-Cartan equation for $b= \sum x_i \theta_i$ gives
$$ m(e^b) = \sum_{i,j} (x_i x_j - x_j x_i) d\theta_i \wedge d\theta_j,$$
and hence the Maurer-Cartan algebra of $\bL$ is the symmetric algebra on $n$-letters, $x_1, \cdots x_n$. Taking into account the convergence issue over $\Lambda$, the Maurer-Cartan algebra is given by
\begin{equation}\label{eqn:MCforxitori}
 \dfrac{ \Lambda\{\!\{ x_1, \cdots, x_n \}\!\} }{ \langle \langle  x_i x_j - x_j x_i \mid 1 \leq i \neq j  \leq n \rangle \rangle}
\end{equation}
which can be thought of as a certain completion of the polynomial ring over $\Lambda$ in $n$-variables, and one can think of $x_1,\cdots,x_n$ as variables for $(\Lambda_+)^n$.
\end{example}

There is a way to include energy-$0$ Maurer-Cartan deformation, which is to use $\Lambda_\mathrm{U}$-connection $\rho$ on $\bL$. Fix a generator $X_1,\cdots, X_l$ of $H^1(\bL;\mathbb{Z})$, and suppose $\rho$ has holonomy $\rho_i \in \Lambda_\mathrm{U}$ along $PD(X_i)$. The Maurer-Cartan deformation by $\sum x_i X_i$ with $x_i \in \Lambda_+$ can be enhanced using $\rho$ by introducing a new variable $z_i = \rho_i e^{x_i}$. Here, the appearance of the expression $e^{x_i}$ is natural in the sense that $x_i$ appears in  $A_\infty$-operations deformed by $b=\sum x_i d\theta_i$ as an exponential due to the divisor axiom. However, the change of variables $x_i$ into $e^{x_i}$ has a dramatic effect on the valuation, since $e^{x_i}$ always has valuation zero for any $x_i \in \Lambda_+$.

\begin{example}\label{ex:expvar11}
Applying the above discussion to Example \ref{ex:T^nMCspace} with $b=(\rho, \sum_i x_i d \theta_i)$, we obtain an enlarged deformation space whose Maurer-Cartan algebra is given by
\begin{equation}\label{eqn:MCzitori}
\dfrac{ \Lambda\{ z_i \colon i \in \mathbb{Z} \}}{ \langle z_i z_j = z_{i+j} : i,j \in \mathbb{Z} \rangle}
\end{equation}
where $\Lambda\{ z_i \colon i \in \mathbb{Z} \}$ consists of an infinite sum $ \sum_{k=1}^{\infty} a_k z_{i_k}$ with $\lim val(a_k) = \infty$.  
One does not need to take the closure of the ideal $\langle z_i z_j = z_{i+j} : i,j \in \mathbb{Z} \rangle$ since it is automatically closed due to some general fact about affinoid algebras, see \cite[Proposition 3, 6.1.1]{BGR} and \cite[6.1.4]{BGR}. 
 
Notice that \eqref{eqn:MCforxitori} describes the ring of convergent (analytic) functions on $(\Lambda_+)^n$, whereas \eqref{eqn:MCzitori} is that of $(\Lambda_{\mathrm{U}})^n$. The exponential coordinate change $z_i = \rho_i e^{x_i} = \rho_i + \rho_i x_i + \cdots$ clearly explains the necessity of the condition $\lim val(a_k) = \infty$,
since otherwise $\sum_{k=1}^{\infty} a_k z_{i_k}$ would not give a well-defined element in \eqref{eqn:MCforxitori} after coordinate change back to $x_i$.
\end{example}

The last example concerns the Maurer-Cartan deformation by immersed generators of a Lagrangian, which can be intuitively thought of as smoothing out the corresponding self-intersections.

\begin{example}\label{ex:conimc}
Consider the deformed conifold $ \{ (u_1,v_1,u_2,v_2) \in \mathbb{C}^4 : u_1 v_1 - u_2 v_2 = \epsilon\}$ with $\epsilon \neq 0$. It admits a double conic fibration by writing the defining equation as $u_1 v_1 = z-a$ and $u_2 v_2 = z-b$ (with $b-a = \epsilon$). Let $M:=  \{ (u_1,v_1,u_2,v_2) \in \mathbb{C}^4 : u_1 v_1 - u_2 v_2 = \epsilon\} \setminus \{z=0\}$, which is the anti-canonical divisor complement of a deformed conifold. The projection to $z$-plane defines a double-conic fibration, and 
two matching cycles $\mathbb{L}_0$ and $\mathbb{L}_1$ over the paths drawn in Figure \ref{fig:conizbase} are Lagrangian 3-spheres. 

\begin{figure}[h]
	\begin{center}
		\includegraphics[scale=0.5]{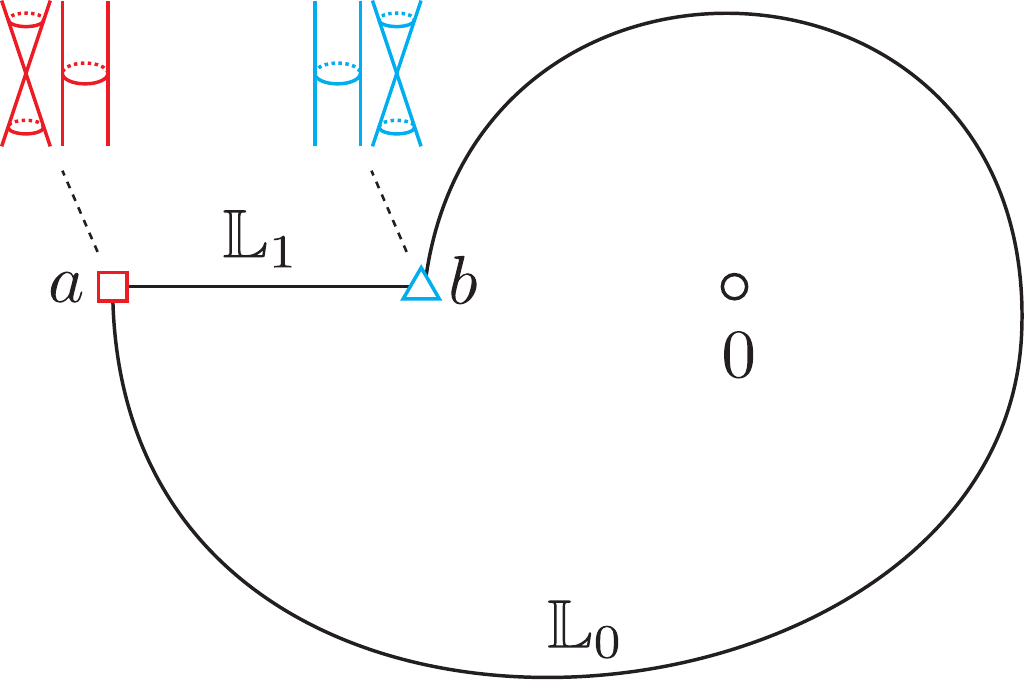}
		\caption{A double conic fibration on the deformed conifold with two singular fibers}
		\label{fig:conizbase}
	\end{center}
\end{figure}

Consider the union $\mathbb{L}:=\mathbb{L}_0 \oplus \mathbb{L}_1$ regarded as an immersed Lagrangian. Floer theory on $\mathbb{L}$ has been examined in \cite{CPU} using a certain simplicial model. $\mathbb{L}$ has the Maurer-Cartan algebra
$$ \dfrac{ \Lambda\{\!\{x,y,z,w \}\!\}}{\langle \langle xyz-zyx, yzw, wzy, zwx-xwz, wxy - yxw, xz, zx, yw, wy\rangle \rangle}  $$
where $x,y,z,w$ are taken from the path algebra $\Gamma Q$ for
\begin{equation*}\label{eqn:conifoldquiver} 
Q:\xymatrix{
\bullet   \ar@/^2pc/[rr]^x \ar@/^1pc/[rr]^z  
&& \bullet \ar@/^1pc/[ll]^{y} \ar@/^2pc/[ll]^{w} }.
\end{equation*}
Here, the relations $xz=zx=yw=wy=0$ come from composability of arrows in $Q$.
\end{example}

In Example \ref{ex:conimc}, we can alternatively use the semi-simple coefficient ring $\Bbbk_\Lambda:= \Lambda \pi_{0} \oplus \Lambda \pi_{1} $ with $\pi_i \cdot \pi_i =\pi_i$ and $\pi_0 \cdot \pi_1 = \pi_1 \cdot \pi_0 =0$ which can tell us about composability of paths as well. In this case, we take $\Bbbk_\Lambda \{\!\{ x,y,z,w\}\!\}$ as a replacement for $\Lambda\{\!\{x,y,z,w \}\!\}$ with $x,y,z,w \in \Gamma Q$, where $\pi_0 x = x \pi_1 =x$, $x \pi_0  =  \pi_1 x =0$, $z \pi_0 = z \pi_1= 0$, and similar for $y$, $z$, $w$. In what follows, we will work over such a coefficient ring to efficiently include the case when $\mathbb{L}$ is given as a union of several Lagrangians.

\section{Preliminaries on homological algebra in filtered setting}

We review basic algebraic constructions that are relevant to our geometric applications below, adapting \cite[Section 2]{EL-duality} to $T$-adic setting.
 A related construction in the filtered and curved setting  seems to be carried out in \cite{DDL}. 
To begin with, we fix once and for all our base ring $\Bbbk$ to be a semisimple ring $\Bbbk_\Lambda=\oplus_{i =1}^r \Lambda \pi_i$ such that $\pi_i \cdot \pi_i = \pi_i$ and $\pi_i \cdot \pi_j = 0$ for $i \neq j$. 
There exists a valuation $val : \Lambda \to \R$ such that $val(T^{\lambda}) = \lambda$, which has an obvious extension to $\Bbbk_\Lambda$.  
We set $\Bbbk_{\mathbb{C}} = \oplus_{i=1}^r \mathbb{C} \pi_i$ so that $\Bbbk_\Lambda = \Bbbk_{\mathbb{C}} \otimes_{\C} \Lambda$.  For both $\Bbbk_\Lambda$ and $\Bbbk_\C$, the sum $\sum_{1 \leq i \leq r} \pi_i$ of idempotents defines the unit, which we denote by $\mathbf{1}$.

\subsection{Non-Archimedean topology on $\Bbbk_\Lambda$-modules}\label{subsec:underlyingmods}

In what follows, the underlying (chain-level) space $V$ of any algebraic structures we are to  consider will be of the following form. $V$ is a free $\Bbbk_\Lambda$-bimodule endowed with a \emph{valuation} $\nu : V \to \mathbb{R} \cup \{ \infty\} $ such that
\begin{itemize}
\item $ \nu (c X) =  val (c) + \nu(X)$ for $c \in \Bbbk$ and $X \in V$;
\item $\nu (X+Y) \geq \min \{ \nu (X), \nu (Y) \}$ for $X,Y \in V$;
\item $\nu (X) =  \infty$ if and only if $X=0$.
\end{itemize}
Notice that $X \mapsto e^{-\nu(X)}$ defines a non-Archimedean norm, and hence a topology on $V$. 
The valuation  $\nu$ on $V$ gives rise to a filtration of $V$ defined by
$$ \mathcal{F}^{\lambda} V :=\{ X \in V \colon \nu (X) \geq \lambda \}$$
that satisfies
\begin{itemize}
\item $ \mathcal{F}^{\lambda} V \subset \mathcal{F}^{\lambda'} V$ for $\lambda > \lambda'$,
\item $ c \mathcal{F}^{\lambda} V \subset \mathcal{F}^{\lambda + val(c)} V$ for $c \in \Lambda$.
\end{itemize}
The completion of $V$ with respect to the filtration  agrees with the completion of the normed vector space $V$ with respect to the norm $e^{-\nu}$.

$\{ \mathcal{F}^{\lambda} V\}$ induces a filtration on the $k$-fold tensor product $T_k V:=V^{\otimes k}$ of $V$ (over $\Bbbk_\Lambda$) by
\begin{equation}\label{eqn:filtontensorR}
\mathcal{F}^{\lambda} T_k V=\mathcal{F}^{\lambda} (V \otimes \cdots \otimes V) := \bigcup_{\lambda_1 + \cdots + \lambda_k = \lambda} \mathcal{F}^{\lambda_1} V \otimes \cdots \otimes \mathcal{F}^{\lambda_k} V
\end{equation}
One can also consider the completion of $TV = \oplus_{k} V^{\otimes k}$ with respect to the filtration
$$ \mathcal{F}^{\lambda} TV:= \bigcup_k \bigcup_{\lambda_1 + \cdots + \lambda_k = \lambda} \mathcal{F}^{\lambda_1} V \otimes \cdots \otimes \mathcal{F}^{\lambda_k} V $$
naturally extending \eqref{eqn:filtontensorR}.

\subsection{$\AI$-algebras and coalgebras in Novikov setting}\label{subsec:aiinlambda}
An $\AI$-algebra $(V,\{m_k\}_{k \geq 1})$ over $\Bbbk_\Lambda$ is a free $\Z$-graded filtered $\Bbbk_\Lambda$-bimodule $V$ equipped with a sequence of multilinear operations
$$ m_k : V^{\otimes k} \to V$$
of degree $2-k$ which satisfies the relations \eqref{eqn:airelforLfooo}. 
Here, the filtered condition means that $V$ is equipped with a filtration as in \ref{subsec:underlyingmods} and that $m_k$ for each $k$ is a filtered map, i.e.,
$$ m_k ( \mathcal{F}^{\lambda_1} V_{i_1}, \cdots, \mathcal{F}^{\lambda_k} V_{i_k} ) \subset \mathcal{F}^{\lambda_1 + \cdots + \lambda_k + \lambda_d} V_{i_1+\cdots i_k+2-k}$$
for all $k$. Analogous conditions will be required for other algebraic structures too.

\begin{remark}\label{rmk:signdgai}
We use the following sign change that turns an $\AI$-algebra with $m_{k \geq 3} \equiv 0$ into a dga
$$ d(X) = 	m_1 (X), \qquad X \cdot Y = (-1)^{|X|} m_2 (X,Y).$$
(See for e.g., \cite[Appendix A]{CHL_gl}.)
Adding extra signs is necessary due to difference between $\AI$- and dg-sign conventions. To avoid a potential confusion, an $\AI$-algebra $(V,m_1,m_2, m_{\geq 3} \equiv 0)$ will be referred to as an $\AI$-dga. 
\end{remark}
 
For two $\AI$-algebras $(V_1,\nu_1)$ and $(V_2,\nu_2)$, an $\AI$-homomorphism $\phi$ is defined as a sequence of filtered multilinear maps $\phi_k : T_k V_1 \to V_2$ of degree $1-k$ satisfying
$$ \sum_k \sum_{i_1,\cdots, i_k \geq 1} m_l^{V_2} (\phi_{i_1} (\mathbf{X}^{(1)} ), \cdots, \phi_{i_l} (\mathbf{X}^{(l)}) )  = \sum (-1)^{|\mathbf{X}^{(1)}|'} \phi_{k_1} (\mathbf{X}^{(1)}, m_{k_2}^{V_1} (\mathbf{X}^{(2)} ), \mathbf{X}^{(3)}  ).$$
where we write $\mathbf{X} \in B_k V_1$ as $\mathbf{X} = \mathbf{X}^{(1)} \otimes \cdots \otimes \mathbf{X}^{(l)}$ on the left hand side with $\mathbf{X}^{(i)} \in T_j V$ for some $j$, and the decomposition on the right hand side is similarly defined. 
 
An $\AI$-right module $(E,\nu_E)$ over $V$ is a filtered $\Bbbk_\Lambda$-bimodule that is equipped with a sequence of filtered multilinear maps $\{n_{1|k}\}$, $ n_{k |1} : E \otimes T_k V  \to E$ of degree $1-k$, satisfying
 $$ \sum (-1)^{|Y|'+ |\mathbf{X}^{(1)}|'} n_{1|k_1}  (Y, \mathbf{X}^{(1)}, m_{k_2} (\mathbf{X}^{(2)}), \mathbf{X}^{(3)}) + \sum n_{1|k_1} ( n_{1|k_2} ( Y, \mathbf{X}^{(1)} ), \mathbf{X}^{(2)}))=0.$$
Left modules are similarly defined. 
Let $E_1$ and $E_2$ be two $\AI$-right modules over $V$.
A degree $d$ pre-$\AI$-morphism $f$ between two $\AI$-modules over $V$ is given by a sequence of filtered multilinear maps $\{ \psi_{k,1}\}$, $\psi_{k,1} : T_k V \otimes E_1 \to E_2$ of degree $1-k+d$.
We denote by $\hom_V^d (E_1,E_2)$ the set of all degree $d$ pre-$\AI$-module homomorphisms from $E_1$ to $E_2$,
and set $\hom_V (E_1, E_2) := \oplus_d \hom_V^d (E_1,E_2)$. 

There is a differential $\mathfrak{M}_1$ (in $\AI$-convention) on $\hom_V (E_1, E_2)$ defined by
\begin{equation}\label{eqn:mathfrakm1}
\begin{array}{rl}
\mathfrak{M}_1 (\psi)_{1|k} (Y, X_1,\cdots,X_k) := & \sum (-1)^{ |\psi'|+ |Y|'+ |\mathbf{X}^{(1)}|'} \psi_{1|k_1} ( Y, \mathbf{X}^{(1)}, m_{k_2}  (\mathbf{X}^{(2)}) ,  \mathbf{X}^{(3)} ) \\ 
& + \sum   (-1)^{|\psi|'}  \psi_{1|k_1} ( n^{E_1}_{1|k_2} (  Y, \mathbf{X}^{(1)} ),  \mathbf{X}^{(2)}  ) \\
& + \sum   n_{1|k_1}^{E_2} (   \psi_{1|k_2} ( Y, \mathbf{X}^{(1)}) ,\mathbf{X}^{(2)} ).
\end{array}
\end{equation}
It is elementary to check that $\mathfrak{M}_1^2=0$.
$\mathfrak{M}_1$-closed pre-morphisms are usually called $\AI$-module homomorphisms. 
 When $E_1= E_2=E$, we can compose (in $\AI$-convention) two morphisms $\{ \psi_{1|k}\}$ and $\{ \phi_{1|k}\}$ from $E$ to itself by
 \begin{equation}\label{eqn:mathfrakm2}
 \mathfrak{M}_2(\phi, \psi)_{1|k} (Y,X_1, \cdots, X_k) = \sum  (-1)^{|\phi|'}  \phi_{1|k_1} (\psi_{1|k_2} (Y,\mathbf{X}^{(1)}) , \mathbf{X}^{(2)} ),
 \end{equation}
and $(\hom_V (E, E), \mathfrak{M}_1, \mathfrak{M}_2)$ defines an $\AI$-dga in this case.

We next discuss an $A_\infty$-coalgebra $C$, which is the dual notion to an $\AI$-algebra. 
 
\begin{defn}
An $A_\infty$-coalgebra $C$ over $\Bbbk_\Lambda$ is a free $\mathbb{Z}$-graded filtered $\Bbbk_\Lambda$-module equipped with a family of filtered maps 
$ \Delta_k : C \to C^{\otimes k}$
of degree $k-2$ such that, for each $k \geq 1$,
\begin{equation}\label{eqn:coalgrel}
 \sum_{\substack{ r+s+t=k \\ u = r+1+t }} (-1)^*(1^{\otimes r} \otimes \Delta_s \otimes 1^{\otimes t}) \Delta_u    =0.
\end{equation}

 For two $\AI$-coalgebras $C$ and $C'$ over $\Bbbk_\Lambda$, a morphism between them is a sequence of filtered linear maps $\{ \phi_k : C \to C'^{\otimes k} \}$ satisfying
 $$ \sum(-1)^{\ast} (id^{\otimes k_1} \otimes \Delta_{k_2} \otimes id^{\otimes k_3} ) \circ \phi_{k_1+k_2 + k_3} = \sum (\phi_{i_1} \otimes \phi_{i_2} \otimes \cdots \otimes \phi_{i_m} ) \circ \Delta_m. $$
\end{defn}

The signs in \eqref{eqn:coalgrel} follow the Koszul convention (with respect to the shifted degrees). For instance, in \eqref{eqn:coalgrel},
$$(-1)^* (1^{\otimes r} \otimes \Delta_s \otimes 1^{\otimes t}) (\mathbf{c}^{(1)}, \mathbf{c}^{(2)}, \mathbf{c}^{(3)} ) = (-1)^{|\mathbf{c}^{(1)}|'} \mathbf{c}^{(1)} \otimes \Delta_s (\mathbf{c}^{(2)}) \otimes \mathbf{c}^{(3)}$$
where $\mathbf{c}^{(1)}\otimes \mathbf{c}^{(2)} \otimes \mathbf{c}^{(3)}  =(c_1\otimes \cdots \otimes c_r) \otimes (c_{r+1} \otimes \cdots \otimes c_{r+s}) \otimes (c_{r+s+1} \otimes \cdots \otimes c_{r+s+t})$.

A \emph{dg-coalgebra} $(C,d,\Delta: C \to C  {\otimes} C)$ is a special case  of an $A_\infty$-coalgebra with $\Delta_{3 \geq} \equiv 0$, where $d$ and $\Delta$ differ from $\Delta_1$ and $\Delta_2$ by $d (c) =  \Delta_1(c)$, $\Delta (c)=\sum (-1)^{|c_1|} c_1 \otimes c_2$ when $\Delta_2(c) = \sum c_1 \otimes c_2$. The sign change makes $(d,\Delta)$ satisfy the usual graded co-Leibnitz rule. 
We will call $(C,\Delta_1,\Delta_2, \Delta_{\geq 3 } \equiv 0)$ an $A_\infty$-dg-coalgebra to avoid potential confusion.

\subsection{Bar construction}

We next recall bar/cobar constructions in our setting.
A dg-coalgebra is closely related with an $\AI$-algebra structure via non-reduced bar construction. (The bar construction will refer to the reduced one in this paper.) Namely,
$A_\infty$-structure $\{m_k\}$ on $V$ over $\Bbbk_\Lambda$ is equivalent to a dg-coalgebra structure on $\WT{B} V:= T (sV) = {\oplus}_{k} T_k(sV) $, which is a non-reduced bar construction of $V$ (see for e.g. \cite{keller2001}). Here, $sV$ is a suspension of $V$ which shifts degree by $-1$, and hence for $x \in V$, $|sx| = |x|-1$, which also agrees with the shifted degree $|x|'$. (If there is no danger of confusion, we will use the same notation $x$ for $sx$ in $TV$.)
$\Delta$ and $d$ on $TV$ are defined by
\begin{equation}\label{eqn:DeltaTV}
\Delta: \WT{B} V \to \WT{B} V \otimes \WT{B} V \qquad x_1 \otimes \cdots \otimes x_k \mapsto \sum_i (x_1 \otimes \cdots \otimes x_i ) \otimes (x_{i+1} \otimes \cdots \otimes x_k)
\end{equation}
\begin{equation}\label{eqn:codiffTV}
d : \WT{B} V \to \WT{B} V \qquad  x_1 \otimes \cdots \otimes x_k \mapsto \sum_{i,l} (-1)^{|x_1|' + \cdots + |x_i|'} x_1 \otimes \cdots \otimes m_l (x_{i+1} ,\cdots, x_{i+l}) \otimes \cdots \otimes x_k.
\end{equation}
One can easily check that this defines a dg-coalgebra (with respect to the shifted degree), and that the condition $d^2 = 0$ is equivalent to the $A_\infty$-relations among $\{m_k\}_{k \geq 1}$.

In this paper, we will mainly consider the \emph{reduced} bar construction, or simply the bar construction, for an augmented $\AI$-algebra, which is given as follows. 
Let $V$ be an augmented $\AI$-algebra, i.e., an $\AI$-algebra together with an $A_\infty$-homomorphism $\varepsilon: V \to \Bbbk_\Lambda$. 
If $V$ is unital, one can decompose $V = \bar{V} \oplus \Bbbk_\Lambda$ where $\bar{V} := \ker \varepsilon$ is a (non-unital) $\AI$-algebra. 

We further assume that $\varepsilon$ is strict, in the sense that it does not involve higher components $V^{\otimes k} \to \Bbbk_{\Lambda}$ for $k \geq  2$. This implies that $\bar{V}$ is preserved by any $A_\infty$-operations. For instance,  $CF(L,L;\Lambda)$ for a compact unobstructed Lagrangian $L$ admits a (strict) augmentation $\varepsilon$ after passing to the minimal model if necessary, where $\varepsilon$ is the projection to the unit component.

\begin{defn}
The bar construction of an augmented $\AI$-algebra $(V,\varepsilon)$ over $\Bbbk_\Lambda$ is a dg-coalgebra on $B V = T (s\bar{V})$ where $\bar{V}=\ker \varepsilon$ and the operations $d$ and $\Delta$ are given by the formulas \eqref{eqn:codiffTV} and \eqref{eqn:DeltaTV} (with $V$ replaced by $\bar{V}$).
\end{defn}

Note that $BV$ includes $\Bbbk_\Lambda$ as a component, and is coaugmented via the inclusion of this component. In general, the coaugmentation of an $\AI$-coalgebra $C$ is an $\AI$-coalgebra homomorphism $\eta: \Bbbk_\Lambda \to C$.

\subsection{Cobar construction}\label{subsubsec:cobarconst}
To define cobar the construction of an $\AI$-coalgebra over $\Bbbk_\Lambda$, we need to enlarge the tensor algebra as follows.

\begin{defn}\label{def:tensorcocomple}
For a filtered $\Bbbk_\Lambda$-bimodule $C$, let $\widetilde{T} C$ be the inverse limit of $\{TC / \mathcal{F}_l TC\}_{l \geq 0}$ where $\mathcal{F}_l TC = \oplus_{i=l+1}^{\infty} C^{\otimes i}$. $\overline{T}C$ is defined as the subspace of $\widetilde{T} C$ consisting of elements all of whose factors have valuations bounded from below. 
One can express a general element of $\widetilde{T} C$ as
$$\sum_{l=0}^{\infty} \sum_{|\vec{\alpha}| =l} c_{\vec{\alpha}} x_{\vec{\alpha}} \qquad  (c_{\vec{\alpha}} \in \Bbbk_\Lambda \,\, \mbox{and} \,\, \vec{\alpha} = (\alpha_{1},\cdots, \alpha_{l}) \in I^{l} \,\, \mbox{and}\,\, x_{\vec{\alpha}} = x_{\alpha_1} \otimes \cdots \otimes x_{\alpha_l}).$$ 
(When $C$ is graded, then the sum is required to be supported on finitely many degrees.) 
This element belongs to $\overline{T}C$ if
\begin{equation}\label{eqn:tbarvalsummand}
val(c_{\vec{\alpha}}) + \sum_{i=1}^{l} \nu (x_{\alpha_{i}}) > \lambda
\end{equation}
for some $\lambda \in \R$ for all $l$ and $\alpha$. 
\end{defn}

Given an $A_\infty$-coalgebra $(C,\Delta_k)$ over $\Bbbk_\Lambda$ with the coaugmentation $\eta: \Bbbk_\Lambda \to C$, there is a natural way to obtain a dg-algebra, called the cobar construction. 
We assume $\eta$ is strict in the sense that $\eta_k\equiv 0 $ for $k \geq 2$.

\begin{defn}\label{def:coalgtbar}
Let $C$ be a coaumented $\AI$-coalgebra over $\Bbbk_\Lambda$, and $\bar{C}:= C / \eta(\Bbbk)$. Its cobar construction is the dg-algebra $\Omega C := \overline{T}(s^{-1} \bar{C})$, where the differential $\delta:=d_{\Omega C}$ is given by
\begin{equation}\label{eqn:cobardiff}
\delta( x_1 \otimes \cdots \otimes  x_k) = \sum_{i,j} (-1)^{|x_1|' + \cdots + |x_{i-1}|'} x_1 \otimes \cdots \otimes x_{i-1} \otimes \Delta_j (x_i) \otimes x_{i+1} \otimes \cdots \otimes x_k
\end{equation}
The multiplication $\cdot$ on $\Omega C$ is defined to be the usual concatenation of tensors.
\end{defn}
\eqref{eqn:cobardiff} is an element of $\overline{T}(s^{-1} \bar{C})$ since $\Delta_j$'s are filtered. 
One can check that $\delta$ satisfies the graded Leibnitz rule (with respect to the degree shifted by $s^{-1}$) and $\delta^2=0$, and hence $\Omega C$ defines a dga. Also, $\Omega C$ is augmented via the projection to the component $\Bbbk_\Lambda$ in $\overline{T}(s^{-1} \bar{C})$.

\begin{remark}\label{rmk:threecobar}
One can analogously perform the cobar construction on the length completion $\tilde{T} C$. We write $\hat{\Omega} C$ for the resulting dga. On the other hand, for a dg-coalgebra $C$, the construction can be carried out for the usual tensor algebra $T \bar{C}$, which we denote by $\underline{\Omega} C$. Clearly, $\hat{\Omega} C= \varprojlim_l \underline{\Omega} C / \mathcal{F}_l \underline{\Omega} C$ where $\mathcal{F}_l \underline{\Omega} C$ denotes the subspace of elements with tensor length$\geq l+1$. For a dg-coalgebra $C$ over $\Bbbk_{\Lambda}$ (in the filtered setting), we have $\underline{\Omega} C \hookrightarrow \Omega C  \hookrightarrow  \hat{\Omega} C $.
\end{remark}

\subsection{Dual of $\AI$-algebras and $\AI$-coalgebras}
In our $T$-adic setting, the most natural way to take the dual of a vector space over $\Lambda$ is to consider bounded linear maps. Namely, linear functionals are required to satisfy certain boundedness conditions with respect to the valuation functions. We first discuss the dual of a $\Bbbk_\Lambda$-module.

Let us consider a non-graded $\Bbbk_\Lambda$-bimodule $V$ over $\Bbbk_\Lambda$ with a valuation $\nu : V \to \R \cup \{ \infty\}$. Let $V^{\vee}$ denote the set of all \emph{right} $\Bbbk_\Lambda$-module maps $V \to \Bbbk_\Lambda$ which are bounded. Recall that a $\Bbbk_\Lambda$-linear map $f: V \to \Lambda$ is said to be \emph{bounded} if
\begin{equation}\label{eqn:boundeddual}
 \inf \{ val(f (X)) \in \Bbbk_\Lambda \colon X\in V \,\, \mbox{with} \,\, \nu(X) = 0\} \neq - \infty.
\end{equation}
In this case, one can define $\nu^{\vee} : V^{\vee} \to \R \cup \{\infty\}$ to be this infimum. Such $f$ gives a continuous linear map between the two non-Archimedean normed spaces $(V,e^{-\nu})$ and $(\Bbbk_\Lambda, e^{-val})$, and $V^\vee$ is usually referred to as the  topological (or continuous) dual of $V$. One can check that $e^{-\nu^{\vee}}$ defines a non-Archimedean norm on $V^\vee$.
Since \eqref{eqn:boundeddual} is equivalent to $\inf \{ val(f (X)) - \nu(X) \in \Bbbk_\Lambda \colon X\in V\}$, we have
$$ val (f(X)) \geq \nu^{\vee} (f) + \nu (X)$$
for any $X \in V$. The $\Bbbk_\Lambda$-bimodule structure on $V^{\vee}$ is  defined by $( a \cdot f \cdot  b ) (X) = a \cdot f( b \cdot X)$ for $a,b \in \Bbbk_\Lambda$. One can analogously define the left dual $^{\vee} \! V$ of $V$, the set of bounded left $\Bbbk_\Lambda$-module maps from $V$ to $\Bbbk_\Lambda$, and $^{\vee} \! V$ becomes a $\Bbbk_\Lambda$-bimodule via $( a \cdot f \cdot  b ) (X) = f(X \cdot a) \cdot b$.

For a graded $\Bbbk_\Lambda$-module $V= \oplus_{d \in \Z} V_d$, we define $V^\vee= \oplus_d \left( V_{d} \right)^{\vee}$, where we set the grading on $ \left( V_{d} \right)^{\vee}$ to be $-d$ so that $\left( V^\vee \right)_d = \left( V_{-d} \right)^{\vee}$.
This resembles taking the graded dual of $V$, but the usual algebraic dual of each graded piece now being replaced by the topological dual.

Suppose that $\{X_i \colon i \in I\}$ freely generate the degree $d$ component $V_d$ over $\Bbbk_\Lambda$, and for each $i$, $X_i \in \pi_{\beta_i} \cdot V_d \cdot \pi_{\alpha_i}$ for some $\alpha_i$ and $\beta_i$.
For this choice of generators, one can define the \emph{coordinate function} $x_i \in\pi_{\alpha_i} \cdot (V_d)^\vee \cdot \pi_{\beta_i} \subset (V_d)^\vee$ of degree $-d$ by
$ x_i (X_j) := \delta_{ij} \, \pi_{\alpha_i}$.
and extending it $\Bbbk_\Lambda$-linearly.
Obviously, $x_i$ is bounded, and $\nu^{\vee} (x_i) = - \nu (X_i)$. In general, an element of $(V^{\vee})_d$ can be written as a possibly infinite linear combination
\begin{equation}\label{eqn:geneeltofvee}
\sum_{j =1}^{\infty} b_i x_{i} \qquad b_i \in \Bbbk_{\Lambda}
\end{equation}
satisfying $val(b_i) + \nu^\vee (x_i) > \lambda$ for all $i$ and some fixed constant $\lambda \in \R$. If $V_d$ is finitely generated, then \eqref{eqn:geneeltofvee} reduces to a finite linear combination over $\Bbbk_\Lambda$. A general element of $V^\vee$ is a finite sum of series like  \eqref{eqn:geneeltofvee} with different degrees.

\subsubsection{Dual of a tensor algebra}
We make a few remarks about the relation between the tensor algebra and the (topological) dual. For a free graded $\Bbbk_\Lambda$-bimodule $(V,\nu)$, the degree $d$ piece of $TV$ is given by
$$(TV)_d = \bigoplus_k \bigoplus_{d_1 + \cdots + d_k =d} V_{d_1} \otimes \cdots \otimes V_{d_k}.$$
Recall that $TV= \oplus_d (TV)_d$ is equipped with a filtration induced by $\nu$, and it makes sense to take its topological dual $(TV)^\vee$. Namely, it consists of linear maps $f: TV \to \Bbbk_\Lambda$ which support on finitely many degrees and satisfy $f(\mathcal{F}^{\lambda} T V ) \subset \mathcal{F}^{\lambda+\lambda_f} \Bbbk_\Lambda$

We always have a $\Bbbk_\Lambda$-bimodule homomorphism 
\begin{equation}\label{eqn:tenftenl}
(V_{d_1})^{\vee} \otimes \cdots \otimes (V_{d_k})^{\vee} \to  (V_{d_k} \otimes \cdots \otimes V_{d_1})^{\vee} 
\end{equation}
by letting $(y_1 \otimes \cdots \otimes y_k) \, (Z_k \otimes \cdots \otimes Z_1) :=y_1(  \cdots y_{k-1} (y_k (Z_k) \, Z_{k-1} ) \cdots Z_1),$
which clearly defines a bounded map.
It induces maps $T_k (V^\vee) \to T_k (V^\vee)$ and $ T(V^\vee) \to (TV)^\vee $ which are isomorphisms for a finitely generated $V$.
For the left-dual $^\vee \! V$, we instead use $(y_1 \otimes \cdots \otimes y_k) \, (Z_k \otimes \cdots \otimes Z_1) =y_k ( Z_k  \, y_{k-1} (Z_{k-1}   \cdots y_1(Z_1)))$, and proceed analogously.
From now on we will mainly consider the right dual to simplify the exposition.

\begin{lemma}\label{lem:bartvhat1}
If $V$ is finitely generated graded module over $\Bbbk_\Lambda$, then $(T V)^\vee \cong \overline{T} (V^\vee)$,
where $\overline{T} (V^\vee)$ is given as in Definition \ref{def:tensorcocomple}.
\end{lemma}
\begin{proof}
Let us fix a basis $\{X_i : i \in I\}$ of $V$ 
with $X_i \in \pi_{\beta_i} V \pi_{\alpha_i}$,
 and consider associated coordinate functions $\{x_i \colon i \in I\}$. $V \otimes \cdots \otimes V$ admits a basis $\{X_{i_1} \otimes \cdots \otimes X_{i_k} : k \geq 1, \alpha_{i_l} =  \beta_{i_{l+1}} \}$. Therefore an element of $(TV)^\vee$ is completely determined by its value at these element. Namely, it can be expressed as
\begin{equation}\label{eqn:bartvhat}
\sum_{k=1}^{\infty} \sum_{v \in I^k} c_v \vec{x}_{v}
\end{equation}
where $\vec{x}_{v} = x_{v_1} \otimes \cdots \otimes x_{v_k}$. For this to be an element of $(TV)^\vee$, it has to be a bounded map, and $c_v$ is nonzero for finitely many degrees only. Namely, the set of values of \eqref{eqn:bartvhat} at normalized basis elements 
$$\left\{T^{-\sum_k val(X_{v_{k}})} X_{v_{1}} \otimes \cdots \otimes X_{v_{k}} : (v_1,\cdots, v_k)  \in I^k, k=1,2,\cdots\right\}$$ with valuation  zero should be bounded below. Thus there exists $\lambda \in \R$ such that
$$ val(c_v) - \sum_{l=1}^{k} \nu (X_{v_{l}})=  val(c_v) + \sum_{l=1}^{k} \nu (x_{v_{l}})  > \lambda$$
for all $v$. Together with the fact that \eqref{eqn:bartvhat} supports only finitely many degrees, this is precisely a description of an element of $\overline{T} (V^{\vee})$.
\end{proof}

\subsubsection{Dual of a finitely generated $\AI$-algebra over $\Bbbk_\Lambda$}
Suppose we are given a filtered $\AI$-algebra $(V,\{m_k\}_{k\geq1})$ finitely generated over $\Bbbk_\Lambda$. 
For $f \in V^\vee$, we define
\begin{equation}\label{eqn:aiatoaica}
\Delta_k (f) (X_1    \otimes \cdots \otimes X_k) ) = (-1)^{|X_1|' + \cdots + |X_k|'} f(m_k(X_1,\cdots, X_k)) ),
\end{equation}
where we implicitly use the fact that \eqref{eqn:tenftenl} is an isomorphism for finitely generated $V$.
It is not difficult to show that $\Delta_k (f)$ is a bounded linear map, and hence $\Delta_k (f) \in (V^{\otimes k})^\vee \cong V^\vee \otimes \cdots \otimes V^\vee$. 
Therefore the dual $V^\vee$ of a filtered finitely generated $\AI$-algebra over $\Bbbk_\Lambda$ admits a structure of a filtered $\AI$-coalgebra over $\Bbbk_\Lambda$.

\subsubsection{Dual of an $\AI$-coalgebra over $\Bbbk_\Lambda$}
Let us next consider a filtered $\AI$-coalgebra $C$, not necessarily finitely generated. As before, consider the dual $C^\vee$ of $C$, the set of bounded $\Bbbk_\Lambda$-right-linear maps on $C$. 
If we define 
\begin{equation}\label{eqn:aicatoaia}
m_k(f_1,\cdots, f_k) (x) := (-1)^{|f_1|' + \cdots + |f_k|'} (f_1 \otimes \cdots \otimes f_k) (\Delta_k (x)),
\end{equation}
then one can easily check that $m_k(f_1,\cdots, f_k)$ is bounded and $m_k$ is filtered.
Finally, $\{m_k\}$ satisfies $\AI$-relations by dualizing the relations among $\{\Delta_k\}$. Hence the dual of a filterd $\AI$-coalgebra over $\Bbbk_\Lambda$ is naturally a filtered $\AI$-algebra over $\Bbbk_\Lambda$.

\begin{remark}\label{rmk:signdiffdg}
The duality between $A_\infty$-algebra and coalgebra \eqref{eqn:aiatoaica} and \eqref{eqn:aicatoaia} involves nontrivial signs.
On the other hand, a dga $(V,d, \, \cdot \,)$ and its dual dg-coalgebra $(V^\vee, \delta, \Delta)$ are related in the following way:
$$(\delta f) (x) = (-1)^{|f|} f(dx), \quad (\Delta f) (x \otimes y) =  f ( x \cdot y)
$$
for $f \in V^\vee$ and $x\in V$.
This is compatible with the sign difference between $A_\infty$ and dg-conventions in our setup.
\end{remark}

\section{Maurer-Cartan space of a Lagrangian and Koszul dual dga}

In this section, we identify the localized mirror obtained by Maurer-Cartan formalism of a Lagrangian $\bL$ (see Section \ref{sec:mcspaceofLag}) with the Koszul dual of the Floer complex of $\bL$. Consider a graded unobstructed \emph{compact} Lagrangian $\mathbb{L}$ in a symplectic manifold $M$. It is possible that $\bL$ is a finite union of several irreducible Lagrangians, $\bL= \oplus_{i \in \Gamma} \bL_i$, in which case we take the base ring $\Bbbk$ to be $\Bbbk_\Lambda=\oplus_{i \in \Gamma} \Lambda \pi_i$. Thus the Floer complex $CF(\bL,\bL;\Lambda)$ is naturally a bimodule over $\Bbbk_\Lambda$.

Throughout, we will assume $CF(\bL,\bL;\Lambda)$ is minimal and supported in nonnegative degrees only, and that all the immersed generators have positive degrees. In particular, we have $CF^0 (\bL,\bL) \cong \Bbbk_\Lambda$ spanned by the unit class, or more precisely the sum of units for irreducible components of $\bL$. The projection to the unit component $CF(\bL,\bL ;\Lambda) \to CF^0 (\bL,\bL ;\Lambda)$ defines a canonical augmentation $\epsilon : CF(\bL,\bL;\Lambda) \to \Bbbk_\Lambda$.

\subsection{The Maurer-Cartan algebras of Lagrangian revisited}\label{subsec:MCcobar}
For simplicity, let us write $V_\bL=CF(\bL,\bL;\Lambda)$ from now on, which is a filtered $\AI$-algebra over $\Bbbk_\Lambda$ as in the setting of the previous section. Recall from \ref{subsec:fukaialg} that $V_\bL$ is equipped with a non-Archimedean valuation which vanishes on $H^\ast (\bL;\C)$ and the standard geometric generators associated with self-intersections.
By our assumption, $V_\bL$ is supported on nonnegative degrees only, and $V^0_{\bL} \cong \Bbbk_\Lambda$.
To make the exposition more explicit, we fix generators $X_1,\cdots, X_N$ of $CF^{>0} (\bL,\bL)$ with $\nu (X_i)=0$ such that
\begin{equation*}\label{eqn:coordTQ}
\begin{array}{lcl}
V_\bL^1 &=& \mathrm{span} \langle  X_1,\cdots, X_l  \rangle     \\
V_\bL^{\geq 2} &=& \mathrm{span} \langle X_{l+1} ,\cdots, X_N \rangle.
\end{array}
\end{equation*}

Recall that the dual $CF(\bL,\bL)^\vee$ is endowed with a structure of a filtered  $A_\infty$-coalgebra. 
For simplicity, we write $V_{\bL}^\vee$ for $CF(\bL,\bL)^\vee$ in what follows. 
$V_{\bL}^\vee$ has counit and coaugmentation by dualizing those of $V_\bL$. We write $\eta: \Bbbk_\Lambda \to V_{\bL}^\vee$ for the coaugmentation, whose image is the (scalar multiples of) coordinate function for the unit class in $V_\bL$. 

We can choose a splitting $V_{\bL}^\vee = \overline{V_{\bL}^\vee} \oplus \Bbbk_\Lambda$ using the fixed generators above, where $ \overline{V_{\bL}^\vee}=V_{\bL}^\vee / \eta(\Bbbk_\Lambda)$.
Namely,
%
we take the coordinate function $x_i$ for $X_i$ as generators of  $\overline{V_{\bL}^\vee}$. Recall that the degree of $x_i$ is taken to be $|x_i|:=-|X_i|$.
The same notation $x_i$ has already been used for a coefficient in $b$ in the Maurer-Cartan algebra in \ref{subsec:MCeqns}, but we will see that the two can be naturally identified. 
For instance, if $\bL$ is a Lagrangian torus, a degree 1 generator $x_i$ can be thought of as a coordinate for the 1-cocycles $d \theta_i$ whereas other $x_{j}$'s are dual to wedge products among $d \theta_i$'s.

The coalgebra structure on $V_{\bL}^\vee$ has the following concrete description. 
Suppose that we have nontrivial coefficient $c_v \in \Bbbk_\Lambda$ of $X_i$ in the $A_\infty$-operation on $v:= X_{a_1} \otimes \cdots \otimes X_{a_k}$:
\begin{equation}\label{eqn:mkTQ}
m_k (X_{a_1}, \cdots, X_{a_k}) =\cdots+  c_v X_i + \cdots.
\end{equation}
For each such a tuple $v$, we have one summand $  x_{a_k} \otimes \cdots \otimes x_{a_1}$ in $\Delta_k (x_i)$:
$$ \Delta_k (x_i)= \cdots + c_v ( x_{a_k} \otimes \cdots \otimes x_{a_1} ) + \cdots.$$

The cobar construction of $(V_{\bL}^\vee,\eta)$ produces a dg algebra $\Omega V_{\bL}^\vee$. Its underlying $\Bbbk_\Lambda$-module $\overline{T} (s^{-1}  \overline{V_{\bL}^\vee})$ is generated by  formal (noncommutative) series in $x_i$'s with a fixed degree such that valuations of coefficients are bounded below. Note that in $\overline{T} (s^{-1} \overline{V_{\bL}^\vee})$, the degree of $x_i$ should be inverse-shifted, that is, $|x_i|'' := |x_i| +1$.
We will often omit the tensor product symbol, and simply write $ x_{a_k} \otimes \cdots \otimes x_{a_1}:=x_{a_k} \cdots  x_{a_1}$ to denote the corresponding element in $(s^{-1} V_{\bL}^\vee)^{\otimes k}$. 
Then $\delta$ is given by
\begin{equation}\label{eqn:geomdelta1}
\delta(x) = \sum_k \sum_{v} c_v x_{a_k} \cdots x_{a_1} 
\end{equation}
where the inner sum is taken over all $v= X_{a_1} \otimes \cdots \otimes X_{a_k}$ such that
$ m_k (X_{a_1} ,\cdots, X_{a_k}) = \cdots + c_v X +\cdots$
with $c_v \neq 0$. 

Ignoring the dg-algebra structure, $\Omega V_{\bL}^\vee$ can be thought of as  the formal function ring of the $\Bbbk_\Lambda$-bimodule $CF(\bL,\bL)$ consisting of  {\em noncommutative} power series. Notice that constant functions are also included as $T(s^{-1} \overline{V_{\bL}^\vee}) = \Bbbk \oplus \left(\oplus_{k \geq 1} (s^{-1} \overline{V_{\bL}^\vee})^{\otimes k} \right)$. We show that the dg-structure on $\Omega V_{\bL}^\vee$ encodes the obstruction of a Maurer-Cartan deformation of the $\AI$-algebra $CF(\bL,\bL)$.

\begin{prop}\label{prop:mcalgequalbvvee}
Let $\bL$ be a graded unobstructed immersed Lagrangian such that $CF^{<0} (\bL,\bL)=0$ and every immersed generators have positive degrees.
Consider the dg-algebra $\Omega V_{\bL}^\vee$ for $V_{\bL}^\vee= CF(\bL,\bL)^\vee$.
Then the 0-th cohomology $H^0 ( \Omega V_{\bL}^\vee,\delta)$ is isomorphic to the Maurer-Cartan algebra $A_\bL$ of $\bL$.
\end{prop}

\begin{proof}
We prove this for $\bL$ with a single irreducible component only, and it can be easily generalized to other cases. By definition, the degree 0 component of $\Omega V_{\bL}^\vee$ is given by  $\overline {T} s^{-1} \left(CF^1 (\bL,\bL)\right)^\vee$,
which consists of infinite series in $x_1,\cdots, x_l$ with bounded coefficients, i.e. 
$$\sum_{i=k}^{\infty} \sum_{v \in \{1,\cdots, l\}^k}  c_v x_v \quad \mbox{with} \quad \inf\{c_v : i=1,2,\cdots,  v \in \{1,\cdots,l\}^k \} > -\infty.$$ 
This is precisely $\Lambda\{\!\{ x_1,\cdots, x_l \}\!\}$ in \eqref{eqn:MCalgL}.

Let $X_{l+1}, \cdots, X_{N'}$ be the degree 2 generators of $CF(\bL,\bL)$. Then the degree (-1) component of $\Omega V_{\bL}^\vee$ is the set of infinite series consisting of words in $x_1,\cdots, x_l, x_{l+1},\cdots, x_{N'}$ which have exactly one $x_{j}$ for $l+1 \leq j \leq N'$. On the other hand, it is obvious from the construction that $\delta(x_j)=f_j (x_1,\cdots, x_l) \in  ({\Omega}V_{\bL}^\vee)_0$ for $x_j$ with $l+1 \leq j \leq N'$ 
if and only if the following holds:
$$ m_0(1) + m_1(b) + m_2(b,b) + \cdots =  f_1 (x_1,\cdots,x_l) X_{l+1}  + \cdots + f_{N'} (x_1,\cdots, x_l) X_{N'}$$
where $b= \sum_{i=1}^l x_i X_i$, and we linearly expand the left hand side by \eqref{eqn:earconv}.

Now, a general degree 1 element  $f \in (\Omega V_{\bL}^\vee)_1$ can be written as 
$$f=\sum_{k=0}^{\infty} g_k(x_1,\cdots, x_l) x_{i_k} h_k(x_1,\cdots, x_l)  $$
for some polynomials $g,h$ and $l+1 \leq i_k \leq N'$, where the lengths of terms become greater than any fixed number after certain stage, and $T$-adic valuations of terms are bounded below.
Since $\delta(g_k) = \delta(h_k)=0$ in the above expression by degree reason, $\delta(f)$ has precisely the same description as the series \eqref{eqn:glfilhlfirst} after applying the Leibnitz rule. Hence, the image of $(\Omega V_{\bL}^\vee)_1$ under $\delta$ coincides with the closure of the two-sided ideal generated by $\delta(x_{l+1}), \cdots, \delta(x_{N'})$. Therefore the $0$-th $\delta$-cohomology computes
$$ \dfrac{ \Lambda \{\!\{ x_1, \cdots, x_l\}\!\}}{\langle \langle \delta(x_{l+1}), \cdots, \delta (x_{N'}) \rangle \rangle}= \dfrac{ \Bbbk \{\!\{ x_1, \cdots, x_l\}\!\}}{\langle \langle f_1, \cdots, f_{N'} \rangle \rangle},$$
which is exactly the Maurer-Cartan algebra of $\bL$.
\end{proof}

In general, $\Omega  V_\bL^\vee $ may carry nontrivial information in higher degree component, and hence taking the $0$-th cohomology $A_\bL$ may lose some geometric information of the Lagrangian $\bL$. For instance, if $\bL$ is simply connected, then $A_\bL$ is trivial.

\begin{defn}
We will call $\Omega V_{\bL}^\vee$ the Maurer-Cartan dga of $\bL$, and we denote it by $\mathcal{A}_\bL$.
\end{defn}

We remark that the curved version of this dga has been already considered in \cite{CHL2}, which includes the superpotential as a curvature term.
The Maurer-Cartan dga has another description using the bar construction. In general, one has the following purely algebraic statement.

\begin{prop}\label{prop:omevbvd}
For a finitely generated $\AI$-algebra $V$, $\Omega (V^\vee)$ is isomorphic to $(BV)^\vee$.
\end{prop}

\begin{proof}
Lemma \ref{lem:bartvhat1} shows that both $\Omega (V^\vee)$ and $(BV)^\vee$ have the same underlying vector space.
The remainder of the proof is comparing the algebraic operations and their signs, which is elementary.
\end{proof}

Applying this to $V_\bL=CF(\bL,\bL)$, we see that the Maurer-Cartan dga of $\bL$ can be also expressed as $(B V_\bL )^\vee$, and we will mostly use this alternative description in what follows. The main advantage of doing so is that $(BV)^\vee$ makes sense even for an infinite dimensional $V$. 
Mimicking the proof of Proposition \ref{prop:mcalgequalbvvee}, one can derive a simple formula for the differential $d$ on this dga, whose proof is left as an exercise.
\begin{lemma}\label{lem:donbvvee}
For an $A_\infty$-algebra $V$, let $X_1, \cdots, X_N$ freely span $V$ over its coefficient ring, and formally write $\tilde{b} := \sum_{i=1}^N x_i X_i$. Then $d$ on $(BV)^\vee$ is implicitly given by $m(e^{\tilde{b}}) = \sum_{i=1}^{N}  (-1)^{|x_i|'} d(x_i) X_i$.
\end{lemma}
 
\subsection{Koszul dual algebras and generalized Maurer-Cartan algebras}
For a unital algebra $A$ over a field $\Bbbk$ with an augmentation $\epsilon : A\to \Bbbk$, its Koszul dual algebra is defined by $A^{\textnormal{\textexclamdown}} =  \mathrm{Ext}^\ast (\Bbbk, \Bbbk)$ equipped with a Yoneda product. Here, the $A$-module structure on  $\Bbbk$ is induced by $\epsilon$ which splits the unit $\Bbbk  \to A$.
The main question around $A^{\textnormal{\textexclamdown}}$ is whether or not taking Koszul dual twice comes back to $A$ itself. Such algebras has been extensively studied, and generalized into various directions since it first appeared in \cite{Pri}.

Our interest lies in the $\AI$-(or dg-) version of this construction, especially on the chain level rather than the cohomology $\mathrm{Ext}$-algebras. Koszul duality in this context has already been investigated in many literatures such as \cite{LPWZ} or more categorical approach \cite{Pos}, etc. In this section, we recall the Koszul dual of $\AI$-algebras adapted to our filtered setting, and explain its relationship with the Maurer-Cartan formalism.

We begin with $(V,\varepsilon)$, a unital filtered augmented $\AI$-algebra over $\Bbbk_\Lambda$ with a valuation $\nu$, and set $\bar{V}:=\ker \varepsilon$. As before, $\Bbbk_\Lambda$ can be regarded as an $\AI$-module over $V$ via the augmentation $\epsilon$. Consider the set 
\begin{equation}\label{eqn:preaihoms}
\hom_V (\Bbbk_\Lambda,\Bbbk_\Lambda) = \{ \{ f_{k,1} \}_{k \geq 1} \colon f_{k,1}: \Bbbk_\Lambda \otimes V^{\otimes k}  \to \Bbbk_\Lambda\,\,\mbox{filtered}	 \}
\end{equation}
of pre-$\AI$ homomorphisms from $\Bbbk_{\Lambda}$ to itself. 
Recall from \ref{eqn:mathfrakm1} and \ref{eqn:mathfrakm2} that $\hom_V (\Bbbk_\Lambda,\Bbbk_\Lambda)$ is an $\AI$-dga with a differential $\mathfrak{M}_1$ and the multiplication $\mathfrak{M}_2$. 
 
\begin{defn}
For an augmented $A_\infty$-algebra $(V,\epsilon)$ over $\Bbbk_\Lambda$, 
$\hom_V (\Bbbk_\Lambda,\Bbbk_\Lambda)$ (or its $A_\infty$ quasi-isomorphism class) is called the Koszul dual of  $V$, which is a $A_\infty$-dga with  operations given by \eqref{eqn:mathfrakm1} and \eqref{eqn:mathfrakm2}. (One needs to put additional signs as in Remark \ref{rmk:signdgai} to fit into the standard dg-sign conventions.) We will denote it by $E(V)$. 
\end{defn}

We shall see below that the Koszul dual  for an $\AI$-algebra $V=CF(\bL,\bL)$ indeed coincides with the Maurer-Cartan dga of $\bL$. 
Recall from \ref{subsec:aiinlambda} that in our $T$-adic setting, $\hom_V (\Bbbk_\Lambda,\Bbbk_\Lambda)$ consists of multilinear maps that decompose into finitely many homogeneous components, each of which is bounded. 
Apart from the boundedness issue (which does not create additional difficulty here), it should be well-known. See, for e.g., \cite[Proposition 14]{EL-duality} for the same statement in unfiltered setting.

\begin{prop}\label{prop:evequalsbv}
$E(V)=\hom_V (\Bbbk_\Lambda,\Bbbk_\Lambda)$ is quasi-isomomrphic to $(BV)^\vee$. Thus it is quasi-isomorphic to $\Omega (V^\vee)$ when $V$ is finitely generated.
\end{prop}

\begin{proof}
Since $V^{\otimes k} \otimes \Bbbk_\Lambda \cong V^{\otimes k}$ and the action of unit is pre-determined by the unital property, a homogenous element $\{f_{k,1} :  \Bbbk_\Lambda \otimes V^{\otimes k}  \to \Bbbk_\Lambda \}_{k \geq 1}$ of degree $1-k+d$ in $E(V)$ is equivalent to a bounded linear map $ BV=T(s \bar{V}) \to \Bbbk_\Lambda$ of degree $d$. 
(Formally, one can write $f= \oplus_k f_{k,1}$.) Taking this model, we obtain a bijective correspondence between $\hom_V (\Bbbk_\Lambda,\Bbbk_\Lambda)$ and $(BV)^\vee$ . 

We next compare the algebraic operations.
The differential $M_1$ on $BV^\vee$ is given by
$$M_1 ( f ) (X_1 \otimes \cdots \otimes X_k) = d ( f ) (X_1 \otimes \cdots \otimes X_k) = (-1)^{|f|}   f \left( d (X_1 \otimes \cdots \otimes X_k ) \right)$$
$$=(-1)^{|f|} f\left( \sum (-1)^{|\mathbf{X}^{(1)}|'} \mathbf{X}^{(1)} \otimes m( \mathbf{X}^{(2)} ) \otimes \mathbf{X}^{(3)}\right)$$ 
which agrees with  \eqref{eqn:mathfrakm1} under our correspondence above, since the last two terms on the right hand side of \eqref{eqn:mathfrakm1} vanish for inputs from the augmentation ideal.
Note that the hidden factor $\Bbbk_\Lambda$ absorbed in $V^{\otimes k}$ has a nontrivial shifted degree in view of $\hom_V (\Bbbk_\Lambda,\Bbbk_\Lambda)$, so $|f|$ here plays a role of $|\psi|'$ in \eqref{eqn:mathfrakm1}. 
Finally, taking into account the sign change \eqref{rmk:signdiffdg} between $\AI$- and dg-conventions, the product $M_2$ on $BV^\vee$ is given by
$$ M_2 (f,  g) ( X_1 \otimes  \cdots \otimes X_k) =\sum  (-1)^{|f|} (f \otimes g) \left( \mathbf{X}^{(1)} \otimes  (\mathbf{X}^{(2)} \right) = \sum (-1)^{|f|}  f( g(\mathbf{X}^{(1)}) \mathbf{X}^{(2)} ),$$
and this matches with \eqref{eqn:mathfrakm2} under the correspondence.
\end{proof}

 Applying this to $V_\bL=CF(\mathbb{L},\mathbb{L})$ for a graded unobstructed Lagrangian $\bL$, we conclude that the Maurer-Cartan dga $\mathcal{A}_\mathbb{L}$ is nothing but (a dga model of) the Koszul dual of the $\AI$-algebra $CF(\mathbb{L},\mathbb{L})$.

Going back to classical Koszul duality theory for algebras, the important class of objects is, roughly speaking, a graded algebra $A$ over $\Bbbk$ such that $\hom_A (\Bbbk,\Bbbk)$ (graded with respect to our convention) has a trivial cohomology at every nonzero degree. In this case, $A$ is called a Koszul algebra, and its double Koszul dual gets back to $A$.
 
On the other hand, recall that the Maurer-Cartan algebra $A_\bL$ of a Lagrangian $\bL$ neglects nonzero-degree part of the cohomology of $\hom_{V_\bL} (\Bbbk_\Lambda,\Bbbk_\Lambda)$.
We speculate that the localized mirror functor which can be understood as a functor
\begin{equation}\label{eqn:locmirfunct}
 V_\bL \, \mathrm{mod} \to A_\bL \, \mathrm{mod}
\end{equation}
establishes an equivalence if and only if the Floer complex $V_\bL$ satisfies 
an analogous condition to the Koszulity of an algebra, for e.g., the cohomology of $\hom_{V_\bL} (\Bbbk_\Lambda,\Bbbk_\Lambda)$ being supported at degree 0, only. 
Here, the left hand side of \eqref{eqn:locmirfunct} can be thought of as the subcategory of the compact Fukaya category generated by $\bL$.
For a general $\bL$, one may need to consider the extended mirror functor
$$ V_\bL  \, \mathrm{mod} \to \mathcal{A}_\bL  \, \mathrm{dg}\textnormal{-}\mathrm{mod}\quad M \to \hom_{V_\bL} (\Bbbk_\Lambda, M)$$
which reduces to \eqref{eqn:locmirfunct} by regarding $\hom_{V_\bL} (\Bbbk_\Lambda, M)$ as a module over the $0$-th cohomology of $\hom_{V_\bL} (\Bbbk_\Lambda, \Bbbk_\Lambda)$.

\subsection{Dual pairs of objects in $\AI$-categories}\label{subsec:kzdualalg}

Consider two objects $L=\oplus_{i =1}^r L_i$ and $G=\oplus_{i =1}^r G_i$ in a unital (filtered) $A_\infty$-category $\mathcal{C}$ over the field $\Lambda$. Here, each $L_i$ and $G_i$ are objects of $\mathcal{C}$, and one may need to replace $\mathcal{C}$ by its additive enlargement if necessary, in order to make sense of the direct sum. Thus, morphism spaces between $L$ and $G$ as well as their endomorphisms are naturally modules over $\Bbbk_\Lambda=\oplus_i \Lambda \langle \pi_i \rangle $.

Suppose $G$ and $L$ satisfy the following conditions.
\begin{itemize}
\item $\hom_\mathcal{C} (G,L) \cong \Bbbk_{\Lambda}$ as $\hom_\mathcal{C} (G,G)$-modules.
\item $\hom^{>0}_\mathcal{C} (G,G)=0$ and $\hom^{<0}_\mathcal{C} (L,L) =0$
\item $\hom_{\mathcal{C}} (L,L)$ is finitely generated over $\Bbbk_\Lambda$ with  $\hom^0_{\mathcal{C}} (L,L) \cong \Bbbk_\Lambda$.
\end{itemize}
%
Such a pair $(G,L)$ can be thought of as a \emph{Koszul dual} pair in the  sense we explain now.

Since $\hom^{>0}_\mathcal{C} (G,G)=0$, the left $\hom_{\mathcal{C}} (G,G)$-module structure on $\hom_\mathcal{C} (G,L)$ given as
$$ n_{k |1}:=m_{k+1} : \hom_{\mathcal{C}}^{i_1} (G,G) \otimes \cdots \otimes \hom_{\mathcal{C}}^{i_k} (G,G) \otimes \hom_\mathcal{C} (G,L) \to \hom_\mathcal{C} (G,L)$$
must be trivial unless $k =1$ and $i_1 = \cdots i_k=0$, and we have $n_{1|1} (\one_{G} , P)=P$ for $P \in\hom_\mathcal{C} (G,L)$ where $\one_{G}$ is the unit on $\hom_{\mathcal{C}} (G,G)$. 
An analogous statement is true for the right $\hom_{\mathcal{C}} (L,L)$-module $\hom_\mathcal{C} (G,L)$, and in this case, the only nontrivial operation is the action of a scalar multiple of the unit in $\hom_{\mathcal{C}} (L,L)$.

\begin{lemma}\label{lem:augstrict}
Fix a generator $P$ of $\hom_\mathcal{C} (G,L)$ so that $\hom_\mathcal{C} (G,L) = \Bbbk_\Lambda \langle P \rangle$.
Define $\varepsilon=\varepsilon_L : \hom_\mathcal{C} (G,G) \to \Bbbk_\Lambda$ by
$$ m_2(Z,P)=\varepsilon (Z) \, P.$$
Then $\varepsilon $ is a strict augmentation on $\hom_{\mathcal{C}} (G,G)$. 
\end{lemma}

\begin{proof}
We show that $\{\varepsilon_k\}_{k \geq 1}: \hom_\mathcal{C} (G,G) \to \Bbbk_\Lambda$ defined by $\varepsilon_1 := \varepsilon$, $\varepsilon_{k \geq 2} \equiv 0$ gives an $\AI$-homomorphism.
From the $\hom_\mathcal{C} (G,G)$-module structure on $\hom_\mathcal{C} (G,L)$, it is easy to see that $\varepsilon (Z) =0$ unless $\deg Z=0$. Also, we have $\varepsilon (m_1(Z)) = 0$ since 
$$ \varepsilon (m_1(Z)) P= m_2 (m_1(Z), P), $$
and the right hand side equals $ m_1 (m_2(Z,P)) - m_2 ( Z, m_1(P))$,
each of which vanishes since $m_1(P)=0$.
It now suffices to check the following identity:
$$ \varepsilon (m_2(Z,Z') ) = \varepsilon (Z) \cdot \varepsilon (Z').$$
To see this, observe that
\begin{eqnarray*}
m_2( m_2(Z,Z'), P) &=&  m_2(Z,m_2(Z',P))  - m_3 (m_1(Z),Z',P) + m_3 (Z,m_1(Z'),P) \\
&=& m_2 (Z,m_2 (Z',P)) \\
 &=& \varepsilon (Z') m_2 (Z, P) = \left( \varepsilon (Z) \varepsilon (Z') \right) P.
\end{eqnarray*}
Here, $m_3 (m_1(Z),Z',P) = m_3 (Z,m_1(Z'),P)=0$ since $m_1(Z) = m_1 (Z')=0$.
\end{proof}
One can similarly define an augmentation on $\hom_\mathcal{C} (L,L)$, but it is easy to see from degree reason that the resulting augmentation is the projection to the unit component $\hom^0_{\mathcal{C}} (L,L) \cong \Bbbk_\Lambda$.

Although $\hom_\mathcal{C} (G,L)$ is trivial regarded as either a left $\hom_\mathcal{C} (G,G)$- or a right $\hom_{\mathcal{C}} (L,L)$-module, its $(\hom_\mathcal{C} (G,G), \hom_\mathcal{C} (L,L))$-bimodule structure is still quite rich.
%
Explicitly, it is induced by
the $A_\infty$-operations on $\mathcal{C}$, 
$$m_{l+k+1}: \hom_\mathcal{C} (G,G)^{\otimes l} \otimes \hom_\mathcal{C}(G,L) \otimes \hom_\mathcal{C} (L,L)^{\otimes k} \to \hom_\mathcal{C}(G,L)
$$
For notational simplicity, denote $\AI$-algebras $\hom_\mathcal{C} (L,L)$ and $\hom_\mathcal{C} (G,G)$ by $V$ and $W$, respectively. As before, we write $\bar{V}:= \hom^{>0}_{\mathcal{C}} (L,L)$ (the augmentation kernel) and $\bar{W}:=\ker \varepsilon$.
Since $\hom_\mathcal{C} (G,L) \cong \Bbbk_\Lambda$, the above map can be 
thought of as a pairing between $T (sW)$ and $T(sV)$. Here, the degree shift makes the pairing $\Z$-graded (i.e., it has degree zero). 
We can alternatively view this pairing in the following perspective. 

\subsubsection*{Koszul map}\label{subsubsec:defkmap}
One can first take direct sum over $k$ after removing the unit component of $V$  to get
\begin{equation}\label{eqn:varphiFFSS}
\tilde{\kappa}_l :   W^{\otimes l}  \otimes BV \to \Bbbk_\Lambda,
\end{equation}
where $BV=  T  (s\bar{V})$ is the bar-construction of $V$. 
By dualizing \eqref{eqn:varphiFFSS}, we obtain
\begin{equation}\label{eqn:defkzmap}
\kappa_l : W^{\otimes l} \to (BV)^\vee.
\end{equation}
\eqref{eqn:defkzmap} is well-defined due to the boundedness condition on $m_k$-operations which implies that \eqref{eqn:varphiFFSS} gives a bounded linear map for each fixed element of $W^{\otimes l}$. We will provide an explicit formula for $\kappa$ below.

Recall that $ (BV)^\vee = \hom_V (\Bbbk_\Lambda,\Bbbk_\Lambda)$ is the Koszul dual dga of $V$. The map $\kappa=\{\kappa_l\}_{l \geq 1}$ will be referred to as the \emph{Koszul map}.
The Koszul map admits the following explicit description. Let $\{X_1, \cdots, X_N\}$ freely generate $V$, and suppose $X_i \in \pi_{\beta_i} \cdot V \cdot \pi_{\alpha_i}$ . Denote their dual variables by $x_1 ,\cdots, x_N$ which can be  naturally identified as elements of $BV^\vee$, and take a formal linear combination $\tilde{b}= \sum_{i=1}^N x_i X_i$ (which should not be confused with $b$ in \eqref{eqn:m1bm2bb} consisting of degree 1 elements only). Then
\begin{equation}\label{eqn:mkappab}
 \sum_k m_{l+k+1} (Z_1,\cdots, Z_l, P, \tilde{b},\cdots,\tilde{b}) =  P \, \kappa_l (Z_1,\cdots,Z_l),
\end{equation}
where we use the following convention when pushing the formal variables to the back to obtain $\kappa_l (Z_1,\cdots,Z_l)$, i.e.,
\begin{equation}\label{eqn:expruleback}
 m_k(Z_1, \cdots, Z_l, P, x_{i_1} X_{i_1} ,\cdots, x_{i_{m}} X_{i_{m}}) = m_k (Z_1, \cdots, Z_l, P , X_{i_1},\cdots X_{i_{m}}) x_{i_{m}} \cdots x_{i_1} 
\end{equation}
for $k=l+m+1$. This makes $\kappa$ into a $\Bbbk_\Lambda$-module homomorphism. The expression $P \, \kappa_l$ looks opposite to our earlier convention \eqref{eqn:earconv} in which we put $x_i$'s before Floer generators, but it does not create any conflict with the Maurer-Cartan dga as we will see.
 
 More concretely, the left hand side of \eqref{eqn:mkappab} computes
\begin{equation}\label{eqn:coordexpkappa}
\begin{array}{lcl}
\displaystyle\sum_k \sum_{i_1,\cdots,i_k}  m(\vec{Z}, P,x_{i_1} X_{i_1}, \cdots,  x_{i_k}X_{i_k} ) &=& \displaystyle\sum_k \sum_{i_1,\cdots,i_k}   m(\vec{Z}, P,X_{i_1}, \cdots,   X_{i_k} ) x_{i_k} \cdots x_{i_1},  \\
&=& P \, \left(\displaystyle\sum_k \sum_{i_1,\cdots,i_k} c_{i_1,\cdots, i_k} x_{i_1} \cdots x_{i_k} \right)
\end{array}
\end{equation}
where $\vec{Z} = Z_1 \otimes \cdots \otimes Z_l$ and we put $m(\vec{Z}, P,X_{i_1}, \cdots,   X_{i_k} ) = P \, c_{i_1,\cdots, i_k} $ for $c_{i_1,\cdots, i_k} \in \Bbbk_\Lambda$. Therefore
$\kappa_l (Z_1,\cdots, Z_l)=  \sum_k \sum_{i_1,\cdots,i_k} c_{i_1,\cdots, i_k} x_{i_1} \cdots x_{i_k}.$

We next show that the Koszul map $\kappa:=\{\kappa_l\}_{l \geq 1}$ is an $\AI$-algebra homomorphism.
Let us begin by checking the degree of $\kappa$. If $x_{i_1} \cdots x_{i_m}$ nontrivially appears in ${\kappa}_l (Z_1, \cdots, Z_l) $, then the corresponding $m_k$-operations must be nonzero. That is, we have to look at the case when the operation
$$m_{l+m+1} (Z_1,\cdots, Z_l, P, X_{i_1},\cdots, X_{i_m} )$$ 
can have a nonzero output. Since the only possible output is $P$ itself, we see that
$$ |Z_1| + \cdots + |Z_l| + |X_{i_1}| + \cdots + |X_{i_m}| + 2 - (l + m + 1) =0$$
(since $|P|=0$).
Therefore
\begin{eqnarray*}
|x_{i_1}| + \cdots + |x_{i_m}| &=& (1-|X_{i_1}) + \cdots +(1- |X_{i_m}|) \\
&=& n + |Z_1| + \cdots + |Z_l|  + 2 - (l + m + 1) \\
&=& |Z_1| + \cdots + |Z_l|   + (1 - l),
\end{eqnarray*}
and hence the degree of $\kappa_l$ \eqref{eqn:defkzmap} is given by $1-l$.

\begin{prop}\label{prop:algebraickappa1}
$\kappa : W \to (BV)^\vee$ is an $\AI$-algebra homomorphism
\end{prop}

\begin{proof}
We write $M_1$ and $M_2$ for the operations of the $A_\infty$-dga $(BV)^\vee$, i.e., 
$$ M_1 (x) = d(x) ,\qquad M_2 ( x,y) = (-1)^{|x|} x \otimes y.$$
We have to verify the identity
$$ M_1 ( \kappa_l (Z_1,\cdots, Z_l) ) + \sum_{l_1+l_2= l}  M_2 \left( \kappa_{l_1} (Z_1, \cdots, Z_{l_1}), \kappa_{l_2} (Z_{l_1+1}, \cdots, Z_l) \right)$$
$$=  \sum_{k_1 + k_2 =l+1} (-1)^{|Z_1|' + \cdots + |Z_j|'} \kappa_{k_1} ( Z_1, \cdots, Z_j, m_{k_2} (Z_{j+1}, \cdots, Z_{j + k_2}),Z _{j+k_2+1} \cdots, Z_l).$$
Let us consider the $\AI$-relations for the tuple $(Z_1,\cdots, Z_l, P, e^{\tilde{b}})$ where 
$$e^{\tilde{b}} = 1 + \tilde{b} + \tilde{b} \otimes \tilde{b} + \tilde{b} \otimes \tilde{b} \otimes \tilde{b} + \cdots,$$
and  $\tilde{b}^{\otimes i}$ is inserted to $\AI$-operations as $m_k (-, \cdots, -, \overbrace{\tilde{b}, \cdots, \tilde{b}}^{i})$. The relation is equivalent to 
\begin{equation*}
\begin{array}{c}
 \displaystyle\sum (-1)^{|\vec{Z}|'} m(\vec{Z}, P , e^{\tilde{b}} , m(e^{\tilde{b}}), e^{\tilde{b}})+ \sum (-1)^{|\vec{Z}^{(1)}|'-1} m (\vec{Z}^{(1)}, m(\vec{Z}^{(2)}, P ,e^{\tilde{b}}),e^{\tilde{b}}) \\
= \displaystyle\sum (-1)^{|\vec{Z}^{(1)}|'} m ( \vec{Z}^{(1)}, m(\vec{Z}^{(2)} ), \vec{Z}^{(3)}, P , e^{\tilde{b}}),
\end{array}
\end{equation*}
where we omit the subscript $k$ in $m_k$ for simplicity.
Here, $|\vec{Z}|'$ is the sum of the shifted degrees of factors in $\vec{Z}$.
The right hand side equals
$$ \sum (-1)^{|\vec{Z}^{(1)}|'} m ( \vec{Z}^{(1)}, m(\vec{Z}^{(2)} ), \vec{Z}^{(3)}, P , e^{\tilde{b}})=  P  \sum (-1)^{|\vec{Z}^{(1)}|'} \kappa ( \vec{Z}^{(1)}, m(\vec{Z}^{(2)} ), \vec{Z}^{(3)} ),$$
and the second term on the left hand side equals
\begin{eqnarray*}
\sum (-1)^{|\vec{Z}^{(1)}|'-1} m (\vec{Z}^{(1)}, m(\vec{Z}^{(2)}, P ,e^{\tilde{b}}),e^{\tilde{b}}) &=& \sum (-1)^{|\vec{Z}^{(1)}|'-1} m (\vec{Z}^{(1)}, P \,\kappa(\vec{Z}^{(2)}) ,e^{\tilde{b}})\\
&=& \sum (-1)^{|\vec{Z}^{(1)}|'-1} m (\vec{Z}^{(1)}, P,e^{\tilde{b}}) \, \kappa(\vec{Z}^{(2)})  \\
&=&\sum  (-1)^{| \kappa(\vec{Z}^{(1)})|} P\, \kappa(\vec{Z}^{(1)}) \cdot \kappa (\vec{Z}^{(2)}) \\
&=& P  \sum   M_2( \kappa(\vec{Z}^{(1)}), \kappa(\vec{Z}^{(2)}) )
\end{eqnarray*} 
since $\kappa_l$ has degree $1-l$.
Lastly, applying $m(e^{\tilde{b}}) = \sum_{i=1}^{N} (-1)^{|x_i|'} d(x_i) X_i$ (see Lemma \ref{lem:donbvvee}), the first term on the left hand side becomes
\begin{equation*}
\begin{array}{l}
(-1)^{|\vec{Z}|'}  m(\vec{Z}, P , e^{\tilde{b}} , m(e^{\tilde{b}}), e^{\tilde{b}}) \\
 = \displaystyle\sum_k \sum_{i_1,\cdots,i_k} \sum_{1\leq a \leq k } (-1)^{|\vec{Z}|' + |x_{i_a}|'}  m(\vec{Z}, P,x_{i_1} X_{i_1}, \cdots, x_{i_{a-1}} X_{i_{a-1}}, d(x_{i_a})X_{i_a}, x_{i_{a+1}} X_{i_{a+1}},\cdots, x_{i_k} X_{i_k} ) \\
=  \displaystyle\sum_k \sum_{i_1,\cdots,i_k} \sum_{1\leq l \leq k } (-1)^{|\vec{Z}|'  +|x_{i_a}|'+\ast } m(\vec{Z}, P,  X_{i_1}, \cdots,   X_{i_a}, \cdots,  X_{i_k} ) x_{i_k} \cdots x_{i_{a+1}} d (x_{i_a}) x_{i_{a-1}} \cdots x_{i_1},
\end{array}
\end{equation*}
where $\ast = |x_{i_1}| + \cdots +|x_{i_{a-1}}|$ comes from the Koszul convention since $|d(x_{i_a}) X_{i_a}|' = -1$. Observe that
\begin{eqnarray*}
 |\vec{Z}|' +|x_{i_a}|'+ \ast &=& (  - |X_{i_1}|' - \cdots - |X_{i_k}|' -1 ) +|x_{i_a}|'+ \ast \\
 &=&( |x_{i_1}| + \cdots |x_{i_k}| -1) + |x_{i_a}|' + (|x_{i_1}| + \cdots +|x_{i_{a-1}}|) \\
&\equiv&  |x_{i_{a+1}}| + \cdots + |x_{i_k}| \mod 2
\end{eqnarray*}
Comparing with \eqref{eqn:coordexpkappa}, we see that
$$m(\vec{Z}, P , e^{\tilde{b}} , m(e^{\tilde{b}}), e^{\tilde{b}}) = P \, d ( \kappa (\vec{Z}) ) =P \, M_1 ( \kappa (\vec{Z}) ), $$
which completes the proof.

\end{proof}

We will mainly focus on the induced map on the $0$-th cohomology of 
$\kappa : \hom_\mathcal{C} (G,G) \to \mathcal{A}_L=(B \hom_\mathcal{C} (L,L))^\vee$ in our geometric situation below.

\section{Dual pairs in wrapped Fukaya categories}\label{sec:koszulpairmain}
We apply the homological algebra tool developed so far to the Fukaya $\AI$-algebra, and study the Koszul duality between $\AI$-algebras from two Lagrangians in some special geometric relation. Let $(M,\omega=d \Theta)$ be a Liouville manifold, and consider an exact compact Lagrangian $L=\oplus_{i=1}^r L_i$ which, together with suitable Floer data, gives an object in $\mathcal{W}Fuk(M)$. Since $L$ consists of $r$ irreducible components, we are naturally to work over the semi-simple ring $\Bbbk_\C$ or $\Bbbk_\Lambda$. 
Analogously to the algebraic setting in \ref{subsec:kzdualalg}, suppose that there exists another Lagrangian $G=\oplus_{i=1}^r G_i$ in $\mathcal{W}Fuk(M)$ so that the pair $(G,L)$ admits the following properties.

\begin{assumption}\label{assume:kospairinw}
$G$ and $L$ satisfy
\begin{itemize}
\item $\hom_{\mathcal{W}Fuk(M)} (G,L) \cong \Bbbk_{\mathbb{C}}$ as $\hom_{\mathcal{W}Fuk(M)} (G,G)$-modules;
\item $\hom^{<0}_{\mathcal{W}Fuk(M)} (L,L) =0$ and $\hom^0_{\mathcal{W}Fuk(M)} (L,L) \cong \Bbbk_\C$;
\item $\hom^{i >0}_{\mathcal{W}Fuk(M)} (G,G)=0$, and $G$ generates $L$ in $\mathcal{W}Fuk(M)$.
\end{itemize} 
\end{assumption}
We work with this assumption throughout the section.
The first condition implies that (after rearranging indinces suitably) $L_i$ intersects $G_i$ exactly at one point, say $P_i$, with degree 0, and $L_i \cap G_j =\emptyset$ for $j \neq i$. Then $\hom_{\mathcal{W}Fuk(M)} (G,L) $ is spanned by $P=\sum P_i$.
We also remark that the second condition only constraints  the degree of immersed generators of $L$. 
Note, in particular, that $G$ and $L$ form a Koszul dual pair in $\mathcal{W}Fuk(M)$ in the sense of \ref{subsec:kzdualalg}. Our goal is to compare the wrapped Floer cohomology of $G$ and the Maurer-Cartan algebra of $L$ using the Koszul map $\kappa$ algebraically defined in \ref{subsec:kzdualalg}.  For notational simplicity, we write $\mathcal{W}^\C$ and $\mathcal{W}^\Lambda$ for respectively $\mathcal{W}Fuk (M)$  and  $\mathcal{W}Fuk^\Lambda (M)$ from now on.

Our setting generalizes the geometric one in \cite[Introduction]{EL-koszul} in the sense that an individual component $L_i$ can be non-simply-connected. This makes each homogenous piece of its bar construction infinite dimensional, and it is natural to work over $\Lambda$ for this reason. Then it is crucial to introduce a certain completion of $\hom_{\mathcal{W}^\Lambda} (G,G)=CW (G,G;\Lambda)$ (related with the $T$-adic topology) to compare with the Maurer-Cartan algebra of $L$.
We will also examine in \ref{subsec:locsyztfibers} the case of non-exact tori $L$ sitting as fibers of a Lagrangian torus fibation, which will be useful in applications to local SYZ examples later.

\subsection{The double Koszul dual and the completed cobar construction $\hat{\Omega} B$.}  
Let us first compare Floer theory of $G$ and $L$ satisfying Assumption \ref{assume:kospairinw} over $\C$-coefficient. 
Following \cite[Introduction]{EL-koszul}, we begin with the Yoneda map for $G$
\begin{equation}\label{eqn:yonedccel}
\hom_{\mathcal{W}^\C} (L,L) \to \hom_{\hom_{\mathcal{W}^\C} (G,G)} (\hom_{\mathcal{W}^\C} (G,L), \hom_{\mathcal{W}^\C} (G,L)) = \hom_{\hom_{\mathcal{W}^\C} (G,G)} ( \Bbbk_\C, \Bbbk_\C),
\end{equation}
which is a quasi-isomorphism since $G$ generates $L$. We remark that $\Bbbk_\C$ above is taken to be a left module over $\hom_{\mathcal{W}^\C} (G,G)$. 
\cite[Proposition 14]{EL-duality} (or Proposition \ref{prop:evequalsbv} adapted to $\C$-coefficient setting) shows that 
$$\hom_{\hom_{\mathcal{W}^\C} (G,G)} ( \Bbbk_\C, \Bbbk_\C)  \cong \left. ^\sharp \! \! \right.   \left(B \hom_{\mathcal{W}^\C}  (G,G) \right)$$
where $\left. ^\sharp \! \! \right.(-)$ is the graded dual (over $\Bbbk_\C$), that is, the direct sum of the sets of left $\Bbbk_\C$-module maps from individual graded pieces to $\Bbbk_\C$. The bar construction $B$ in this case is the one for $\Bbbk_\C$-module, taken with respect to the (strict) augmentation $\varepsilon_{\C} : \hom_{\mathcal{W}^\C} (G,G) \to \Bbbk_\C$ given by $m_2(\widetilde{Z},\widetilde{P})=\varepsilon_{\C} (\widetilde{Z}) \widetilde{P}$, which is related with $\varepsilon$ given in Lemma \ref{lem:augstrict} via base change. Notice that we are using the notation for exact generators here (see  Definition \ref{def:exactgens}). On $\mathcal{W}^\Lambda$, the associated  geometric generator $Z$ and $P$ differ from $\widetilde{Z}$ and $\widetilde{P}$ by $T$-th power determined by their actions. Note also that the augmentation $\varepsilon_{\C}$ depends significantly on $L$.

For simplicity, let us write $V_\C$ and $W_\C$ for $\hom_{\mathcal{W}^\C} (L,L)$ and $\hom_{\mathcal{W}^\C} (G,G)$ respectively. Likewise, we will use the notations $V_\Lambda:=\hom_{\mathcal{W}^{\Lambda}} (L,L)$ and $W_\Lambda := \hom_{\mathcal{W}^{\Lambda}} (G,G)$.
Taking the bar construction on the $\AI$ quasi-isomorphism $V_\C \to   \left. ^\sharp \! \! \right. (B W_\C)$ in \eqref{eqn:yonedccel}, we have a dg-coalgebra quasi-isomorphim $ BV_\C \stackrel{\simeq}{\longrightarrow} B \left( \left. ^\sharp \! \! \right. (BW_\C) \right)$, or dually, 
\begin{equation}\label{eqn:bbcoverc}
(B   \left. ^\sharp \! \! \right.  (BW_\C) )^\sharp \stackrel{\simeq}{\longrightarrow} (BV_\C)^\sharp
\end{equation}
since $\mathrm{Ext}$ group vanishes over a divisible group.
The $0$-th cohomology $H^0 (BV_\C^\sharp) $ of the right hand side of \eqref{eqn:bbcoverc} can be thought of as the ``formal" Maurer-Cartan algebra of $L$ in contrast to the actual Maurer-Cartan algebra $A_L =
H^0 (BV_\Lambda^\vee)$ which involves the continuous dual.
On the other hand, $(B   \left. ^\sharp \! \! \right.  (BW_\C) )^\sharp$ is the double Koszul dual, where the Koszul dual in this context uses the graded dual ``$\sharp$" instead of the continuous dual ``$\vee$". 
 
We lift the above quasi-isomorphism to $\Lambda$-coefficients, and unravel the double Koszul dual a little further. For a graded left $\Bbbk_\Lambda$ module $M$, we denote by $M^{\sharp_\Lambda}$ the graded dual of $M$ in $\Lambda$-coefficient, that is, its degree $d$ piece is the set $\hom_{\Bbbk_{\Lambda}} (M_{-d} , \Bbbk_{\Lambda})$ of right $\Bbbk_{\Lambda}$ module homomorphisms. $  \left. ^{\sharp_\Lambda} \! \! \right.  M$ is similarly defined.
Applying $(-) \otimes_\C \Lambda$ (or $(-) \otimes_{\Bbbk_\C} \Bbbk_\Lambda)$)  to $ BV_\C \stackrel{\simeq}{\longrightarrow} B \left( \left. ^\sharp \! \! \right. (BW_\C) \right)$,
$$ BV_{\Lambda}  \stackrel{\cong}{\longrightarrow}  B \left( \left. ^\sharp \! \! \right. (BW_\C) \right) \otimes \Lambda,$$
which leads to
$$ \bigoplus_d \hom_{\Bbbk_{\Lambda}} \left( \left( B \left( \left. ^\sharp \! \! \right. (BW_\C) \right) \right)_d \otimes \Bbbk_\Lambda, \Bbbk_\Lambda \right)   \stackrel{\cong}{\longrightarrow}  (BV_{\Lambda})^{\sharp_\Lambda}$$
by universal coefficient theorem, and hence
\begin{equation}\label{eqn:formallambda}
 \bigoplus_d \hom_{\Bbbk_{\C}} \left( \left( B \left( \left. ^\sharp \! \! \right. (BW_\C) \right) \right)_d, \Bbbk_\Lambda \right)   \stackrel{\cong}{\longrightarrow}  (BV_{\Lambda})^{\sharp_\Lambda}.
\end{equation}
The left hand side is the double Koszul dual over $\Lambda$, but still formal in the sense that it ignores the $T$-adic convergence.
We first show that it is quasi-isomorphic to the completed cobar construction $\hat{\Omega} B W_{\Lambda}$ of $BW_\Lambda$ mentioned in Remark \ref{rmk:threecobar}.

\begin{prop}\label{prop:unfiltereddercompl}
$(BV_{\Lambda})^{\sharp_\Lambda}$ is quasi-isomorphic to $\hat{\Omega} B W_{\Lambda}$ (where $V_\Lambda=\hom_{\mathcal{W}^{\Lambda}} (L,L)$ and $W_\Lambda = \hom_{\mathcal{W}^{\Lambda}} (G,G)$).
\end{prop}

\begin{proof}
We show that the left hand side of \eqref{eqn:formallambda} is quasi-isomorphic to $\hat{\Omega} B W_{\Lambda}$.
For simplicity, let us write $\mathbb{W}:=B \left. ^\sharp \! \! \right.  (BW_\C) $ in the rest of the proof. 
$\mathbb{W}$ admits a filtration induced by the number of tensor powers from the latter bar construction.
Namely, we have
\begin{equation}\label{eqn:filtnonmathfrakw}
\mathcal{F}^0 \mathbb{W} \subset \mathcal{F}^1 \mathbb{W} \subset \mathcal{F}^2 \mathbb{W} \subset \cdots
\end{equation}
where $\mathcal{F}^l \mathbb{W}:= \oplus^{l}_{i=0}\left( s  \left. ^\sharp \! \! \right.  (BW_\C) \right)^{\boxtimes i}$, and each $\mathcal{F}^l \mathbb{W}$ is preserved by the differential as it decreases the length. Notice that we use $\boxtimes$ for the tensor product appearing in the latter bar construction $B$ in order to avoid confusion.
The filtration satisfies $\cup_{i=1}^{\infty}  \mathcal{F}^i \mathbb{W} = \mathbb{W}$, although it is not bounded. 

$\mathbb{W}$ can also be understood as a double complex described in the diagram below (in the diagram, $W$ denotes $W_\C$ for simplicity), where the filtration $\mathcal{F}^l$ is simply taking its first $l$ rows.
 
\begin{equation}\label{eqn:doublecpxmathfrakw}
\xymatrix{  \left. ^\sharp \! \! \right. \bar{W}_0 \ar[r]&  \left. ^\sharp \! \! \right. \bar{W}_{-1} \oplus  \left. ^\sharp \! \! \right. (\bar{W}_0^{\otimes 2}) \ar[r]  &  *\txt{$ \left. ^\sharp \! \! \right. \bar{W}_{-2} \oplus  \left. ^\sharp \! \! \right. (\bar{W}_{-1} \otimes   \bar{W}_0 )$ \\ $\oplus \left. ^\sharp \! \! \right.  ( \bar{W}_0 \otimes   \bar{W}_{-1} ) \oplus  \left. ^\sharp \! \! \right.( \bar{W}_0^{ \otimes 3})$}  \ar[r]   & \cdots\\
 & \left. ^\sharp \! \! \right. \bar{W}_0 \boxtimes  \left. ^\sharp \! \! \right. \bar{W}_0 \ar[r] \ar[u]& *\txt{$  \left. ^\sharp \! \! \right.  \bar{W}_{-1} \boxtimes   \left. ^\sharp \! \! \right.   \bar{W}_0   \oplus  \left. ^\sharp \! \! \right.\bar{W}_0 \boxtimes \left. ^\sharp \! \! \right.   \bar{W}_{-1}$\\$  \oplus \left. ^\sharp \! \! \right.    \bar{W}_0 \boxtimes \left. ^\sharp \! \! \right.  (  \bar{W}_0^{\otimes 2}   ) \oplus  \left. ^\sharp \! \! \right.  (\bar{W}_0^{\otimes 2} ) \boxtimes \left. ^\sharp \! \! \right.   \bar{W}_0 $} \ar[r] \ar[u] 
  & \cdots \\
 & & \left. ^\sharp \! \! \right. \bar{W}_0 \boxtimes  \left. ^\sharp \! \! \right. \bar{W}_0 \boxtimes  \left. ^\sharp \! \! \right. \bar{W}_0 \ar[r]  \ar[u]& \cdots  \\
}
\end{equation}
The horizontal differential is induced by that on $BW_\C$ (extended to the tensor product $\boxtimes$ by the Leibnitz rule), and the vertical differential is dual to $a \otimes b \mapsto a \boxtimes b$.

Now, consider the following graded dual $\mathbb{W}^{\sharp_{\Lambda}}$ of $\mathbb{W}$ in $\Lambda$-coefficients:
\begin{equation*}
\left( \mathbb{W}^{\sharp_{\Lambda}}  \right)_d= \hom_{\C} ( \mathbb{W}_d ,\Lambda) 
\end{equation*}
By the universal coefficient theorem, $ \mathbb{W}^{\sharp_{\Lambda}}$ is quasi-isomorphic to the graded dual $(BV_\C)^{\sharp_{\Lambda}}$ of $BV_\C$ in $\Lambda$-coefficient. 
We then find what the cohomology of $\mathbb{W}^{\sharp_{\Lambda}}$ computes.
Observe that $\mathbb{W}^{\sharp_{\Lambda}}$ can be obtained as the inverse limit (completion) of $ \left\{ \left( \mathcal{F}^l \mathbb{W} \right)^{\sharp_{\Lambda}} \right\}_{l \geq 1}$  with respect to the natural restriction map
$$ \iota_{l' l}: \left( \mathcal{F}^{l'}
 \mathbb{W} \right)^{\sharp_{\Lambda}} \to \left( \mathcal{F}^l \mathbb{W} \right)^{\sharp_{\Lambda}}$$ 
 for $l' >l$,
where  $\left( \mathcal{F}^l \mathbb{W} \right)^{\sharp_{\Lambda}}$ is the graded dual of $\mathcal{F}^l \mathbb{W}$ in $\Lambda$-coefficient. It is not difficult to see that the completion is isomorphic $\mathbb{W}^{\sharp_\Lambda}$ on the chain level.

On the other hand, recall $\underline{\Omega} B W_{\Lambda} = \underline{\Omega} B W_{\C} 
\otimes \Lambda $ (see Remark \ref{rmk:threecobar}) also admits a filtration from the tensor powers associated with the cobar construction, given as
$$  \mathcal{F}_l \underline{\Omega} B W_{\Lambda} := \oplus_{i=l+1}^{\infty} \left( s^{-1} BW_{\Lambda} \right)^{\boxtimes i}$$
so that 
$$ \mathcal{F}_0 \underline{\Omega} B W_{\Lambda} \supset \mathcal{F}_1 \underline{\Omega} B W_{\Lambda} \supset \mathcal{F}_2 \underline{\Omega} B W_{\Lambda} \supset   \cdots.$$
Here we use $\boxtimes$ to denote the tensor product associated with $\underline{\Omega}$.
It can also be written as a double complex below, where the filtration $\mathcal{F}_l$ takes rows above the $l$-th:
\begin{equation}\label{eqn:doubleomegab}
\xymatrix{ \cdots \ar[r]&\bar{W}_0 \boxtimes \bar{W}_0 \boxtimes \bar{W}_0& &	  \\
\cdots \ar[r]& *\txt{$ (   \bar{W}_{-1} \boxtimes   \bar{W}_0 ) \oplus (  \bar{W}_0 \boxtimes   \bar{W}_{-1} )$\\$  \oplus (  \bar{W}_0 \boxtimes  \bar{W}_0^{\otimes 2}   ) \oplus (  \bar{W}_0^{\otimes 2} \boxtimes   \bar{W}_0 ) $} \ar[r] \ar[u]&\bar{W}_0 \boxtimes \bar{W}_0&  \\
\cdots \ar[r]&*\txt{$  \bar{W}_{-2} \oplus  ( \bar{W}_{-1} \otimes   \bar{W}_0 )$ \\ $\oplus (  \bar{W}_0 \otimes    \bar{W}_{-1} ) \oplus  \bar{W}_0^{ \otimes 3}$} \ar[r] \ar[u]& \bar{W}_{-1} \oplus \bar{W}_0^{\otimes 2} \ar[r] \ar[u]&  \bar{W}_0 \\}.
\end{equation}
As in Remark \ref{rmk:threecobar}, $\hat{\Omega} B W_{\Lambda}$ is the inverse system formed by $ I_l (\underline{\Omega} B W_{\Lambda}) : =  \dfrac{\underline{\Omega} B W_{\Lambda} }{ \mathcal{F}_l \underline{\Omega} B W_{\Lambda} }$.

Observe that there exists a natural map 
\begin{equation}\label{eqn:flflflfl}
f_l :  I_l (\underline{\Omega} B W_{\Lambda})  \to \left( \mathcal{F}^l \mathbb{W} \right)^{\sharp_{\Lambda}}
\end{equation}
which is compatible with the restriction maps $\iota_{l' l }$. This is induced by the evaluation of a tensor product of maps in 
$ \left. ^\sharp \! \! \right. \left( \otimes_{i=1}^{a_1}  (\bar{W}_\C)_{d_{i,1}} \right) \boxtimes \left. ^\sharp \! \! \right. \left( \otimes_{i=1}^{a_2}  (\bar{W}_\C)_{d_{i,2}} \right) \boxtimes \cdots \boxtimes  \left. ^\sharp \! \! \right. \left( \otimes_{i=1}^{a_l} (\bar{W}_\C)_{d_{i,l}} \right)$
(belonging to $\mathbb{W}_d$ for $d=\sum_{i,j} (1-d_{i,j}) -k$)
at an element in $ \left( \otimes_{i=1}^{a_1}  (\bar{W}_\C)_{d_{i,1}} \right) \boxtimes \left( \otimes_{i=1}^{a_2}  (\bar{W}_\C)_{d_{i,2}} \right) \boxtimes \cdots \boxtimes \left( \otimes_{i=1}^{a_l}   (\bar{W}_\C)_{d_{i,l}} \right) \subset \underline{\Omega} B W_\C$ extended to $\underline{\Omega} B W_{\Lambda}$ by linearity.
This is well-defined since the evaluation of a map in $\mathcal{F}^l \mathbb{W}$ at any element of $\mathcal{F}_{l} \underline{\Omega} B W_{\Lambda}$ vanishes by definition. 

It is straightforward to check that $f_l$ is a chain map.
We claim that $f_l$ \eqref{eqn:flflflfl} is a quasi-isomorphism for each $l$. To see this, recall that $\!\left. ^\sharp \! \! \right.  (BW_\C)$ is quasi-isomorphic to $V_\C$, and hence it has a finite dimensional cohomology. By the universal coefficient theorem, $BW_\C$ should also have a finite dimensional cohomology.
Let us first look into the cohomology of $ \mathcal{F}^l \mathbb{W}$ itself. 

$ \mathcal{F}^l \mathbb{W}$ admits a bounded decreasing filtration which restricts the original one on $\mathbb{W}$ \eqref{eqn:filtnonmathfrakw}
$$  \mathcal{F}^l \mathbb{W} \supset \mathcal{F}^{l-1} \mathbb{W} \supset \cdots \supset \mathcal{F}^0 \mathbb{W}$$
The $E_1$-page of the associated spectral sequence is nothing but the horizontal cohomology of the double complex \eqref{eqn:doublecpxmathfrakw}. Therefore, $E_1$-page consists of the tensor products 
$$ \left. ^\sharp \! \! \right. H^{i_1} (BW_\C) \boxtimes  \left. ^\sharp \! \! \right. H^{i_2} (BW_\C) \boxtimes \cdots \boxtimes  \left. ^\sharp \! \! \right. H^{i_k} (BW_\C) = \left. ^\sharp \! \! \right. \left( H^{i_1} (BW_\C) \boxtimes   H^{i_2} (BW_\C) \boxtimes \cdots \boxtimes  H^{i_k} (BW_\C) \right)$$ 
with $k \leq l$ where the equality holds since each $H^{\ast} (BW_\C)$ is finite dimensional. 
 Similarly, $ I_l (\underline{\Omega} B W_{\Lambda}) $ can be thought of as the one obtained by erasing rows above $l$-th row in \eqref{eqn:doubleomegab} (or replacing them by $0$). It also admits a bounded filtration whose associated spectral sequence computes the horizontal cohomology at $E_1$-page. Thus $E_1$ consists of 
\begin{equation*}
\begin{array}{ll}
 H^{i_1} (BW_{\Lambda}) \boxtimes   H^{i_2} (BW_\Lambda) \boxtimes \cdots \boxtimes  H^{i_k} (BW_{\Lambda}) \\
 = \hom_\C ( \left. ^\sharp \! \! \right. \left( H^{i_1} (BW_\C) \boxtimes   H^{i_2} (BW_\C) \boxtimes \cdots \boxtimes  H^{i_k} (BW_\C) \right),\Lambda ),
 \end{array}
 \end{equation*}
where the equality holds since $H^\ast (BW_\C)$ is finite dimensional. 

We conclude that the filtrations on $\mathcal{F}^l \mathbb{W}$ and $I_l (\underline{\Omega} B W_{\Lambda})$ both admit spectral sequences which become finite dimensional from $E_1$-pages. Moreover, they are dual to each other (in the sense of taking ($\hom_\C(- ,\Lambda)$) in each page. Since both of filtrations are bounded, their associated spectral sequence converges, and hence $H^\ast I_l (\underline{\Omega} B W_{\Lambda})$ is isomorphic to $\hom_\C( H^\ast (\mathcal{F}^l \mathbb{W}), \Lambda)$. The claim follows from this by the universal coefficient theorem.

\end{proof}

 When $W$ is a dga with $H^\ast(BW)$ locally finite, \cite[Proposition 4.1.9, Remark 4.1.11]{Booth} proves that the double Koszul dual of $W$ is the completed cobar construction $\hat{\Omega} B W$ which was understood as the ``derived" completion of $W$.
For an augmented algebra, it is shown in \cite{Seg} that the double Koszul dual gives rise to the completion of the algebra with respect to the augmentation kernel.
 
\subsection{Formal neighborhood of $L$}\label{subsec:formalnbdhat}
We are interested in what \eqref{eqn:formallambda} induces on the $0$-th cohomology.
$H^0 (BV_{\Lambda}^{\sharp_{\Lambda}})$ can be understood as a formal noncommutative power series analogue of $A_L$ in Definition \ref{def:mcalgell}. Using the notation in \ref{subsec:MCeqns}, $H^0 (BV_{\Lambda}^{\sharp_{\Lambda}})$ consists of series 
$\sum_{l=0}^\infty \sum_{I \in (\Z_{>0})^l} \lambda_I  x^I$ modulo elements of the form
$\sum_{k=1}^{N'} \sum_{l=0}^{\infty} \sum_{|I | + |J| = l} \lambda_{k, I, J} \,\, x^{I} \, f_k \, x^{J}$ 
where there are no restrictions for the sets of coefficients $\{ \Lambda_I \}$ and $\{ \lambda_{k, I , J}   \}$. We will write $\hat{A}_L:=H^0 (BV_{\Lambda}^{\sharp_{\Lambda}})$. It is not difficult to see the natural map $A_L (=H^0 (BV_{\Lambda}^\vee)) \to \hat{A}_L$ induced from $BV_{\Lambda}^\vee \to BV_{\Lambda}^\sharp$ is an inclusion, and
$A_L$ can be intuitively interpreted as the subspace of $\hat{A}_L$ formed by \emph{convergent} power series.


\begin{prop}\label{cor:formalversionk}
In the setting as above, $\hat{A}_L$ is isomorphic to the completion $\widehat{H^0 (W_\Lambda)}$ of $H^0 (W_\Lambda) $ with respect to the maximal ideal $H^0 (\bar{W}_\Lambda)$ for $W_\Lambda = \hom_{\mathcal{W}^{\Lambda}} (G,G)$ and $\bar{W}_\Lambda = \ker  \epsilon_L$.
\end{prop}

Note that the augmentation $\epsilon_L$ induces an algebra homomorphism $H^0 (W_\Lambda) \to \Bbbk_\Lambda$, and its kernel is given by $H^0 (\bar{W}_\Lambda)$. In particular, $H^0 (\bar{W}_\Lambda)$ is a maximal ideal of $H^0 (W_\Lambda)$.

\begin{proof}
From the discussion so far, the question boils down to identify the $0$-th cohomology of the completion of $\left\{  I_l (\underline{\Omega} B W_{\Lambda}) \right\}_{l \geq 1}$.
For this purpose, let us look into the cohomology of $I_l (\underline{\Omega} B W_{\Lambda})$. Recall that it can be understood as the double complex described in \eqref{eqn:doubleomegab}. Let us consider the spectral sequence for $I_l (\underline{\Omega} B W_{\Lambda})$ viewed as the double complex given in \eqref{eqn:doubleomegab}, but we begin with the vertical differential this time which is simply switching $\otimes$ to $\boxtimes$. Indeed, this is exact except on  the first low, and on the $l$-th row since we cut-off the $l+1$-th row and above. To see this, one can use a homotopy $r$ on $\underline{\Omega} B W_{\Lambda}$
$$r:  \bar{W}_{\Lambda}^{\otimes i_1} \boxtimes  \bar{W}_{\Lambda}^{\otimes i_2} \boxtimes \cdots \otimes  \bar{W}_{\Lambda}^{\otimes i_k} \to   \bar{W}_{\Lambda}^{\otimes i_1 + i_2} \boxtimes \cdots \boxtimes  \bar{W}_{\Lambda}^{\otimes i_k}$$
which is defined to be zero for $i_1 \neq 1$ and equal to the natural isomorphism
$$ \bar{W}_{\Lambda}  \boxtimes  \bar{W}_{\Lambda}^{\otimes i_2} \boxtimes \cdots \boxtimes  \bar{W}_{\Lambda}^{\otimes i_k} \to   \bar{W}_{\Lambda}^{\otimes 1 + i_2} \boxtimes \cdots \boxtimes   \bar{W}_{\Lambda}^{\otimes i_k}.$$ 

Due to our assumption on the grading on $W$, the only possible degree $0$ piece on the $E_2$-page is $H^0 (\bar{W}_{\Lambda})$, where horizontal differential takes the ($0$-th) cohomology of $\bar{W}_\Lambda$ in this degree. 
The image of the differential in $H^0 (\bar{W}_{\Lambda})$ on this page consists of elements of the form $a_1 \cdot a_2 \cdots a_l$ where ``$\,\, \cdot \,\, $" denotes the \emph{associative} product on $H^0 (\bar{W}_{\Lambda})$ induced by $m_2$ where $a_i \in H^0 (\bar{W}_{\Lambda})$. Therefore, it stabilizes to the quotient $H^0 (\bar{W}_{\Lambda}) / \prod_{i=1}^l H^0 (\bar{W}_{\Lambda})$ afterwards. The diagram below shows the computation of the differential on the $E_2$-page when $l=3$.

{\footnotesize
\begin{equation*}
\xymatrix{\cdots \ar[r]& 
*\txt{$  \left( H^0 (\bar{W}_\Lambda) \otimes   H^0 (\bar{W}_\Lambda) \boxtimes   H^0 (\bar{W}_\Lambda) \boxtimes H^0 (\bar{W}_\Lambda)\right) \oplus$ \\ $ \left( H^0 (\bar{W}_\Lambda) \boxtimes   H^0 (\bar{W}_\Lambda) \otimes   H^0 (\bar{W}_\Lambda) \boxtimes H^0 (\bar{W}_\Lambda)\right) \oplus $ \\  $\left( H^0 (\bar{W}_\Lambda) \boxtimes   H^0 (\bar{W}_\Lambda) \boxtimes   H^0 (\bar{W}_\Lambda) \otimes H^0 (\bar{W}_\Lambda)\right) $} \ar[r]. 
& 0 & &	  \\ \cdots \ar[r]&
0\ar[r] \ar[u]&  0 \ar[r] \ar[u]& 0 &  \\ \cdots \ar[r]& H^{-3} (\bar{W}_\Lambda)  \ar[r] \ar[u]& H^{-2} (\bar{W}_\Lambda)  \ar[r] \ar[u]& H^{-1} (\bar{W}_\Lambda)  \ar[r] \ar[u]&  H^{0} (\bar{W}_\Lambda)  \\}.
\end{equation*}
}

 \begin{center}
 \begin{tikzpicture}
\matrix (mat) [table]
{
 $0$ & $0$  & $0$ & $0$  &  $0$    \\
$a_1 \boxtimes  a_2 \otimes a_3 \boxtimes a_4$ & $a_1 \boxtimes  a_2 \cdot a_3 \boxtimes a_4$  &  $0$ & $0$ &  $0$ \\
  & $a_1 \otimes a_2 \cdot a_3 \boxtimes a_4 $ &  $a_1 \cdot a_2 \cdot a_3 \boxtimes a_4$ & $0$ & $0$  \\
  & & $a_1 \cdot a_2 \cdot a_3 \otimes a_4$ & $a_1 \cdot a_2 \cdot a_3 \cdot a_4 $ & $0$    \\
};
\foreach \x in {1,...,4}
{
  \draw 
    ([xshift=-.5\pgflinewidth]mat-\x-1.south west) --   
    ([xshift=-.5\pgflinewidth]mat-\x-5.south east);
  }
\foreach \x in {1,...,4}
{
  \draw 
    ([yshift=-.5\pgflinewidth]mat-1-\x.north east) -- 
    ([yshift=-.5\pgflinewidth]mat-4-\x.south east);
}   
 
\begin{scope}
\draw[<-][shorten >= -2pt, shorten <= -3pt]  (mat-2-1.east) -- (mat-2-2.west);
\draw[->][shorten >= 6pt, shorten <= 6pt]  (mat-3-2.center) -- (mat-2-2.center);
\draw[<-][shorten >= -3pt, shorten <= -5pt]  (mat-3-2.east) -- (mat-3-3.west);
\draw[->][shorten >= 6pt, shorten <= 6pt]    (mat-4-3.center) -- (mat-3-3.center);
\draw[<-][shorten >= -3pt, shorten <= -5pt]  (mat-4-3.east) -- (mat-4-4.west);
\end{scope}
\end{tikzpicture}
\end{center}

\end{proof}

Proposition \ref{cor:formalversionk} admits the following interpretation in view of mirror symmetry. When $G$ is a generator of $\mathcal{W}^\Lambda$, then $H^0 (W_\Lambda) = H^0 (\hom_{\mathcal{W}^\Lambda} (G,G))$ can be understood as the ring of (global) functions of the mirror space. In this perspective, the completion $\widehat{H^0 (W_\Lambda)} (\cong \hat{A}_L)$ describes the formal neighborhood of the point that is mirror to $L$, as its corresponding maximal ideal is 
the kernel of the augmentation $\varepsilon_L : H^0 (W_{\Lambda} ) \to \Bbbk_\Lambda$ associated with $L$.

\subsection{Analytic neighborhood of $L$}
The isomorphism $\widehat{H^0 (W_\Lambda)} \to \hat{A}_L=H^0 (BV_{\Lambda}^{\sharp_{\Lambda}})$ in Proposition \ref{cor:formalversionk} can be described as follows. Recall that we have an algebra homomorphism 
$$ \kappa : H^0 (W_\Lambda) \to A_L$$
where we abuse notation by writing the same symbol $\kappa$ for its induced map on the $0$-th cohomology (see Proposition \ref{prop:algebraickappa1}). After composing $A_L \to \hat{A}_L$, we obtain
$$ \tilde{\kappa} : H^0 (W_\Lambda) \to A_L \to \hat{A}_L .$$
Let $\mathfrak{m}$  denote $ \mathrm{ker} ( \varepsilon_L : H^0 (W_{\Lambda} ) \to \Bbbk_\Lambda) = H^0 (\bar{W}_{\Lambda})$ for simplicity.
Consider an element $(w_k)_{k \geq 1}$ in the $\mathfrak{m}$-adic completion $\widehat{H^0 (W_\Lambda)}$, that is, $w_k \in H^0 (W_\Lambda) / \mathfrak{m}^k$ such that $w_k \equiv w_i \mod \mathfrak{m}^i$ for $i \leq k$. Since $\tilde{\kappa}$ is a ring homomorphism, it maps an element in $\mathfrak{m}^k$ to an element of $A_L$ whose word-length is at least $k$. In other words, its image under $\tilde{\kappa}$, viewed as a linear map $BV_{\Lambda} \to \Lambda$, vanishes on $\mathcal{F}^{k-1} BV_{\Lambda} = \oplus_{l=0}^{k-1} V_{\Lambda}^{\otimes l}$.  Thus, if $\mathbf{X}  \in B_l V_{\Lambda}$, then $\kappa  (w_k) (\mathbf{X})$ stabilizes after $k = l+1$. Therefore $\tilde{\kappa}$ extends to a map $\hat{\kappa} :\widehat{H^0 (W_\Lambda)} \to \hat{A}_L$.

From the discussion above, $\hat{\kappa} :\widehat{H^0 (W_\Lambda)} \to \hat{A}_L$ is simply an extension of $\kappa$  which was defined implicitly by 
\begin{equation}\label{eqn:kappacformula1}
\kappa: H^0(W_\Lambda) \to A_L \qquad \sum_{i \geq 0} m_{i+2}^{\mathcal{W}^\Lambda} (Z,P, \overbrace{b,b,\cdots, b}^{i}) =  P \,  \kappa  (Z) 
\end{equation}
where $b= \sum x_i X_i$ is a formal linear combination of degree 1 generators $\widetilde{X}_i$ of $V =\hom_\mathcal{W^\Lambda} (L,L)$ coupled with their dual coordinate functions $x_i$. More specifically, $ \kappa (Z)$ is obtained by expanding the right hand side by linearity using the rule \eqref{eqn:expruleback}, and $\hat{\kappa}$ allows a series in the $\mathfrak{m}$-adic completion in the place of $Z$.
Note that $\sum x_i X_i = \sum \tilde{x_i} \widetilde{X}_i$ since the valuations of $x_i$ and $X_i$ compensate each other, and one can rewrite \eqref{eqn:kappacformula1} in terms of exact generators in order to have the analogue of $\hat{\kappa}$ in $\C$-coefficient setting. In this case, \eqref{eqn:kappacformula1} should be expanded in terms of $\tilde{x_i}$'s, as otherwise the expansion would involve some nontrivial powers of $T$'s.

Through an obvious (proper) inclusion $A_L \subset \hat{A}_L$, $A_L$ can be viewed as a subalgebra of the formal completion of $H^0 (\hom_{\mathcal{W}^\Lambda } (G,G)) (\cong \hat{A}_L)$. In view of interpretation made at the end of \ref{subsec:formalnbdhat}, $A_L$ corresponds to a neighborhood of the point mirror to $L$ bigger than its formal neighborhood. 
It will give us some benefit in studying the local structure of the mirror space in Section \ref{sec:locsyzapp1}. The following summarizes our discussion in this section.

 \begin{thm}\label{thm:mainkd}
If two exact Lagrangian branes $L=\oplus_{i}^r L_i $ and $G=\oplus_{i}^r G_i$ in $\mathcal{W}Fuk (M)$ satisfy Assumption \ref{assume:kospairinw},
then the Maurer-Cartan algebra $A_L$ is properly embedded in the $\mathfrak{m}$-adic completion $\widehat{HW^0 (G,G;\Lambda)}$ of $HW^0 (G,G;\Lambda)$ for $\mathfrak{m} = \ker \varepsilon_L$.

The embedding of $A_L$ is the restriction of the inverse of $\hat{\kappa}: \widehat{HW^0 (G,G;\Lambda)} \stackrel{\cong}{\to} \hat{A}_L$ where $\hat{\kappa}$ is given as the natural extension of
\begin{equation}\label{eqn:formulakoszul11}
\sum m_k (Z,P,b,\cdots, b) = P \cdot \kappa(Z) 
\end{equation}
where $b= \sum x_i X_i$ is a formal linear combination of degree 1 generators.
\end{thm}
 
Identified as a subalgebra of $\widehat{HW^0 (G,G;\Lambda)}$, a general element $F$ of $A_L$ can be described as follows. Recall that $\widehat{HW^0 (G,G;\Lambda)} / \mathfrak{m}$ is generated by finitely many elements, say $Z_1, \cdots, Z_N$, as $\Bbbk_\Lambda$-module. By scaling, we may assume that $\kappa(Z_i)$ has valuation $0$ for each $i$. Then $F$ can be identified as  a (noncommutative) power series in $Z_1, \cdots, Z_N$, valuations of whose coefficients are bounded below. More precisely, $F = (w_k)_{k \geq 1}$ with $w_k \in HW^0 (G,G;\Lambda) / \mathfrak{m}^k$ givens as the sum of $\deg \leq (k-1)$ terms in this series. Note that any element in $\widehat{HW^0 (G,G;\Lambda)}$ can also be written as such a series, but with no requirement on  valuations of the coefficients.

\begin{remark}
Recall that the map \eqref{eqn:formulakoszul11} is actually defined on the chain level. 
Taking the formal linear combination $b$ including generators of all degrees as in Lemma \ref{lem:donbvvee}, we have an $\AI$-homomorphism 
\begin{equation}\label{eqn:kappachain1}
\kappa:CW(G,G;\Lambda) \to \mathcal{A}_L.
\end{equation}
 The relationship between algebraic structures on $CW(G,G)$ and the Maurer-Cartan algebra $\mathcal{A}_L$ can be heuristically seen by a usual cobordism argument in Floer theory. For example, Figure \ref{fig:cobord} (a) describes the cobordism that implies compatibility of the chain-level Koszul map \eqref{eqn:kappachain1} with the differential on $CW(G,G)$ and $\mathcal{A}_L$. 
 Likewise, the compatibility of $\kappa$ with products can be intuitively seen as in Figure \ref{fig:cobord} (b).
\end{remark}

\begin{figure}[h]
	\begin{center}
		\includegraphics[scale=0.35]{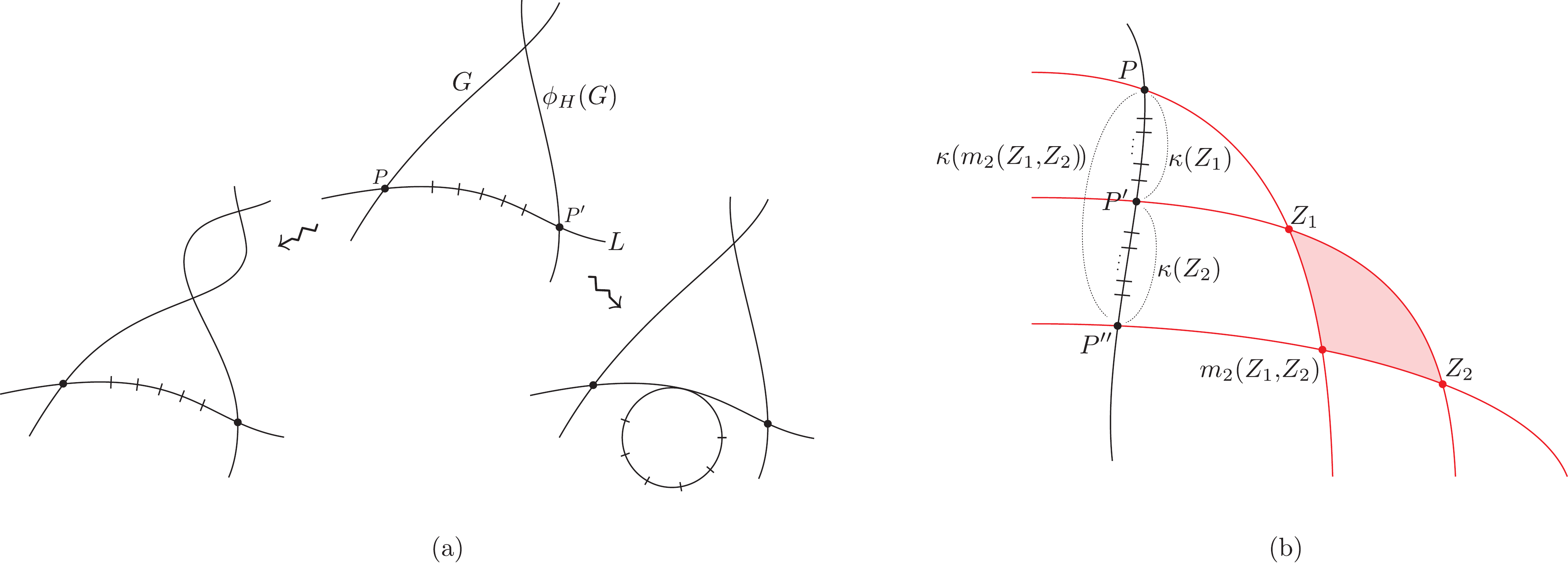}
		\caption{Compatibility between (a) differentials and (b) products on $CW(G,G)$ and $\mathcal{A}_L$}
		\label{fig:cobord}
	\end{center}
\end{figure}


\subsection{(local) SYZ fibration: a generating section and fibers}\label{subsec:locsyztfibers}
We next consider the following geometric setting for a local SYZ fibration. Let $(M,\omega = d \Theta)$ be a Liouville manifold and, suppose it admits a Lagrangian torus fibration $\pi : M \to B$. We will deal with fibers $\pi^{-1} (b)$ sitting near an exact torus fiber $L_0$. They have minimal Maslov number $0$, but are not exact in general.
By Theorem \ref{thm:mainkd}, the Maurer-Cartan algebra of $L_0$ can be obtained as a subspace of the $\mathfrak{m}$-adic completion of $\hom_{\mathcal{W}^\Lambda} (G,G)$ for $\mathfrak{m} = \ker (\varepsilon_{L_0} )$. We are interested in the behavior of nearby torus fibers $L$ and their associated Koszul maps $\kappa: \hom_{\mathcal{W}^\Lambda} (G,G) \to A_L $. 

We claim that there exists an $1$-form $\Theta$' satisfying $d \Theta' = \omega$ and $\Theta'|L = d f$ for some function $f: L \to \R$, if $L$ is sufficiently close to $L_0$. Let $u_1,\cdots u_n$ and $\rho_1,\cdots \rho_n$ be action-angle coordinates such that $(u_1,\cdots,u_n)=(0,\cdots,0)$ for points in $L_0$. If we set $a_i:= \int_{\beta_i} \Theta|_L $ for $\beta_i$ a loop along the direction of $\rho_i$, then $(u_1,\cdots,u_n)= (a_1,\cdots,a_n)$ gives the action coordinates for $L$ by Stokes' theorem since $d \Theta = \omega$.

Let us take contractible open sets $U_1, U_2 (\subset B)$ with $U_1 \subset U_2$, both of which contain $(0,\cdots,0)$ and $(a_1,\cdots,a_n)$ where we use $u_i$ as coordinates on the base $B$ of $\pi$. Then we choose compactly supported functions $h_i (u_1,\cdots, u_n)$ for $1 \leq i \leq n$ such that $h_i|_{U_1} \equiv a_i$ and $h_i|_{B\setminus U_2} \equiv 0$. They give rise to an 1-form $\widetilde{\Theta}' := \Theta - \sum h_i d \rho_i$ on $M$.
When $L$ is sufficiently close to $L_0$, we can make $|d h_i|$ small enough so that  $d \widetilde{\Theta}'$
is still symplectic.

Notice that $d \widetilde{\Theta}'$ and $\omega$ differ from each other by a compactly supported exact $2$-form. Hence Moser's trick applies to the situation to produce a (compactly supported) diffeomorphism $\psi$ such that $\psi^\ast (d \widetilde{\Theta}' ) =  \omega$. As $\omega = d \tilde{\Theta}'$ on $U$ that contains $L$, $\psi$ preserves $L$. In particular, if we set $\Theta':=\psi^\ast \widetilde{\Theta}'$, then $d \Theta'=\omega$ and $\Theta'|_L = \widetilde{\Theta}'|_L$ is exact since its integration over any loops in $L$ vanishes by construction. By construction, the close form $\Theta'|_G - \Theta|_G$ vanishes away from a contractible subset of $G$, and hence is exact. Thus $G$ is an exact Lagrangian with respect to $\Theta'$, and the pair $(G,L)$ of $\Theta'$-exact Lagrangians fits into the setting of Theorem \ref{thm:mainkd}. Therefore the Maurer-Cartan algebra of $L$ can be obtained as a suitable completion of $\hom_{\mathcal{W}^\Lambda} (G,G)$. 

We can make the above discussion more concrete by relating Floer theory of $(G,L)$ with that of $(G,L_0)$ as follows.
The closed $1$-form $\alpha:=\Theta' - \Theta$ induces a symplectomorphism $\phi$ by integrating the symplectic vector field $X_\alpha$ given by $\iota_{X_\alpha} \omega = \alpha$. Since $\alpha = - \sum a_i d\rho_i $ on $\pi^{-1} (U_1)$, the symplectomorphism $\phi$ simply translates $L$ back to $L_0$ via $ (u_1, \cdots, u_n) \mapsto (u_1- a_1 , \cdots, u_n - a_n )$. 
On the other hand, $\pi^{-1} (U_2) \cap G$ lies in some contractible open set since $G$ intersects each torus fiber at one point, and the $1$-form $\alpha$ admits a well-defined primitive, say $\widetilde{F}$, over this set. We extend $\widetilde{F}$ outside $\pi^{-1}(V)$ using suitable bump functions to get $F: M \to \R$.  Thus, restricting ourselves to $G$, we have $\phi(G) = \phi_F (G)$ since $dF$ and $\alpha$ agrees with each other on a region near $G$.

From the above discussion, we conclude that Floer theory of the pair $(G,L)$ can be identified with that of $(\phi  (G), \phi  (L)) = (\phi_{F} (G), L_0)$, and hence that of $(G,L_0)$ as $F$ is a compactly supported Hamiltonian. Suppose that we make the wrapping for both $G$ and $\phi_F (G)$ happen away from the support of $F$ so that $\hom_{\mathcal{W}^\Lambda} (G,G)$ and $\hom_{\mathcal{W}^\Lambda} (\phi_F (G), \phi_F (G))$ have precisely the same set of geometric generators. However, these set-theoretically identical generators have different valuations on the two $\Lambda$-vector space since $\phi_F$ does \emph{not} preserve $\Theta$ in general. In practice, 
$$ \hom_{\mathcal{W}^\Lambda} (G,G) \stackrel{\cong}{\to} \hom_{\mathcal{W}^\Lambda} (\phi_F (G), \phi_F (G))$$
simply scales these geometric intersection points by suitable power of $T$. In this case, the Koszul map $\hom_{\mathcal{W}^\Lambda} (G,G) \to A_L$ can be easily obtained from that for $(G,L_0)$ after such scaling generators of $\hom_{\mathcal{W}^\Lambda} (G,G)$.

\subsection{Example of $T^\ast S^1$}\label{subsec:ts1exfirst}

Let us first look into the simplest case of $M = T^\ast S^1$ as a warm-up example. 
We use coordinate $(t,s) \in T^\ast S^1=S^1 \times \mathbb{R}$ where $t \in [0,2\pi]/0\sim 2\pi$ as in Example \ref{ex:cotcirclenegt}, so we have $\omega=dsdt$ with $\Theta=sdt$. 
The zero section $L:= S^1$ is exact with respect to $\Theta$, and clearly,
$L$ and the cotangent fiber $G$ satisfy Assumption \ref{assume:kospairinw}.
As before, we take the basis of $H^\ast \hom_{\mathcal{W}^\Lambda} (G,G) (= HW(G,G;\Lambda))$ formed by $\widetilde{Z_i}$ for $i \in \mathbb{Z}$ with the multiplication $ m_2 ( \widetilde{Z_i}, \widetilde{Z_j}) = \widetilde{Z_{i+j}}$ 
being the only nontrivial $\AI$-operation. 
On the chain level, $\widetilde{Z}_i$ can be represented by $T^{-C\left(\frac{i^2}{w} + \frac{w}{(1-\epsilon)^2}\right)} Z_i^{(w)} \in CF(G,G;wH)$ (using notations in Example \ref{ex:cotcirclenegt}) for $w$ large enough. 
Notice that $H^\ast \hom_{\mathcal{W}^\Lambda} (G,G) \cong \Lambda [ \mathbf{Z}, \mathbf{Z}^{-1} ]$ is precisely the function ring of $\Lambda^\times$, where $\mathbf{Z}:=\widetilde{Z_1}$ serves as the standard coordinate, and $\widetilde{Z_0}$ is identified with $1 \in \Lambda [ \mathbf{Z}, \mathbf{Z}^{-1} ]$ 

We write $P$ for the unique intersection point between $G$ and $L$, and at the same time, $P$ denotes the corresponding geometric generator of $\hom_{\mathcal{W}^\Lambda} (G,L)$. In this example, we have $\widetilde{P}=P$ 
since the primitives of $\Theta|_L$ and $\Theta|_G$ are both zero. 
From direct counting (see Figure \ref{fig:ts1} (a)), 
$m_2 (\widetilde{Z_1},\widetilde{P}) = \widetilde{P}$, where the Novikov parameter $T$ does not appear since the computation involves exact generators that has zero actions. Hence, the augmentation $\varepsilon_L : H^\ast \hom_{\mathcal{W}^\Lambda} (G,G) \to \Lambda$ associated with $L$ sends $\widetilde{Z_1}$ to $1$ which implies $\varepsilon_L (\widetilde{Z_i}) = 1$ for all $i$. 
Consequently, $\varepsilon$ has the kernel
$$ \ker \varepsilon_L = \{ \widetilde{Z_i} - \widetilde{Z_0} : i \in \Z \} $$
and in particular, the corresponding point in $\Lambda^\times$ is $\mathbf{Z}=1$.

On the other hand, the Maurer-Cartan algebra $A_L$ is a ring $\Lambda \{ x \}$ consisting of  bounded power series, where $x$ is linear dual to $d t \in H^1 (S^1)$. Setting $b=x d t$, the Koszul map $\kappa$ computes
$$\widetilde{P} \, \kappa(\widetilde{Z_1})= \sum_{k \geq 0} m_{k+2} (\widetilde{Z_1}, \widetilde{P}, b, \cdots, b) = \displaystyle\sum_{k=0}^{\infty} \frac{1}{k!} x^k m_2 (\widetilde{Z_1},\widetilde{P}) = \widetilde{P} \, e^{x}, $$
(again no $T$ appears since everything involved has zero actions)
and hence $\kappa\left(\mathbf{Z}\right) = \kappa (\widetilde{Z_1})  =e^x$. Here, $m_{k+2} (\widetilde{Z_1},\widetilde{P},dt,\cdots, dt) = \frac{1}{k!} m_2 (\widetilde{Z_1},\widetilde{P})$ holds, since the left hand side is contributed by geometrically the same disk as $m_2 (\widetilde{Z_1},\widetilde{P})$, but the actual moduli space is (the compactification of) the configuration space $\{ (t_1,\cdots,t_k) : 0<t_1 < \cdots < t_k<1\}$ due to locations of boundary markings for $b$'s, on which the integration of $ev_3^\ast dt \wedge \cdots \wedge ev_{k+2}^\ast dt =  dt_1 \cdots dt_k$ gives the extra factor $\frac{1}{k!}$.
\begin{figure}[h]
	\begin{center}
		\includegraphics[scale=0.5]{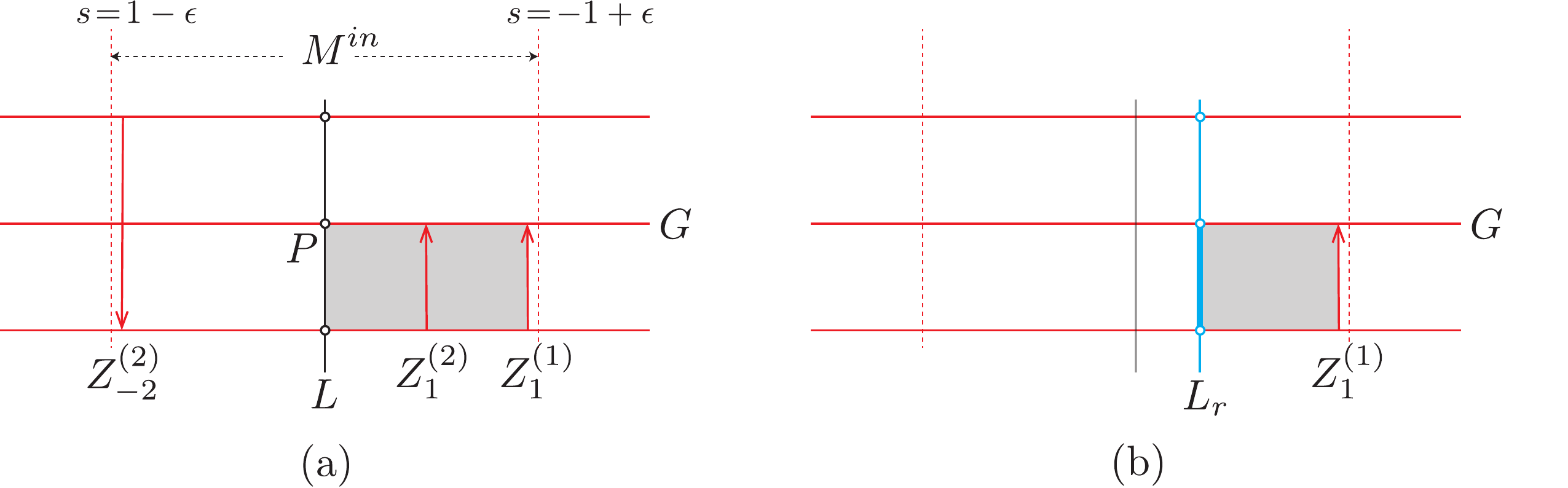}
		\caption{The zero section $L$ and a cotangent fiber $G$ in $T^\ast S^1$}
		\label{fig:ts1}
	\end{center}
\end{figure}

\begin{remark}
We can deform $CF(L,L)$ by $\Lambda_U$-flat connections $\nabla^z$ whose holonomy is parametrized by $z \in \Lambda_U$. The associated $\kappa$ can be computed as 
$$\widetilde{P} \, \kappa(\widetilde{Z_1})=   m_{2}^{\nabla^z} (\widetilde{Z_1}, \widetilde{P}) = \widetilde{P} \, z.$$
Thus $\kappa (\mathbf{Z}) = z$, which is consistent with the earlier computation if we set $z=e^x$.
\end{remark}

By Theorem \ref{thm:mainkd}, the Maurer-Cartan space of $L$ can be obtained as a suitable subspace of the $\mathfrak{m}$-adic completion of $\Lambda [ \mathbf{Z}, \mathbf{Z}^{-1} ]$ where $\mathfrak{m} = \langle \mathbf{Z}-1 \rangle$. Recall that this subspace consists of series
\begin{equation}\label{eqn:b11series}
 c_0 + c_1 (\mathbf{Z}-1) + c_2 (\mathbf{Z}-1)^2 + \cdots
\end{equation}
for $c_i \in \Lambda$, as long as $-\infty < \inf \{val(c_i) \colon i=0,1,2,\cdots\} $. In particular, $x$ can be identified with
$$ (\mathbf{Z}-1) - \frac{1}{2} (\mathbf{Z}-1)^2 + \frac{1}{3} (\mathbf{Z}-1)^3 - \cdots = \log \left((\mathbf{Z}-1) +1 \right),$$
where ``$\log$" on the right hand side means a priori a formal expansion.
In order for such a series \eqref{eqn:b11series} to be well-defined as a genuine  function on $\Lambda^\times$, $\mathbf{Z}$ must be lie in $1 + \Lambda_+$. Notice that $\mathbf{Z}=1 \in \Lambda^\times$ is precisely the point corresponding to $L$. We conclude that the Maurer-Cartan space of $L$ sits in $\Lambda^\times$ as  $B(1,1)$, the open ball centered at $1$ with radius $1$ (which is merely a different way of describing the same set $1+\Lambda_+$). 

One can further deform $L$ by equipping it with a fixed flat connection $\nabla^\alpha$ with holonomy $\alpha \in \Lambda_U$. Repeating the same computation, we see that the object $(L,\nabla^\alpha)$ sits at a point $Z=\alpha$ in $\Lambda^\times$, and the corresponding Maurer-Cartan space is the open ball $B(\alpha,1)$ centered at this point with radius $1$.

We next consider a general torus fiber $L_r:=\{(t,s) \in S^1 \times \mathbb{R} \colon s=r\}$. Since $L_r$ is no longer exact, we cannot directly reduce the computation to $\C$-coefficient case as we did for $L_0$. Let $P_r$
 be the unique intersection points in $G \cap L_r$.
However, it is clear from the picture that the only difference is the energy of the contributing disks. More specifically, one needs to additionally take into account the integration of $\Theta$ along the boundary of the disk that lies on $L_r$. For instance, the shaded disk in Figure \ref{fig:ts1} (b) contributes to
$$m_2(\widetilde{Z_1}, P_r) 
= T^{ -r} P_r,$$
since $\Theta$ restricts to $r dt$ on $L_r$, and its integration over the corresponding boundary segment of the disk is $-r$.
The associated augmentation is given by
$\varepsilon_{L_r} (\mathbf{Z}) = \varepsilon_{L_r} ( \widetilde{Z_1}) = T^{ - r}$. 
Therefore the maximal ideal $\ker \varepsilon_r$ indicates the point $\mathbf{Z}= T^{-r}$ in $\Lambda^\times$.
Similarly, $\kappa( \mathbf{Z} )$ has an additional factor $T^{-r}$ compared with  the case of $L_0$, that is, $\kappa( \mathbf{Z} ) = T^{-r} e^x$, and this completely determines $\kappa$ since it is a ring homomorphism.

Setting $\mathbf{Z}_r:= T^{r} \mathbf{Z}$, one can go back to the same situation as in $L=L_0$, except that $\mathbf{Z}_r$ has a nontrivial valuation. To recover $x$, one can take the formal expansion
$$ (\mathbf{Z}_r-1) - \frac{1}{2} (\mathbf{Z}_r-1)^2 + \frac{1}{3} (\mathbf{Z}_r-1)^3 - \cdots = \log \left((\mathbf{Z}_r-1) +1 \right).$$
Notice that this is consistent with the discussion in \ref{subsec:locsyztfibers}, in that the Koszul map for $L_r$ can be obtained from that of $L_0$ after scaling $\mathbf{Z}$ to $\mathbf{Z}_r=T^{r} \mathbf{Z}$. On the level of algebra, the Maurer-Cartan algebra of $L_r$ is isomorphic to the the completion of $\Lambda [ \mathbf{Z}_r,\mathbf{Z}_r^{-1}]$ at the ideal $\langle \mathbf{Z}_r - 1\rangle$.
Let us describe the corresponding region in $\Lambda_+$. In the original coordinate $\mathbf{Z}$, one has
$$ T^r (\mathbf{Z}-T^{-r}) - \frac{1}{2} T^{2r} (\mathbf{Z}-T^{-r})^2 + \frac{1}{3} T^{3r} (\mathbf{Z}- T^{-r})^3 - \cdots = \log \left(T^r (\mathbf{Z}-T^{-r}) +1 \right),$$
which is valid as a function $\mathbf{Z}$ for $\mathbf{Z} \in T^{-r} + T^{-r} \Lambda_{+}$. Therefore we conclude that the Maurer-Cartan deformation space of $L_r$ sits in the global mirror $\Lambda_+$ as the open ball $B(T^{-r}, e^r)$.

We remark that the balls obtained by varying $r$ are mutually disjoint. This is consistent with the fact that different torus fibers do not intersect each other, and hence Floer theory between them is trivial.

\subsection{$\Lambda_U$-flact connections on $L$}
As one can see in the previous example, in Maurer-Cartan deformation, $\Lambda_U$-connections serve identically as degree one cocycles in $L$ paired with $\C^\ast$-flat connections after exponentiating the variable, as long as the divisor-type axiom holds. 
Therefore, in practice, one may take $(\Lambda_U)^{b_1}$ parametrizing $\Lambda_U$-flat connections on $L$ for related computations, instead of the coordinate functions on $H^1(L)$. Here, $b_1$ denotes the first Betti number of $L$) 

When $\vec{z}$ varies over $(\Lambda_U)^{b_1}$, the associated Koszul map is defined by
\begin{equation}\label{eqn:kosmapconn1}
P \, \kappa_l (Z_1,\cdots,Z_l)=  m_{l+1}^{\nabla^{\vec{z}}} (Z_1,\cdots,Z_l, P),
\end{equation}
and it is essentially determined by the augmentation $\epsilon_L$ (once we know the boundary classes of contributing disks). By our assumption on the degrees, \eqref{eqn:kosmapconn1} is nontrivial only when $l=1$, so it reduces to an algebra homomorphism
\begin{equation}\label{eqn:kosmapconn2}
 H^0 \hom_{\mathcal{W}^\Lambda} (G,G) \to \Lambda \{ z_1^{\pm},\cdots, z_{b_1}^{\pm} \}
\end{equation}
where the right hand side consists of infinite Laurent series $ \sum_{i=1}^{\infty} a_i z^{v_i}$ with $\lim_{i \to \infty} val(a_i) = \infty$. (See Example \ref{ex:expvar11} and the paragraph above it for the related discussions.)

In the applications in Section \ref{sec:locsyzapp1}, we will proceed mainly with connections rather than dealing with divisors on $L$. Readers are warned that one should switch back to the original variable $x_i$ (dual to $X_i \in H^1 (L)$) to apply Theorem \ref{thm:mainkd}.

\section{Examples: local models for SYZ fibrations}\label{sec:locsyzapp1}

We examine the main idea of Theorem \ref{thm:mainkd} using simple, but  illustrative geometric examples, which are typical  local models for SYZ fibrations in low dimensions. When there is no singular fiber, then the fibration is locally the same as $(\mathbb{C}^\times)^n = T^\ast T^n$, which is nothing but the product of  Example \ref{subsec:ts1exfirst} that we have already analyzed in detail. Like before, the Maurer-Cartan space of each torus fiber in $(\mathbb{C}^\times)^n$ forms a small chart in $(\Lambda^\times)^n$. Notice that $(\Lambda^\times)^n$ is $\mathrm{Spec}$ of the endomorphism algebra of the Lagrangian $G \cong \R^n$ in $\mathcal{W}Fuk^\Lambda ((\mathbb{C}^\times)^n)$. These charts are indeed disjoint small balls in $(\Lambda^\times)^n$. More precisely, the Maurer-Cartan space of the  Lagrangian fiber $$L_{\vec{r}} :=\{(z_1,\cdots, z_n) \colon \log  |z_i| = r_i\} $$
equipped with the holonomy $ \nabla^{\vec{c}} = (c_1,\cdots, c_n) \in (\Lambda_{\mathrm{U}})^n$ sits in the global space $(\Lambda^\times)^n$ as a polydisk $B(T^{-r_1} c_1, e^{r_1}) \times \cdots \times B(T^{-r_n} c_n, e^{r_n})$  about the point $ (T^{-r_1} c_1, \cdots, T^{-r_n} c_n )$. 

Before we proceed to our main examples, let us make a short remark on the cotangent bundle of a simply connected manifold and related works.

\subsubsection*{Cotangent bundles}
The wrapped Floer cohomology of a cotangent fiber in $T^\ast L$ has been shown to be quasi-isomorphic to $C_{-\ast} (\Omega_0 L)$ by Abouzaid \cite{abouzaid2012} where $\Omega_0 L$ is the based loop space. On the other hand, the Maurer-Cartan dga $\mathcal{A}_L$ of $L$ can be expressed as $\Omega H_{-\ast} (L;\Lambda)$, where we identify the continuous dual of $H^\ast (L;\Lambda)$ with $H_{-\ast} (L;\Lambda)$ (which is possible since everything is finite dimensional). 

Suppose $L$ is simply connected. Then the $A_\infty$-coalgebra structure on $H_{-\ast} (L;\Lambda)$ is finite by degree reason, i.e, $\Delta_i (c)=0$ except for finitely many $i$'s. Thus one can avoid completion procedure in \ref{subsubsec:cobarconst} for the cobar construction of $H_{-\ast} (L;\Lambda)$, and the resulting dga is quasi-isomorphic to $C_{-\ast} (\Omega_0 L)$ by Adams \cite{Ad}. Combined with the result of Abouzaid, the Maurer-Cartan dga of $L$ can be identified with a suitable completion of the wrapped Floer cohomology of a cotangent fiber. (This point has been already discussed implicitly in \cite{EL-duality}.)
Based on the above observation, we speculate that Theorem \ref{thm:mainkd} generalizes to the nonzero degree components under some milder assumptions on the grading of $G$ and $L$.

\vspace{0.3cm}

We now begin to explore more intriguing examples of local SYZ fibrations that comes with typical type of singularities. We will be interested in the Maurer-Cartan spaces of singular fibers as well .

\subsection{Local model for $I_1$-singular fiber}\label{subsec:i1ex}
Let us consider the following exact symplectic manifold
$$ M= \C^2 \setminus \{z_1 z_2 = \epsilon \},$$
which is a toy model for the $2$-dimensional SYZ fibration that has a unique nodal singular fiber. The fibration is given by 
\begin{equation}\label{eqn:i1syzf}
(z_1,z_2) \mapsto ( |z_1|^2 - |z_2|^2, |z_1 z_2 - \epsilon|)
\end{equation}
and its Floer theory and mirror symmetry have been studied and fully understood through many literatures for e.g., \cite{auroux07}. The singular fiber denoted by $L$ occurs when $|z_1 z_2 - \epsilon|= \epsilon$ and $|z_1|^2 - |z_2|^2 =0$, and $L$ is an immersed sphere with a transversal double point at $z_1=z_2=0$.
It is easy to visualize the Lagrangians in our interest using the conic fibration $w: M \to \C^\ast$, $w(z_1,z_2) = z_1 z_2 - \epsilon$. The projection of torus fibers   draws a concentric circle about the origin, as does that of $L$ in particular (see Figure \ref{fig:A1}). 

The Maurer-Cartan algebra $A_L$ is the ring of bounded power series ring on two variables $u$ and $v$,
$$A_L = \Lambda \{\!\{u,v \}\!\}/ uv = vu,$$
 associated with the immersed generator $U$ and $V$ supported at the nodal point. See for e.g., \cite{Sei_Lec} or \cite{HKL} for more details. In fact, $A_L$ can be computed from the wrapped Floer theory of the dual Lagrangian and the Koszul map $\kappa$.
 
The geometric generators $U$ and $V$ have nontrivial \emph{actions} $\mathcal{A} (U)$ and $\mathcal{A} (V)$ (see \eqref{eqn:actionforimm}) depending on the size of the base circle in $w$-plane that $L$ projects to. Since $U$ and $V$ are complementary to each other supported at the same self-intersection, we have $\mathcal{A} (U) + \mathcal{A} (V) =0$. As before, we denote by $\widetilde{U}$ and $\widetilde{V}$ the associated exact generators, i.e., $\widetilde{U} = T^{\mathcal{A} (U)} U$ and $\widetilde{V} = T^{\mathcal{A} (V)} V$.
  
Let us next consider $G$, a Lagrangian section of \eqref{eqn:i1syzf} that projects to the positive real axis in $w$-plane. Over each point in the positive real axis lies $\C^\ast \cong T^\ast S^1$ given by $z_1 z_2 = const.$ and $G$ is isomorphic to a cotangent fiber of $T^\ast S^1$ along this direction. 
$G$ generates the wrapped Fukaya category of $M$, and hence Theorem \ref{thm:mainkd} applies to the situation in this example.\footnote{One technical subtlety here is that unless we allow immersed Lagrangians as objects of $\mathcal{W}^\C$, we may need to identify $L$ with a  twisted complex $\tilde{L} \stackrel{\delta}{\to} \tilde{L}$ built out of the thimble $\tilde{L}$ emanating from $(z_1,z_2)=(0,0)$.
 As computations with $L$ or with $\tilde{L} \stackrel{\delta}{\to} \tilde{L}$ are not much different from each other, we proceed with $L$ itself, here.}
	The wrapped Floer cohomology of $G$ can be computed by counting holomorphic sections of the fibration $w$ following the idea of \cite[Proposition 4.5]{Pas}. We remark that \cite{CU} computed the wrapped Floer cohomology of analogous Lagrangians in more general situation of $A_n$. Our computation below uses a Hamiltonian similar to ``$H_2$" appearing in \cite[Section 6]{CPU}.
\begin{figure}[h]
\begin{center}
\includegraphics[height=1.8in]{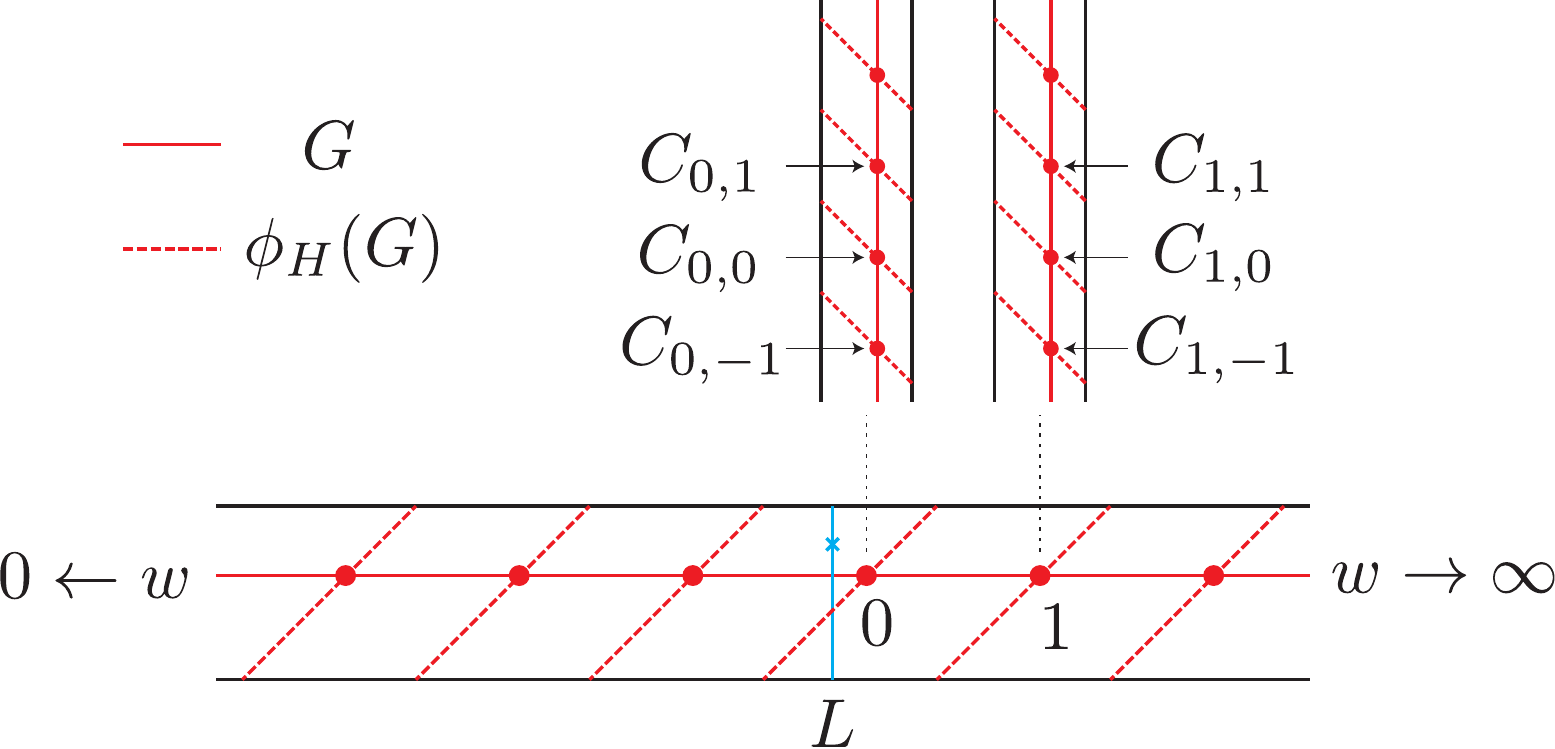}
\caption{Generators of $\hom_{\mathcal{W}^\Lambda} (G,G)$ in the picture of the conic fibration $w=z_1 z_2 -\epsilon$ }\label{fig:pertc2}
\end{center}
\end{figure}

Hamiltonian chords for a multiple of linear Hamiltonian $wH$ can be parametrized by $C_{a,b}$ for lattice points with $|a| \leq a_w$ and $|b| \leq b_w$ where $a_w,b_w \to \infty$ as $w \to \infty$. In Figure \ref{fig:pertc2}, the intersection points in $G \cap \phi_{wH} (G)$ corresponding to these generators are plotted (we will mostly describe disks with these intersection points below, as they are easier to visualize).
Each of them produces a degree $0$ exact generator denoted as $\widetilde{C}_{a,b}$ in the $m_1$-cohomology, and in particular, $H^\ast (\hom_{\mathcal{W}^\Lambda} (G,G))$ is simply an algebra without higher operations, concentrated at degree $0$. The product $m_2$ is given by
\begin{equation}\label{eqn:m2fora1}
m_2 (\widetilde{C}_{a_1,b_1}, \widetilde{C}_{a_2,b_2}) = \sum^{k}_{i=0} {k\choose i} \widetilde{C}_{a_1 + a_2, b_1 + b_2 + i} 
\end{equation}
where $k$ is given by
\begin{equation*}
k= \left\{ 
\begin{array}{cl}
\min \{|a_1|,|a_2|\}  & \textnormal{if $a_1$ and $a_2$ have different signs}  \\
0 & \textnormal{otherwise.}
\end{array}
\right.
\end{equation*}
Hence it is commutative.
It is elementary to deduce from \eqref{eqn:m2fora1} that two elements $\widetilde{C}_{\pm1,0}$ generate $H^\ast (\hom_{\mathcal{W}^\Lambda} (G,G))$ as an algebra, or it will be more clear after we compare this with $A_L$ via the Koszul map $\kappa$.

\begin{figure}[h]
\begin{center}
\includegraphics[height=3in]{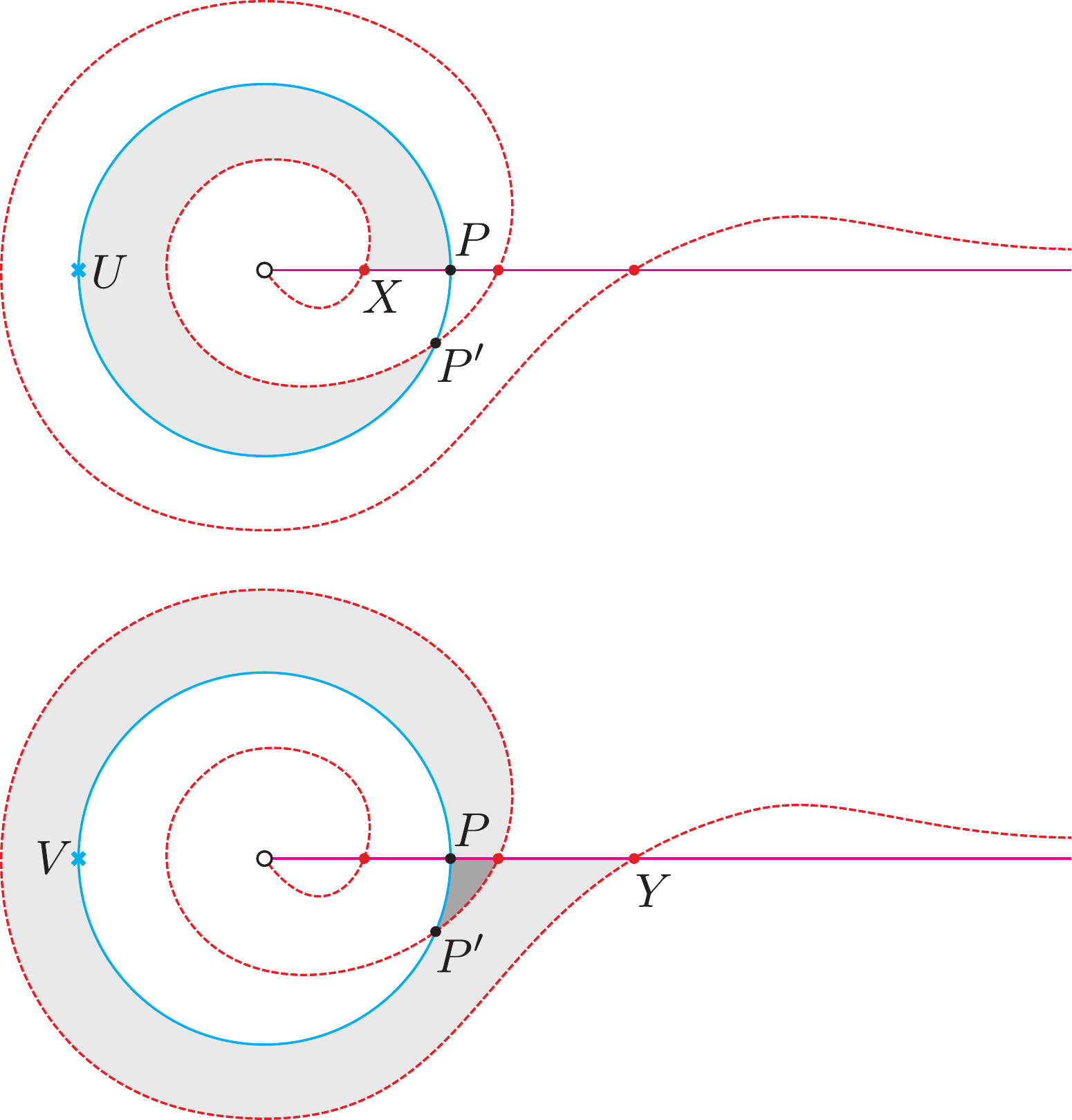}
\caption{Contributing polygons to $\kappa (\widetilde{C}_{\pm1,0})$}\label{fig:A1}
\end{center}
\end{figure}

Let us now consider $\kappa : H^0 (\hom_{\mathcal{W}^\Lambda} (G,G)) \to A_L$.
Shaded regions in Figure \ref{fig:A1} show holomorphic sections of the conic fibration that contribute to $\kappa(\widetilde{C}_{\pm1,0})$. (In the figure, $G$ is  wrapped minimal number of times to show the generators $C_{\pm 1,0}$ as intersection points.) This implies $\kappa(\widetilde{C}_{-1,0}) = \widetilde{u}$ and $\kappa(\widetilde{C}_{1,0})=\widetilde{v}.$ The outputs do not involve any nontrivial powers of $T$'s, since we work with exact generators and variables. The fact that $A_L$ is commutative can be deduced from this without actually solving the Maurer-Cartan equation for $L$.

Since $\kappa$ is a ring homomorphism, we can compute $\kappa(\widetilde{C}_{a,b})$ for other $a,b \in \Z$ using \eqref{eqn:m2fora1}. For example, one has
$$ \kappa(\widetilde{C}_{0,1}) = \tilde{u} \tilde{v} -1$$
from $m_2 (\widetilde{C}_{-1,0}, \widetilde{C}_{1,0} ) = \widetilde{C}_{0,0} + \widetilde{C}_{0,1}$.
The general formula is given by
\begin{equation}\label{eqn:kappacab}
\kappa(\widetilde{C}_{a,b}) = \left\{
\begin{array}{cl}
(\widetilde{u} \widetilde{v}-1)^{b} \widetilde{u}^{|a|} & a <0 \\ 
(\widetilde{u} \widetilde{v}-1)^{b} \widetilde{v}^{|a|} & a \geq 0
\end{array}
\right.
\end{equation}

Inspired by the formula \eqref{eqn:kappacab}, we set $\mathbf{U} = \tilde{C}_{-1,0}$ and $\mathbf{V} = \tilde{C}_{1,0}$, and it gives rise to a simpler and more familiar presentation $H^\ast(\hom_{\mathcal{W}^\Lambda} (G,G)) = \Lambda [ \mathbf{U},\mathbf{V}, \mathbf{U} \mathbf{V}- 1]$ (the same expression also appears in \cite{Pas_CY}).  Namely, it corresponds to the function ring of $\check{M}:=\Lambda^2 \setminus \{\mathbf{U} \mathbf{V}=1\}$. Most importantly, $\mathbf{U}$ and $\mathbf{V}$ can be thought of as  \emph{global} coordinates on $\check{M}$. The crucial reason behind it is the finiteness of the $A_\infty$-operations on wrapped Floer theory.

To the contrary, recall that the Maurer-Cartan algebra of $L$ is given as $A_L = \frac{\Lambda \{\!\{u,v \}\!\} }{ uv = vu}$ which allows an infinite sum of monomials in $u,v$ as long as the valuation is bounded below. Thus $A_L$ is the function ring of $\left(\Lambda_+ \right)^2$. (The maximal subset of $\Lambda^2$ on which every elements in $A_L$ can be evaluated is $\left(\Lambda_+ \right)^2$.) In other words, the Maurer-Cartan space $\mathcal{MC} (L)$ (dual to $A_L$)  is isomorphic to $\left(\Lambda_+ \right)^2$ with coordinates $u$ and $v$. This is along the same line as the coefficients for the bounding cochains in \cite{FOOO}. 

Notice that $\mathcal{MC} (L)$ is bigger than a formal neighborhood which is already assured by Theorem \ref{thm:mainkd}. 
Let us examine how the local chart $\mathcal{MC} (L)$ sits in the global mirror $\check{M}=\Lambda^2 \setminus \{\mathbf{U}\mathbf{V}=1\}$. First of all, the chart is around the origin $(\mathbf{U},\mathbf{V})=(0,0)$ since their augmentation values are zero.
Recall that the exact variables $\tilde{u} = T^{-\mathcal{A} (U)} u$ and  $\tilde{v} = T^{ -\mathcal{A} (V)} v$ precisely match with the exact generators $\mathbf{U}$ and $\mathbf{V}$ in $H^0 ( \hom_{\mathcal{W}^\Lambda} (G,G))$ under $\kappa$. The constraint that $val(u) >0, val(v)>0$ (i.e., $(u,v) \in \left(\Lambda_+ \right)^2$) translated into the condition $val(\mathbf{U}) > -\mathcal{A} (U)$, $val(\mathbf{V}) >  -\mathcal{A} (V)$ where we view $(\mathbf{U}, \mathbf{V})$ as global coordinates on $\check{M}=\Lambda^2 \setminus \{\mathbf{U}\mathbf{V}=1\}$. We see that $\mathcal{MC} (L)$ embeds into $\check{M}$ as the polydisk 
$$\mathcal{U}_L:=\{ (\mathbf{U}, \mathbf{V}) : ||\mathbf{U}|| < e^{ \mathcal{A} (U)}, ||\mathbf{V}|| < e^{ \mathcal{A} (V)} \}.$$ 

\begin{figure}[h]
\begin{center}
\includegraphics[height=2.4in]{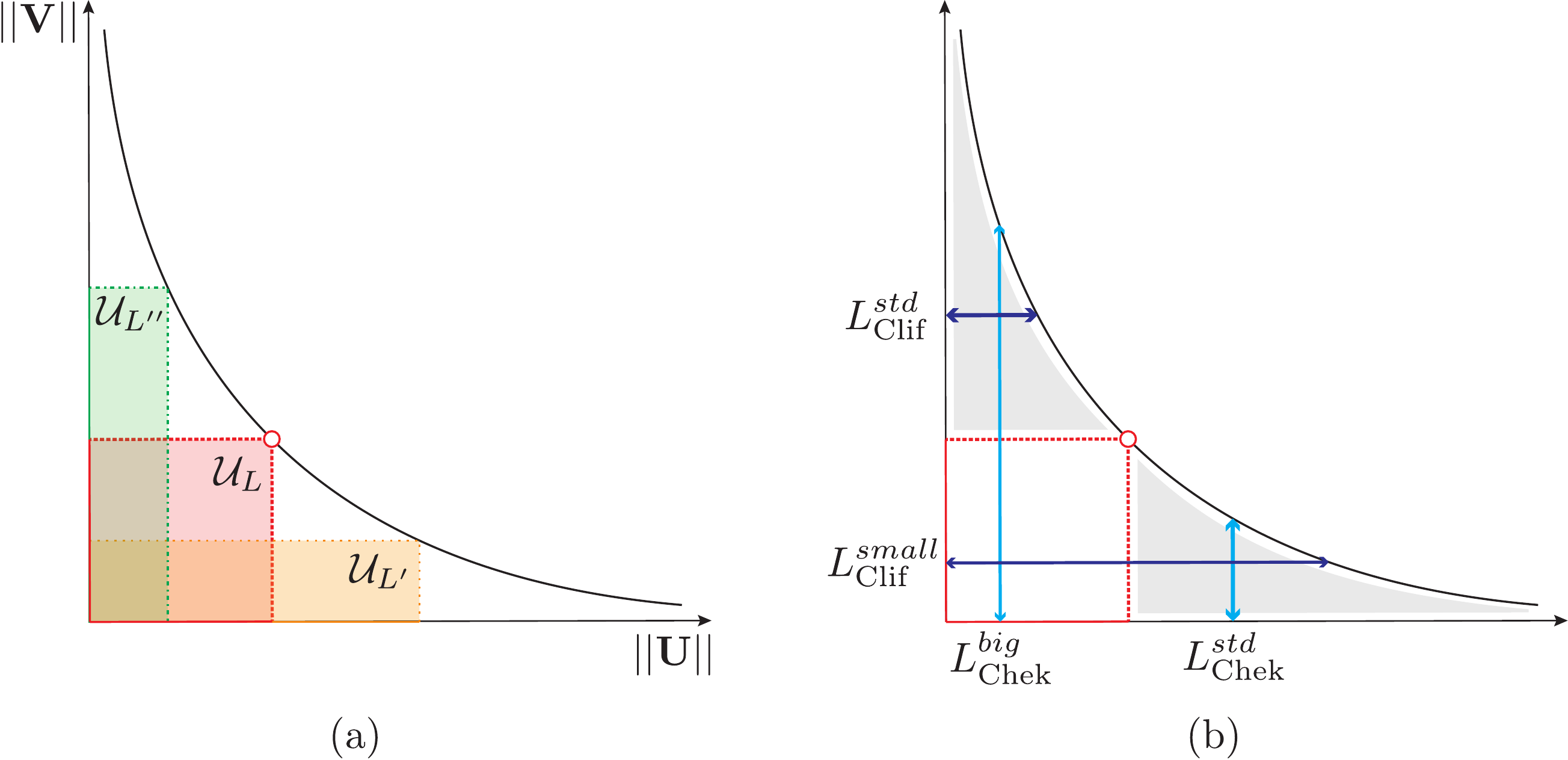}
\caption{Regions in the mirror space occupied by (a) immersed Lagrangian spheres and (b) Lagrangian tori}\label{fig:region_s2}
\end{center}
\end{figure}

One can vary $L$ by choosing non-concentric circle about the origin, while still enclosing the origin and passing through $-\epsilon$. Such a deformation is certainly not a Hamiltonian isotopy, and hence produce non-isomorphic objects in the Fukaya category. We investigate the change of $\mathcal{MC} (L)$ in this case.
We may assume $\mathcal{A} (U) = \mathcal{A} (V)=0$ for the concentric $L$, which can be achieved by adjusting $\epsilon$ if necessary. $\mathcal{U}_L$ is then given as the polydisk with each factor having radius $1$. 

Take an immersed sphere $L'$ obtained by taking a circle smaller than $|w| = \epsilon$, and write $U'$ and $V'$ for the corresponding immersed generator for $L'$. Stokes formula implies that
$\mathcal{A} (U') > 0$ and $\mathcal{A} (V') < 0$. Therefore $\mathcal{MC} (L')$ is embedded in $\Lambda^2 \setminus \{ \mathbf{U} \mathbf{V} = 1\}$ as a polydisk $\mathcal{U}_{L'}$ with radius of $\mathbf{U}$-factor bigger than $e^{-0} =1$, and that of $\mathbf{V}$-factor less than $1$.
On the other hand, an immersed sphere $L''$ with sitting over a circle in $w$-plane with a bigger enclosed area satisfy $\mathcal{A} (U'') < 0$ and $\mathcal{A} (V'') > 0$. In Figure \ref{fig:region_s2} (a), we depict three different charts $\mathcal{U}_{L}$, $\mathcal{U}_{L'}$ and $\mathcal{U}_{L''}$ by plotting the norm $(x,y)=(||\mathbf{U}||, ||\mathbf{V}||)=(e^{-val (\mathbf{U})},e^{-val (\mathbf{V})} )$. One corner of each chart moves along the curve $ab=1$ since the valuations of the two immersed generators add up to zero in any case.

\begin{figure}[h]
\begin{center}
\includegraphics[height=2.5in]{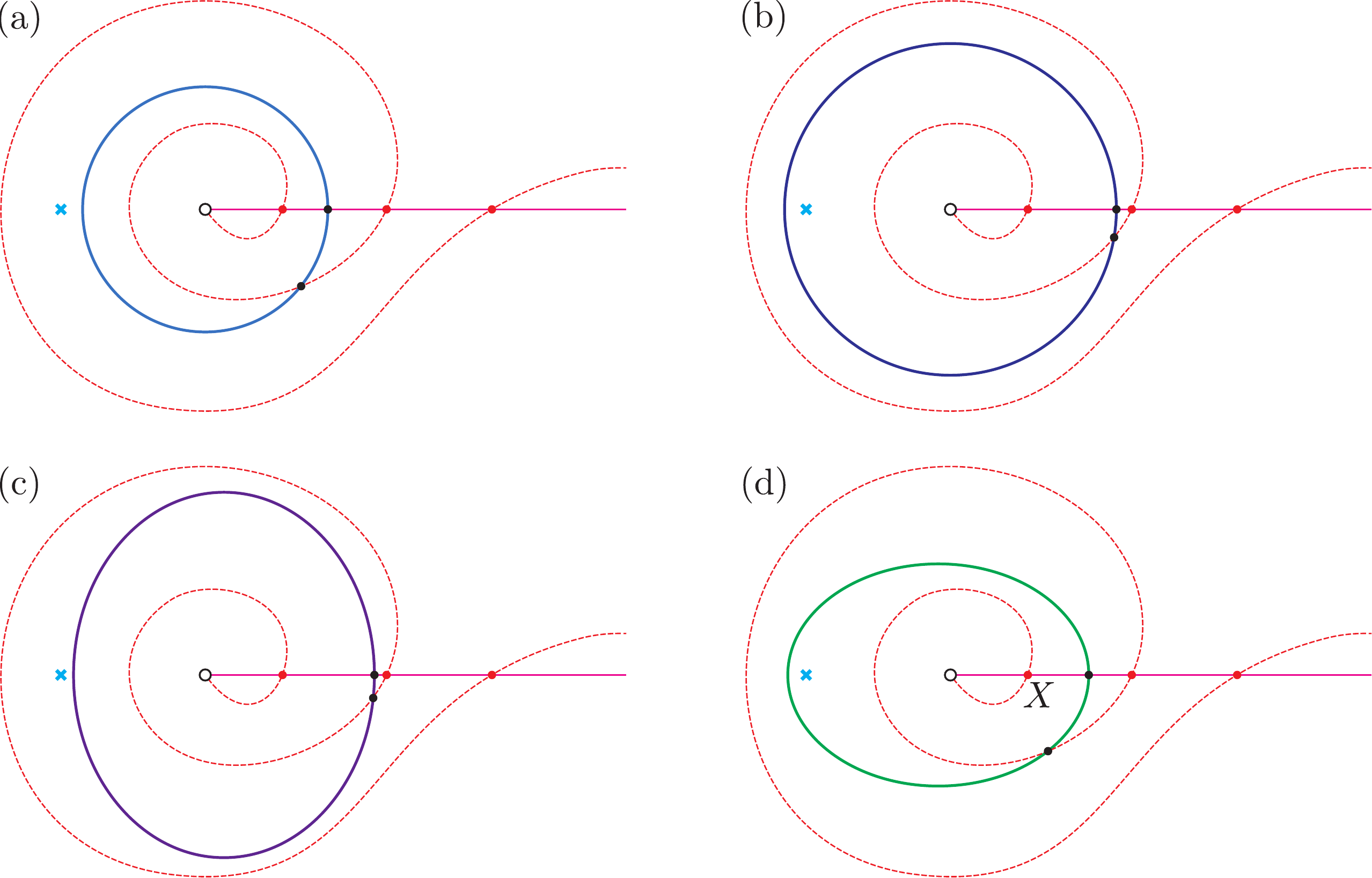}
\caption{Four different types of tori in $M$ distinguished by wall-crossing and valuations: (a) $L_{Chek}^{std}$, (b) $L_{Clif}^{std}$, (c) $L_{Chek}^{big}$ and (d) $L_{Clif}^{small}$}\label{fig:fourtori}
\end{center}
\end{figure}

We next turn our attention to various Lagrangian tori in $M$. We will only consider the tori sitting over circles in $w$-plane with height $|z_1|^2- |z_2|^2=0$.
 These can be divided into four crucially different families of Lagrangian tori, which can be distinguished from each other by their locations relative to wall (or that of $L$) and ``sizes" of their projections relative to that of $L$. The original fibers of the fibration \eqref{eqn:i1syzf} already has Chekanov and Clifford tori, denoted by $L_{Chek}^{std}$ and $L_{Clif}^{std}$ in different sides of the wall. In addition, one can consider a Chekanov torus $L_{Chek}^{big}$ which lies above a circle in $w$-plane bigger than $|w|=\epsilon$, or a Clifford torus $L_{Clif}^{small}$ sitting over a circle smaller than $|w|=\epsilon$.
Notice that the latter two can be Hamiltonian isotopic to none of fibers of \eqref{eqn:i1syzf}. Figure \ref{fig:fourtori} shows typical members of these four different families. 

Let us equip the tori with $\Lambda_{\mathrm{U}}$-connections parametrized by their corresponding holonomies $(z_1,z_2)$. Here, $z_1$ is the holonomy along    a circle in the conic fiber, and $z_2$ is that of a circle in $|w|$-plane. (See \cite{HKL} for more details.)
One can compute $\kappa$ for these tori by directly counting holomorphic sections with appropriate boundary conditions. Alternatively, we may use the fact that $\kappa$ is invariant under the quasi-isomorphism in the Fukaya category, since the construction of $\kappa$ is purely categorical. In either way, one has, for $L_{Chek}$,
\begin{equation*}
\left\{
\begin{array}{l}
z_1 =  \widetilde{C}_{0,1} = \mathbf{U}\mathbf{V} -1 \\
z_2 = \widetilde{C}_{-1,0} = T^{a}\mathbf{U}
\end{array}
\right.
\Rightarrow  
\mathbf{U} = T^{-a} z_2 \in T^{- a}  \Lambda_{\mathrm{U}},
\,\, \mathbf{V} = T^{a} (z_1 +1) \in
\left\{
\begin{array}{ll}
T^{a} \Lambda_+ & z_1 \in -1 + \Lambda_+  \\
T^{a} \Lambda_{\mathrm{U}} & \textnormal{otherwise}
\end{array}
\right.
\end{equation*}
for some $a$ such that $a  < 0$ for $L_{Chek}^{std}$  and $a>0$ for $L_{Chek}^{big}$. Therefore, $e^{-val (\mathbf{U})} = e^a$ and  $0< e^{- val (\mathbf{V} ) }\leq e^{-a}$, which draws vertical lines below the graph of $xy=1$ in Figure \ref{fig:region_s2} (b). The vertical line moves closer to the $y$-axis as the size of the $w$-projection of the Chekanov torus becomes bigger. Likewise, plotting Clifford tori in $(x,y)$-planes give horizontal lines left to $xy=1$ in Figure \ref{fig:region_s2} (b). Lagrangian tori with $|z_1|^2 - |z_2|^2 \neq 0$ are located in the region above the graph of $xy=1$.

\subsection{A pair-of-pants}\label{subsec:popex1}

Let $M$ be a pair-of-pants. $M$ admits a torus fibration with the ``Y"-shaped singular fiber (the Lagrangian skeleton depicted in the left of Figure \ref{fig:pop}), whose Floer theory cannot be defined. Thus, instead of considering the SYZ mirror, we look into the wrapped Floer theory of three noncompact Lagrangian $G_1$, $G_2$ and $G_3$ shown in Figure \ref{fig:pop}. Any two of them already generate the wrapped Fukaya category. The endomorphism algebra of $\oplus G_i$ is built upon the path algebra of a quiver with more than one vertex, and hence is noncommutative. It also has nonzero degree components, so it is hard to extract the conventional mirror out of it. 

We take the following alternative approach. let us first explain the wrapped cohomology for each of $G_i$'s. Wrapping $G_1$ around two punctures it asymptotes, we obtain a similar picture as what we have seen for $T^\ast S^1$, except that the third puncture appears in the middle of the cylinder. 
As the slope $w$ grows, we obtain generators 
$U_0, U_1,U_2,\cdots$ and $V_0, V_1,V_2,\cdots$ (with $U_0=V_0$ being the unit) of $HW(G_1,G_1)$ whose corresponding intersection points (in $G_1 \cap \phi_{wH} (G_1)$)  indicated in Figure \ref{fig:pop}. Exact generators of $HW(G_1,G_1)$ induced from $U_i$ and $V_i$ have degree $0$, and satisfy
$$ m_2(\tilde{U}_i, \tilde{U}_j ) = \tilde{U}_{i+j},\quad  m_2(\tilde{V}_i, \tilde{V}_j ) = \tilde{V}_{i+j}, \quad m_2(\tilde{U}_i, \tilde{V}_j ) =0$$
for $i,j \geq 1$ (see \cite{AAEKO} for more details). Setting $\mathbf{U}:=\tilde{U}_1$ and $\mathbf{V}:=\tilde{V}_1$, we see that the wrapped Floer cohomology of $G_1$ can be identified with
\begin{equation}\label{eqn:hompopg1}
HW (G_1,G_1;\Bbbk) = \dfrac{\Bbbk [\mathbf{U},\mathbf{V}] }{ \langle \mathbf{U} \mathbf{V} \rangle }
\end{equation}
where $\Bbbk$ could be $\C$ or $\Lambda$ depending on which wrapped Fukaya category we are dealing with. Thus the corresponding space is the union of two coordinate axes in $\Bbbk^2$, or $\{ (\mathbf{U},\mathbf{V}) \in \Bbbk^2 \mid \mathbf{U} \mathbf{V} = 0\}$. 

\begin{figure}[h]
\begin{center}
\includegraphics[height=2.1in]{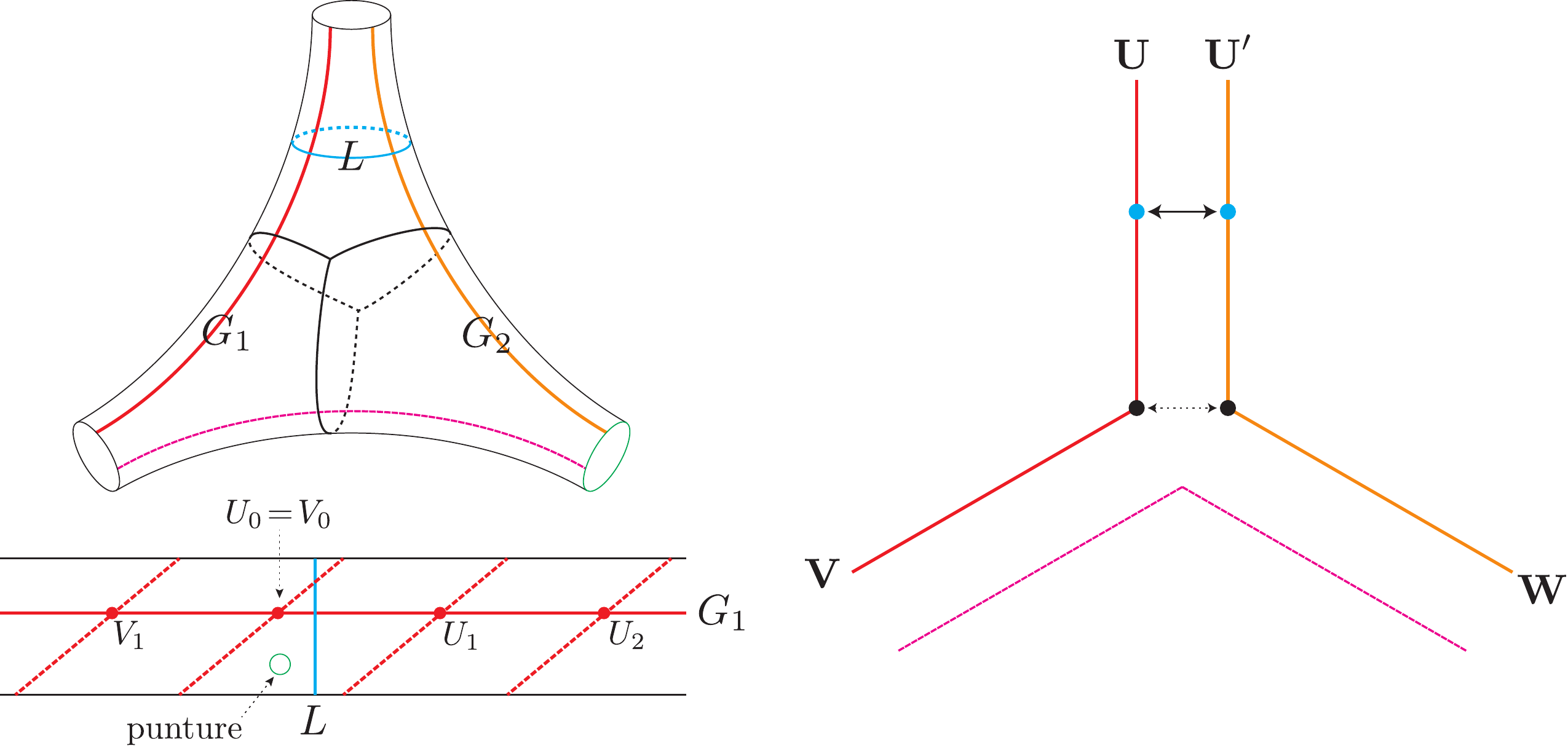}
\caption{Noncompact Lagrangians $G_1$ and $G_2$ and their endomorphism algebras interacting with the Maurer-Cartan algebra of the circle $L$}\label{fig:pop}
\end{center}
\end{figure}

We then consider the circle $L$ in Figure \ref{fig:pop} and its deformation by $x d\theta$. Let us assume for simplicity that $L$ is exact so as not to have any nontrivial powers of $T$'s in the formula below. By direct counting, the kernel $\mathfrak{m}$ of the resulting augmentation  on \eqref{eqn:hompopg1} is the ideal generated by $\mathbf{U}-1$ and $\mathbf{V}$ (or the point corresponding to $L$ is $(\mathbf{U},\mathbf{V}) = (1,0)$). In particular, $\mathbf{V}$ does not survive in the $\mathfrak{m}$-adic completion of \ref{eqn:hompopg1} with respect to this ideal, since $\mathbf{V} = (-1)^n (\mathbf{U}-1)^n \mathbf{V}$ for any $n$. 

Given the locations of $L$ and $G_1$ in $M$, one can show by the same calculation as in Subsection \ref{subsec:ts1exfirst} that $\kappa_{G_1,L} ( \mathbf{U}) = e^x$. Also, we have $\kappa_{G_1,L} ( \mathbf{V}) = 0$ since there are no disks to count. 
This does not violate Theorem \ref{thm:mainkd} (note that $G_1$ generates $L$) 
as $\mathbf{V}$ becomes zero in the completion.
To summarize the Maurer-Cartan space of $L$ lies in \eqref{eqn:hompopg1} via 
$$\kappa_{G_1,L} ( \mathbf{U}) = e^x  \quad \textnormal{and} \quad \kappa_{G_1,L} ( \mathbf{V}) = 0.$$ 
Namely, $A_L = \Lambda\{x\}$ is isomorphic to the subspace of the $\mathfrak{m}$-adic completion of \eqref{eqn:hompopg1} consisting of series with bounded coefficients. $\mathcal{MC} (L) \cong \Lambda_+$ sits in the $\mathbf{U}$-axis ($\subset \{ (\mathbf{U},\mathbf{V}) \in \Bbbk^2 \mid \mathbf{U} \mathbf{V} = 0\}$) as the unit open disk around $(\mathbf{U},\mathbf{V}) = (1,0)$.

On the other hand, exactly the same procedure applying to the pair $(G_2,L)$ identifies the Maurer-Cartan space of $L$ as a subset of the coordinate axis $\mathbf{W}=0$ via 
$$\kappa_{G_2,L} (\mathbf{U}') = e^x \quad \textnormal{and} \quad \kappa_{G_2,L} (\mathbf{W}) = 0,$$ 
where we take the presentation $ \Bbbk [\mathbf{U}',\mathbf{W}] / \langle \mathbf{U}' \mathbf{W} \rangle$ of the wrapped Floer cohomology of $G_2$ as before. (Here, $\mathbf{U}'$ is the exact generator appearing closer to the same puncture as $\mathbf{U}$ for $G_1$.)
Repeating the same computation for other circles in the same leg of $M$, we conclude that in view of Lagrangian moduli, the two coordinate axes $\{ (\mathbf{U},0) \mid \mathbf{U} \in \Bbbk^\ast \}$ (from $G_1$) and $\{ (\mathbf{U}',0) \mid \mathbf{U} \in \Bbbk^\ast \}$ (from $G_2$) should be identified by letting $(\mathbf{U},0) = (\mathbf{U}',0)$, as indicated in Figure \ref{fig:pop}. 

\begin{figure}[h]
\begin{center}
\includegraphics[height=1.9in]{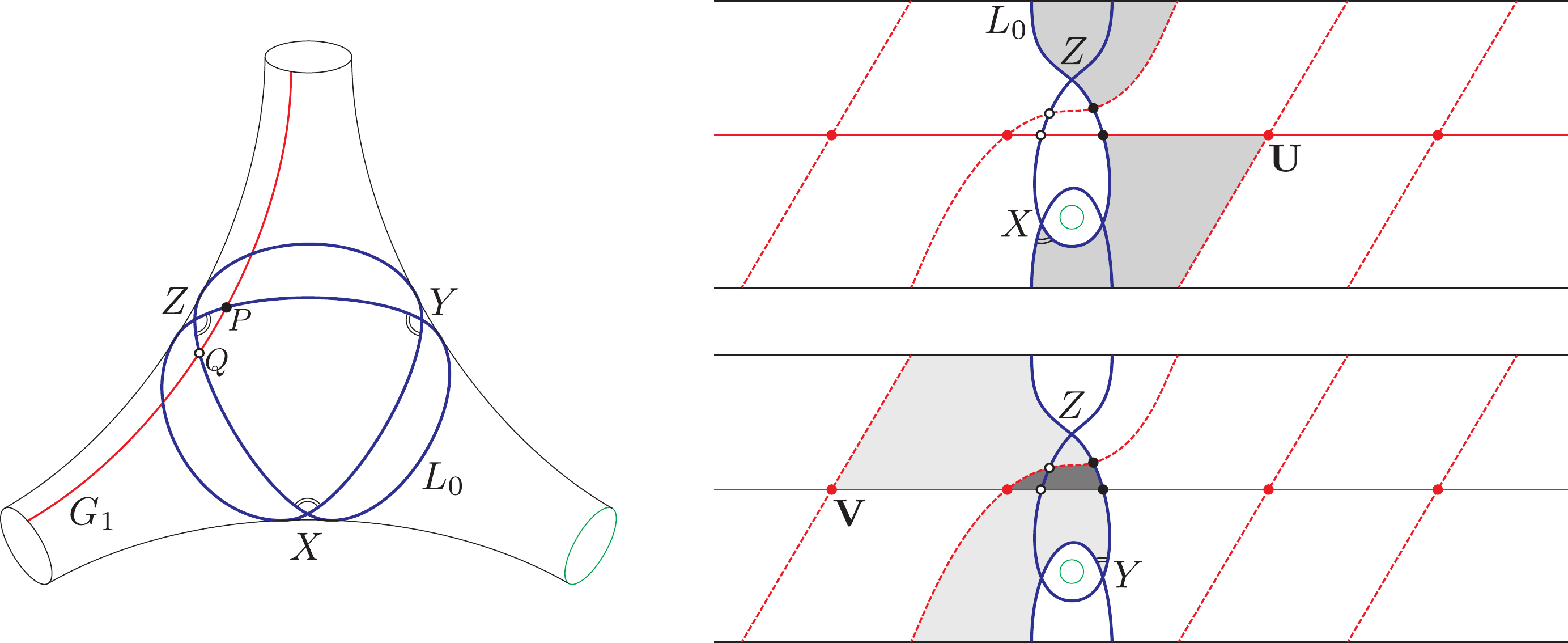}
\caption{The Maurer-Cartan algebra of the Seidel Lagrangian and the endomorphism algebra of $G_1$}\label{fig:seilag}
\end{center}
\end{figure}

Observe that the origin $\mathbf{U}= \mathbf{V}=0$ cannot be covered by any of such circles $L$, since the exponential $e^x$ is always nonzero. We next seek for the object in the Fukaya category that should correspond to the origin. Let us consider the immersed circle $L_0$ with three self-intersection points $X,Y$ and $Z$ described in Figure \ref{fig:seilag}. This immersed Lagrangian has been used to study homological mirror symmetry of surfaces in many literatures, since first introduced by Seidel \cite{Seidel-g2}. 

Let $X,Y,Z$ and $X',Y',Z'$ denote degree $0$ and $1$ generators, respectively.
Observe that for $\omega = d \Theta$, 
\begin{equation}\label{eqn:frontbackcancel}
\int_{L_0} \Theta = \int_{ \partial \Delta_F} \Theta - \int_{\partial \Delta_B} \Theta = \int_{\Delta_F} \omega - \int_{\Delta_B} \omega
\end{equation}
where $\Delta_F$ and $\Delta_B$ are two triangles bounded by $L_0$ with corners $X,Y$ and $Z$ and their boundaries are positively oriented.
By imposing the reflection symmetry on $L_0$ so that \eqref{eqn:frontbackcancel} vanishes,  $L_0$ can be made exact, or $\iota^\ast \Theta = df$ for some function $f$ on the domain $S^1$ of the immersion $\iota : S^1 \looparrowright L_0$. Thus each immersed generator naturally comes with its action $\mathcal{A}$, and one can assign exact generators to $X,Y,Z,X',Y',Z'$ by multiplying $T^{\mathcal{A} ( -)}$ (see \eqref{eqn:actionforimm}). The Maurer-Cartan equation for $L_0$ in terms of these exact generators is given as
\begin{equation}\label{eqn:mcforpop}
m_0(e^b) = (\tilde{y}\tilde{z}-\tilde{z} \tilde{y}) \widetilde{X}' +  (\tilde{z}\tilde{x}-\tilde{x} \tilde{z}) \widetilde{Y}'+  (\tilde{x}\tilde{y}-\tilde{y} \tilde{x}) \widetilde{Z}'  + \tilde{x}\tilde{y}\tilde{z} \one_{L_0} 
\end{equation}
where $b= \tilde{x} \WT{X} + \tilde{y} \WT{Y} + \tilde{z} \WT{Z}(=xX + yY+zZ)$. (For instance, $\WT{X} = T^{\mathcal{A} (X)} X$ and $\tilde{x} = T^{-\mathcal{A} (X)} x$.) Notice that there is no nontrivial power of $T$ appearing in \eqref{eqn:mcforpop}, thanks to the usage of exact generators. 

\begin{remark}\label{rmk:lgmirrorpop}
Strictly speaking, $L_0$ does not completely fit into our setting in Section \ref{sec:koszulpairmain} since it is only $\Z_2$-graded, and a lower (de Rham) degree term such as $\tilde{x}\tilde{y} \tilde{z} \one_{L_0}$ appears in the Maurer-Cartan equation \eqref{eqn:mcforpop}. To keep the $\Z$-graded setting, one could introduce an extra variable ``$e$" with $\deg e=2$ encoding Maslov indices, which turns the last term  in \eqref{eqn:mcforpop} to $\tilde{x}\tilde{y}\tilde{z} e \one_{L_0}$.

On the other hand, the coefficient $W:=\tilde{x}\tilde{y} \tilde{z} = T^{-\mathcal{A}(X) -\mathcal{A}(Y) -\mathcal{A} (Z) } xyz$  is called the potential function defined on the weak Maurer-Cartan space of $L_0$, which give a local Landau-Ginzburg mirror. To deal with Lagrangians more systematically, one may need a generalization of $\kappa$ into weak Maurer-Cartan deformations, which we leave for future research. 
\end{remark}

Therefore the relations are generated by commutators among variables together with $xyz$, and hence the (strict) Maurer-Cartan algebra in $\Lambda$-coefficient can be represented as
$$ A_{L_0} = \dfrac{\Lambda\{\!\{ x,y,z \}\!\} }{ \langle \langle [x,y], [y,z], [z,x], xyz \rangle \rangle }.$$
The associated space $\mathcal{MC} (L_0)$ is given by $\{xyz=0\} \subset (\Lambda_+)^3$.
One can also use $\C$-coefficient, in which case the corresponding Maurer-Cartan deformation of $L_0$ gives a formal neighborhood the origin in the union of three coordinate planes in $\C^3$.


Another feature of $L_0$ slightly off the setting in \ref{sec:koszulpairmain} is that
$G_i$ and $L_0$ intersect at two points, say $P$ of degree 0 and $Q$ of degree $1$. 
Despite the existence of an additional intersection point $Q$, the map $\kappa_{G_i,L_0}$ defined by the precisely same formula \eqref{eqn:formulakoszul11} still produces an algebra homomorphism modulo the ideal $\langle r_{PQ} (x,y,z), r_{QP} (x,y,z) \rangle$ where $r_{PQ}$ and $r_{QP}$ are given by  $m_1^{0,b} (Q) = r_{QP} (x,y,z) P$ nad $m_1^{0,b} (P) = r_{PQ} (x,y,z) Q$. To see this, observe that the $A_\infty$-relation
\begin{equation*}
\begin{array}{l}
m ( \overbrace{m(Z_1,Z_2,P, e^b)}^{\textnormal{multiple of}\,\,Q},e^b) + m(m_2(Z_1,Z_2),P,e^b) \\
- m (Z_1, m(Z_2,P,e^b),e^b) + m(Z_1,Z_2, m_1^{0,b} (P), e^b) =0
\end{array}
\end{equation*}
for $Z_1,Z_2 \in \hom_{\mathcal{W}^\Lambda} (G_i,G_i)$ reduces to
$$\kappa(m_2(Z_1,Z_2)) = \kappa(Z_1) \kappa(Z_2) \mod  \langle r_{PQ}, r_{QP} \rangle$$
since the first and the last terms are multiple of $r_{QP}$ and $r_{PQ}$, respectively.
It is easy to see that $r_{PQ} (x,y,z)=z$ and $r_{PQ} (x,y,z) = xy$ up to scaling in this case. 

\begin{remark}
The pair $(r_{PQ} , r_{QP})$ gives the matrix factorization of $W$ mirror dual to $G_1$ in view of the Landau-Ginzburg mirror of $M$ mentioned in Remark \ref{rmk:lgmirrorpop}. See \cite{CHL_gl} for more details.
\end{remark}

As a result, one obtains an algebra homomorphism
\begin{equation}\label{eqn:kappapop1}
 \kappa_{G_1,L_0}: \dfrac{\Lambda [\mathbf{U},\mathbf{V}] }{\langle \mathbf{U} \mathbf{V}  \rangle} \to \dfrac{A_{L_0} }{ \langle \langle r_{PQ}, r_{QP} \rangle \rangle } \cong \dfrac{\Lambda\{\!\{ x,y\}\!\}}{ \langle \langle xy \rangle \rangle}.
\end{equation}
The homomorphism is determined by images of $\mathbf{U}$ and $\mathbf{V}$, and the shaded holomorphic disks drawn in Figure \ref{fig:seilag} shows
$$ \kappa_{G_1,L_0} (\mathbf{U}) = \tilde{x} = T^{ -\mathcal{A}(X)} x,  \quad \kappa_{G_1,L_0} (\mathbf{V}) = \tilde{y} = T^{ -\mathcal{A} (Y)} y.$$
In particular, $\kappa_{G_1,L_0}$ becomes an isomorphism after taking completion of the left hand side.
Setting $x=y=0$, \eqref{eqn:kappapop1} reduces to the augmentation $\mathbf{U},\mathbf{V} \mapsto 0$, from which one can deduce that $L_0$ (without boundary deformation) sits at the origin $(\mathbf{U},\mathbf{V})= (0,0)$. 

By the same argument, we see also that $L_0$ corresponds to the origin $(\mathbf{U'},\mathbf{W})=(0,0)$ in view of $G_2$. Accordingly, one has to identify the two ``origins". Combined with the previous identification $(\mathbf{U},0) = (\mathbf{U}',0)$ for $\mathbf{U}, \mathbf{U}' \neq 0$, we conclude that $\mathbf{U}$-axis in $\{\mathbf{U} \mathbf{V}=0\}$ from $G_1$ and $\mathbf{U}'$-axis in $\{\mathbf{U}'\mathbf{W}=0\}$ from $G_2$ should be glued together. 

The upshot is the union of three coordinate axes 
\begin{equation}\label{eqn:globalpop1}
\{(\mathbf{U},0,0) \} \cup \{(0,\mathbf{V},0) \}  \cup \{(0,0,\mathbf{W}) \}  
\end{equation}
in $\Lambda^3$, and this is obtained by keeping track of locations of point-like objects (in the Fukaya category) in the ``$\mathrm{Spec}$" of $HW(G_1,G_1)$ and $HW(G_2,G_2)$. (We could also include $HW(G_3,G_3)$, but it is redundant here.)

Putting together \eqref{eqn:kappapop1} and its analogues for $G_2,G_3$, we see that the subspace $\{(x,y,z)\in \Lambda_+^3 \mid xy=yz=zx=0\}$ of the Maurer-Cartan space $\mathcal{MC} (L_0)$ embeds into \eqref{eqn:globalpop1} via 
$$(\mathbf{U},\mathbf{V},\mathbf{W}) = (T^{-\mathcal{A}(X)} x, 0,0) \quad \textnormal{or} \quad\  (0, T^{-\mathcal{A} (Y)} y, 0) \quad \textnormal{or}\quad  (0,0, T^{-\mathcal{A} (Z)}z),$$
which describes a certain neighborhood $\mathcal{V}_{L_0}$ of the origin in \eqref{eqn:globalpop1}.
Note that $\mathcal{V}_{L_0}$ is the critical loci of the potential 
$$W:\mathcal{U}_{L_0} \to \Lambda \quad (x,y,z) \mapsto T^{-\mathcal{A}(X)-\mathcal{A} (Y) -\mathcal{A} (Z)} xyz,$$ 
and hence, it can be thought of as the set of nonzero objects $(L_0,b)$ for $b \in \mathcal{MC} (L_0)$.

\subsection{Smoothing of conifold} 
The last example is the (divisor complement of) deformed conifold discussed in Example \ref{ex:conimc}
$$M:= \{ (u_1,v_1,u_2,v_2,z) \in \C^4 \times \C^\ast \mid u_1 v_1 =z-a, u_2 v_2 =z-b\} \setminus \{ z=0\}.$$
The projection to $z$-plane defines a double conic fibration on $M$, whose fibers degenerate over $z=a$ and $z=b$. 
Near $a$ or $b$, $M$ is locally isomorphic to the product of the example in \ref{subsec:i1ex} with $T^\ast S^1$, and as before, it admits a torus fibration with fibers lying over concentric circles about the origin in $z$-plane, consisting of $T^2$-orbits for the Hamiltonian action $(u_1,v_1,u_2,v_2,z) \mapsto (e^{i \theta_1} u_1, e^{-i\theta_1} v_1, e^{i\theta_2} u_2, e^{-i\theta_2} v_2,z)$.

Take two Lagrangians $L_0$ and $L_1$ to be matching spheres lying over paths drawn in Figure \ref{fig:conifoldbase}. As in the picture, $G_i$ for $i=0,1$ is a Lagrangian section of the torus fibration which intersects $L_i$ exactly once. More details on Floer theory of these Lagrangians can be found in \cite{CPU}. We finally set $G= G_0 \oplus G_1$ and $L= L_0 \oplus L_1$.

\begin{figure}[h]
\begin{center}
\includegraphics[height=1.2in]{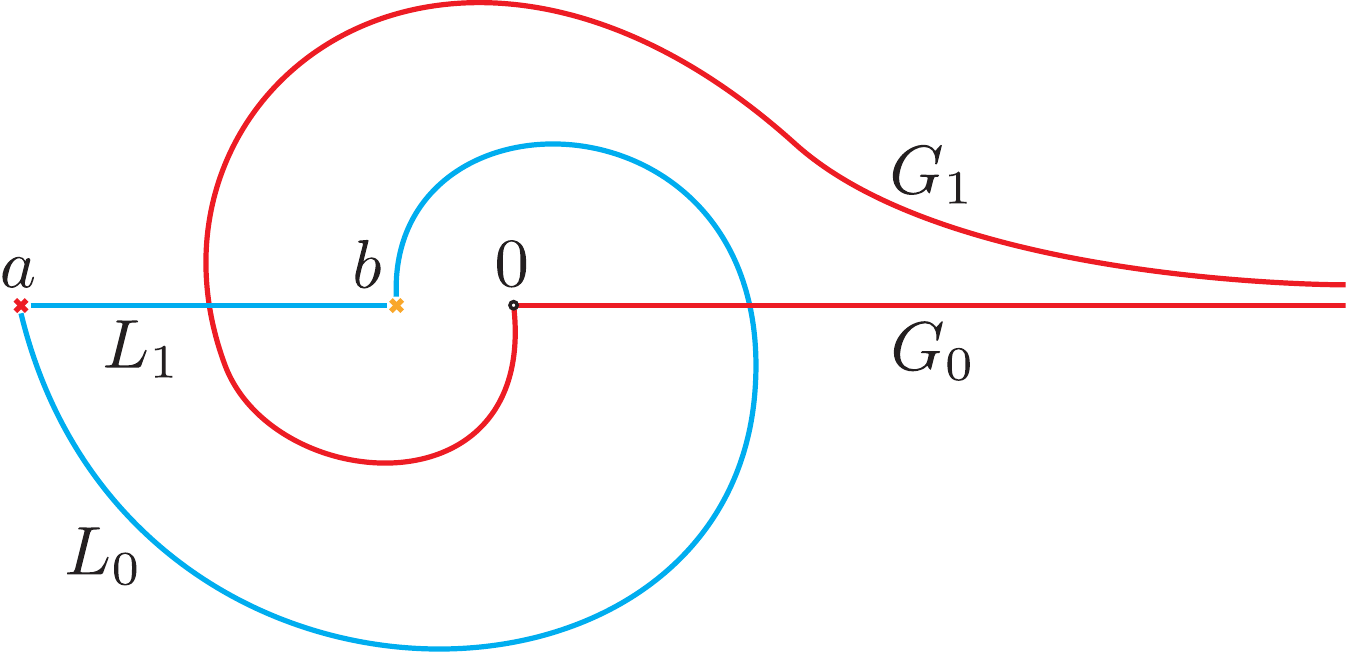}
\caption{two Lagrangian spheres $L_0$, $L_1$ and their duals}\label{fig:conifoldbase}
\end{center}
\end{figure}

The computation of $\kappa$ for $G$ and $L$ can be done in a similar manner to what we did in \ref{subsec:i1ex}, and we do not present much details. For a linear Hamiltonian $wH$ with $w \to \infty$, we obtain the family $\{P_{a,b,c}^i :a,b,c \in \Z\}$ of $wH$-hamiltonian chords from $G_i$ to itself for $i=0,1$ (equivalently, the intersection points in $G \cap \phi_H(G)$). Likewise, we have $\{Q_{a,b,c} : a \in \frac{1}{2} +\Z, b,c \in \Z \}$ spanning $\hom (G_1,G_0)$, and $\{R_{a,b,c} : a \in \frac{1}{2} +\Z, b,c \in \Z \}$ spanning $\hom (G_0,G_1)$. Thier degrees are all $0$. Analogous to the way of indexing in \ref{subsec:i1ex}, $a$ and $(b,c)$ in $P_{a,b,c}^i$ indicate the component of chords along the base and the double-conic fiber directions, respectively. $Q_{a,b,c}$ and $R_{a,b,c}$ are similarly defined, except that they lie between the two different Lagrangians. 
The locations of the intersection points corresponding to the first few generators are shown in Figure \ref{fig:kappaq10}. We denote by $\WT{P}^i_{a,b,c}$ (for $i=0,1$), $\WT{Q}_{a,b,c}$ and $\WT{R}_{a,b,c}$ the associated generators of $ H^\ast \hom(G,G)$. $\WT{P}^i_{0,0,0}$ is the unit in the corresponding component of $ H^\ast \hom(G,G)$.

In \cite{CPU}, $m_2$ between these generators has been explicitly computed. For instance,
$$ m_2 (\WT{P}^0_{a_1,b_1,c_1}, \WT{P}^0_{a_2,b_2,c_2})= \sum_{i=0}^{k_1} \sum_{j=0}^{k_2} {k_1 \choose i} {k_2 \choose j}  \WT{P}^0_{a_1+b_1,a_2+b_2+i,a_3+b_3+j}. $$
where $k= \min\{|a_1|,|a_2|\}$ if $a_1$ and $a_2$ have different signs, and $k=0$ otherwise. Also,
$$ m_2 (\WT{Q}_{a_1,b_1,c_1}, \WT{R}_{a_2,b_2,c_2})=\sum_{i=0}^{k_1} \sum_{j=0}^{k_2} {k_1 \choose i} {k_2 \choose j} \WT{P}^1_{a_1+b_1,a_2+b_2+i,a_3+b_3+j}. $$
where $k_1 = \min\{|a_1|-1/2,|a_2|-1/2\} +1$ and $k_2 = \{|a_1|-1/2,|a_2|-1/2\}$ if $a_2<0<a_1$, and $k_1 = \min\{|a_1|-1/2,|a_2|-1/2\} $ and $k_2 = \{|a_1|-1/2,|a_2|-1/2\} +1 $ if $a_1<0<a_2$ ($k_1=k_2=0$ if $a_1$ and $a_2$ have the same sign). 
Formulas for other products have similar patterns, and we omit. 

Recall from Example \ref{ex:conimc} that the Maurer-Cartan algebra $A_L$ is a quiver algebra generated by four arrows $x,y,z,w$, with relations $xyz = zyx$ and its permutations. These arrows are dual to immersed generators $X,Y,Z,W$ taken as in Figure \ref{fig:kappaq10}. ($Y$ and $W$ are complementary to $X$ and $Z$, respectively.)
\begin{figure}[h]
\begin{center}
\includegraphics[height=1.7in]{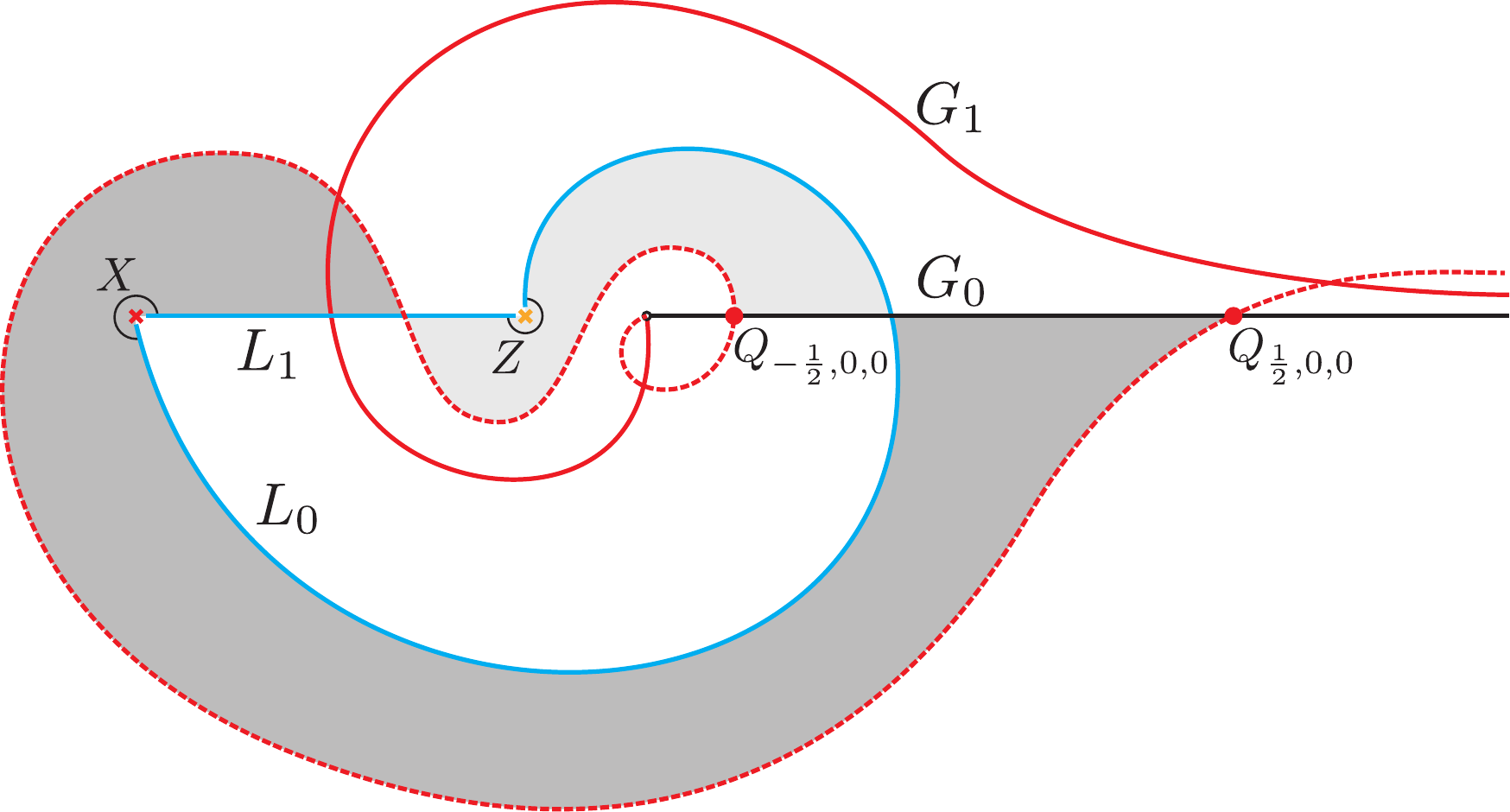}
\caption{The Koszul map $\kappa$ for generators of $\hom_{\mathcal{W}} (G,G)$}\label{fig:kappaq10}
\end{center}
\end{figure}

We can compute $\kappa : H^0 (\hom (G,G)) \to A_L$ using the same strategy as in \ref{subsec:i1ex}. Namely, we first find the images of generators of $H^0 (\hom (G,G))$ with small indices by directly counting holomorphic disks, and then apply the known formula of $m_2$ above. 

First, it is not difficult to check that $\WT{Q}_{\pm \frac{1}{2},0,0}$, $\WT{R}_{\pm \frac{1}{2},0,0}$ and $\WT{P}^i_{\pm 1,0,0}$ (for $i=0,1$) generate the algebra $H^0 (\hom(G,G))$, and hence it suffices to find their images under $\kappa$. Any contributing disks to $\kappa (\WT{Q}_{\pm \frac{1}{2},0,0})$ project to one of the shaded regions in Figure \ref{fig:kappaq10}) by the maximum principle, and each shaded region supports a unique disk since the Lagrangian boundaries lie on the same moment fiber and the fibration is trivial away from an arbitrarily small neighborhood of the double conic fiber containing $X$ or $Z$. Therefore $\kappa (\WT{Q}_{\frac{1}{2},0,0}) = \tilde{x} $ and $\kappa (\WT{Q}_{\frac{1}{2},0,0}) = \tilde{z}$. The other generators can be handled in the same manner.

The general formula for $\kappa$ is given as follows: for $0<a \in \Z$, 
$$ \kappa( \WT{P}^0_{a,b,c} ) = (\tilde{w}\tilde{x})^{a} (\tilde{y}\tilde{x}-1)^b (\tilde{w}\tilde{z}-1)^c,\quad  \kappa( \WT{P}^0_{-a,b,c} ) = (\tilde{y}\tilde{z})^{a} (\tilde{y}\tilde{x}-1)^b (\tilde{w}\tilde{z}-1)^c,$$
$$ \kappa( \WT{P}^1_{a,b,c} ) = (\tilde{x}\tilde{w})^{a} (\tilde{x}\tilde{y}-1)^b (\tilde{z}\tilde{w}-1)^c,\quad  \kappa( \WT{P}^1_{-a,b,c} ) = (\tilde{z}\tilde{y})^{a} (\tilde{x}\tilde{y}-1)^b (\tilde{z}\tilde{w}-1)^c,$$
and for $0< a \in \frac{1}{2} + \Z$,
$$ \kappa( \WT{Q}_{a,b,c} ) = \tilde{x} (\tilde{w}\tilde{x})^{a - \frac{1}{2} }  (\tilde{y}\tilde{x}-1)^b (\tilde{w}\tilde{z}-1)^c, \quad  \kappa( \WT{Q}_{-a,b,c} ) =\tilde{z} (\tilde{y}\tilde{z})^{-a +\frac{1}{2} }  (\tilde{y}\tilde{x}-1)^b (\tilde{w}\tilde{z}-1)^c,$$
$$ \kappa( \WT{R}_{a,b,c} ) = \tilde{w} (\tilde{x}\tilde{w})^{a- \frac{1}{2}} (\tilde{x}\tilde{y}-1)^b (\tilde{z}\tilde{w}-1)^c, \quad \kappa( \WT{R}_{-a,b,c} ) = \tilde{y} (\tilde{z}\tilde{y})^{-a+\frac{1}{2}} (\tilde{x}\tilde{y}-1)^b (\tilde{z}\tilde{w}-1)^c.$$
Negative powers of a polynomial in the formula should be interpreted as series, for e.g., $(\tilde{x} \tilde{y}-1)^{-1} = - 1 - \tilde{x} \tilde{y} - \tilde{x} \tilde{y} \tilde{x} \tilde{y} - \cdots$.
Note that such infinite sums are legitimate in $A_L$. Also, the loops based at the same vertex commute in $A_L$ due to the relations, and hence the order of factors can be switched with some minor effect. For instance,
$$ \tilde{w} (\tilde{x}\tilde{w})^{a- \frac{1}{2}} (\tilde{x}\tilde{y}-1)^b (\tilde{z}\tilde{w}-1)^c = (\tilde{w}\tilde{x})^{a- \frac{1}{2}}  \tilde{w}   (\tilde{z}\tilde{w}-1)^c   (\tilde{x}\tilde{y}-1)^b$$
holds in $A_L$, and the orders of factors in the above formulas are arbitrarily chosen for convenience.
 
\bibliographystyle{amsalpha}
\bibliography{geometry}
\end{document}